\documentclass[11pt, reqno]{amsart}

\usepackage[T1]{fontenc}
\usepackage{babel}
\usepackage{lmodern}
\usepackage{pgf,tikz}\usepackage{mathrsfs}\usetikzlibrary{arrows}
\usetikzlibrary{decorations.pathreplacing, positioning} %calligraphy for big curly bracket
\usepackage{latexsym}
\usepackage{amssymb}
\usepackage{mathrsfs}
\usepackage[leqno]{amsmath}
\usepackage{fancybox,color}
\usepackage{enumerate}
\usepackage[latin1]{inputenc}
\usepackage[active]{srcltx}
\usepackage[colorlinks]{hyperref}
\usepackage{hyperref}
\usepackage{eurosym}
\usepackage{url}
\usepackage{color}
\usepackage{amsthm}

\usepackage{mathtools}
\mathtoolsset{showonlyrefs}

\usepackage{fancyhdr}
\usepackage{xcolor}

\usepackage[foot]{amsaddr}

\newcommand{\kom}[1]{}
\renewcommand{\kom}[1]{{\bf [#1]}}

\def\1{\raisebox{2pt}{\rm{$\chi$}}}

\newtheorem{theorem}{Theorem}[section]
\newtheorem{corollary}[theorem]{Corollary}
\newtheorem{lemma}[theorem]{Lemma}
\newtheorem{proposition}[theorem]{Proposition}
\newtheorem*{thm:gbackharnack}{Theorem \ref{thm:gbackharnack}}
\newtheorem{definition}[theorem]{Definition}
\newtheorem{remark}[theorem]{Remark}

\newcommand{\R}{{\mathbb R}}

\newcommand{\N}{{\mathbb N}}

\newcommand{\eps}{{\varepsilon}}
\def\1{\raisebox{2pt}{\rm{$\chi$}}}

\let\originalleft\left
\let\originalright\right
\renewcommand{\left}{\mathopen{}\mathclose\bgroup\originalleft}
\renewcommand{\right}{\aftergroup\egroup\originalright}

\newcommand{\abs}[1]{\left|#1\right|}
\newcommand{\norm}[1]{\left|\left|#1\right|\right|}
\newcommand{\Rn}{\mathbb{R}^n}

\newcommand{\osc}{\operatorname{osc}}

\newcommand{\aveint}[2]{\mathchoice%
	{\mathop{\kern 0.2em\vrule width 0.6em height 0.69678ex depth -0.58065ex
			\kern -0.8em \intop}\nolimits_{\kern -0.45em#1}^{#2}}%
	{\mathop{\kern 0.1em\vrule width 0.5em height 0.69678ex depth -0.60387ex
			\kern -0.6em \intop}\nolimits_{#1}^{#2}}%
	{\mathop{\kern 0.1em\vrule width 0.5em height 0.69678ex depth -0.60387ex
			\kern -0.6em \intop}\nolimits_{#1}^{#2}}%
	{\mathop{\kern 0.1em\vrule width 0.5em height 0.69678ex depth -0.60387ex
			\kern -0.6em \intop}\nolimits_{#1}^{#2}}}
\newcommand{\ud}{\, d}
\newcommand{\half}{{\frac{1}{2}}}

\newcommand{\Om}{\Omega}

\newcommand{\om}{\omega}

\newcommand{\vp}{\varphi}

\newcommand{\sgn}{\operatorname{sgn}}

\renewcommand{\div}{\operatorname{div}}

\newtheoremstyle{case}{3mm}{-1,5mm}{}{}{}{:}{ }{}
\theoremstyle{case}
\newtheorem{case}{\bf{Case}}

\newtheoremstyle{step}{3mm}{-1,5mm}{}{}{}{:}{ }{}
\theoremstyle{step}

\newcommand{\numberthis}{\addtocounter{equation}{1}\tag{\theequation}}

\makeatletter
\newcommand{\leqnomode}{\tagsleft@true}
\newcommand{\reqnomode}{\tagsleft@false}
\makeatother

\DeclareMathOperator{\Tr}{Tr}
\DeclareMathOperator{\supp}{supp}

\numberwithin{equation}{section}

\usepackage{etoolbox}

\makeatletter
\patchcmd{\@maketitle}
{\if@titlepage \newpage \else}
{\if@titlepage
	\vspace{\baselineskip}
	\else}
{}{}
\makeatother

\let\oldtocsection=\tocsection

\let\oldtocsubsection=\tocsubsection

\let\oldtocsubsubsection=\tocsubsubsection

\renewcommand{\tocsection}[2]{\hspace{0em}\oldtocsection{#1}{#2}}
\renewcommand{\tocsubsection}[2]{\hspace{2em}\oldtocsubsection{#1}{#2}}
\renewcommand{\tocsubsubsection}[2]{\hspace{4em}\oldtocsubsubsection{#1}{#2}}

\calclayout

\newcommand{\overbar}[1]{\mkern 1.5mu\overline{\mkern-1.5mu#1\mkern-1.5mu}\mkern 1.5mu}

\title{Intrinsic Harnack's inequality for a general nonlinear parabolic equation in non-divergence form}

\author{Tapio Kurkinen}
\email{tapio.kurkinen@oist.jp (corresponding author)}

\author{Jarkko Siltakoski}
\email{jarkko.siltakoski@helsinki.fi}
\address{Department of Mathematics and Statistics
	University of Jyv{\"{a}}skyl{\"{a}}
	PO Box 35, FI-40014 Jyv{\"{a}}skyl{\"{a}}, Finland}

\date{\today}


\begin{document}

	\begin{abstract}
		We prove the intrinsic Harnack's inequality for a general form of a parabolic equation that generalizes both the standard parabolic $p$-Laplace equation and the normalized version arising from stochastic game theory. We prove each result for the optimal range of exponents and ensure that we get stable constants.
	\end{abstract}
	\maketitle
	\section{Introduction}
	We prove the intrinsic Harnack's inequality for the following general non-divergence form version of the nonlinear parabolic equation
	\begin{equation}
	\label{eq:rgnppar}
	\partial_t u=\abs{\nabla u}^{q-p}\div\left(\abs{\nabla u}^{p-2}\nabla u\right)=\abs{\nabla u}^{q-2}(\Delta u + (p-2)\Delta_\infty^Nu),
	\end{equation}
	for the optimal range of exponents. The theorem states that a non-negative viscosity solution satisfies the following local a priori estimate
	\begin{equation}
	\label{eq:intro}
	\gamma^{-1}\sup_{B_r(x_0)}u(\cdot,t_0-\theta r^{q})\leq u(x_0,t_0)\leq\gamma\inf_{B_{r}(x_0)}u(\cdot,t_0+\theta r^{q})
	\end{equation}
	for a scaling constant $\theta$ which depends on the value of $u$. This intrinsic waiting time is the origin of the name and is required apart from the singular range of exponents where the elliptic Harnack's inequality holds \cite{Kurkinen2022}. We also establish stable constants at the vicinity of $q=2$.
	
	\sloppy
	When $q=p$, the equation \eqref{eq:rgnppar} is the standard $p$-parabolic equation for which the intrinsic Harnack's inequality was proven by DiBenedetto \cite{Dibenedetto1988} and Kwong \cite{Dibenedetto1992}, see also \cite{Dibenedetto1993}. These results were generalized for equations with growth of order $p$ by DiBenedetto, Gianazza, and Vespri \cite{Dibenedetto2008} and by Kuusi \cite{Kuusi2008}. When $q\not=p$, the equation \eqref{eq:rgnppar} is in non-divergence form. For non-divergence form equations parabolic Harnack's inequalities and related Hölder regularity results were first studied by Cordes \cite{Cordes1956} and Landis \cite{Landis1971}. Parabolic Harnack's inequality for a non-divergence form equation with bounded and measurable coefficients was proven by Krylov and Safonov \cite{Krylov1981}. Further regularity results for general fully nonlinear equations were proven by Wang \cite{Wang1990}, see also \cite{Imbert2013}. To the best of our knowledge, our proof is partly new even in the special case of the $p$-parabolic equation since it does not rely on the divergence structure.
	
	The idea of the proof of the right inequality in \eqref{eq:intro} is to first locate a local supremum and establish a positive lower bound in some small ball around this point. Then we use specific subsolutions as comparison functions to expand the set of positivity over the unit ball for a specific time slice using the comparison principle.  Our proof uses the connection of equation \eqref{eq:rgnppar} and the $p$-parabolic equation established by Parviainen and Vázquez in \cite{Parviainen2020} to construct suitable comparison functions. Heuristically, radial solutions to the non-divergence form problem can be interpreted as solutions to divergence form $p$-parabolic equation in a fictitious dimension $d$, which does not need to be an integer. The proof of the left inequality is based on estimating the values of a function in the specific time slice by using the other inequality with suitable radii and scaling of constants. 
	
	Our proofs often are split into three different cases because the behavior of solutions to \eqref{eq:rgnppar} depends on the value of $q$. For the degenerate case $q>2$, the right-side inequality is proven in \cite{Parviainen2020} and we prove the singular case $q<2$ as well as the case of $q$ near $2$. This is done separately to obtain stable constants as $q\to2$. For the left-side inequality, the singular case was proven in \cite{Kurkinen2022} and we prove the remaining cases.
	
	DiBenedetto's proof uses the theory of weak solutions but since the equation \eqref{eq:rgnppar} is in non-divergence form, unless $q=p$, we use the theory of viscosity solutions instead. Because of this, we cannot directly use energy estimates as in \cite{Dibenedetto2008} or in \cite{Dibenedetto1993}. Even defining solutions is non-trivial for this type of equations. A suitable definition taking singularities of the problem into account was established by Ohnuma and Sato \cite{Ohnuma1997}.  When $q=2$, we get the normalized $p$-parabolic equation arising from game theory which was first examined in the parabolic setting in \cite{Manfredi2010}. This problem has had recent interest for example in \cite{Jin2017}, \cite{Hoeg2019}, \cite{Dong2020} and \cite{Andrade2022}. We also point out that normalized equations have been studied in connection to image processing \cite{Does2011}, economics \cite{Parviainen2014} and machine learning \cite{Calder2019}.
	The general form of \eqref{eq:rgnppar} has been examined for example in
	\cite{Imbert2019} and \cite{Parviainen2020} in addition to \cite{Kurkinen2022}.

	We would like to thank Benny Avelin for pointing out a scaling mistake in the proof of Theorem \ref{thm:parharnack} in the initial version of this paper. This was fixed by replacing previous oscillation estimate with new Corollary \ref{cor:oscfix}. This has no effect on the final result.
	\subsection{Results}
	We work with the exponent range
	\begin{equation}
	\label{eq:range}
	q>\begin{cases}
	1 & \text{ if }p\geq\frac{1+n}{2},\\
	\frac{2(n-p)}{n-1}& \text{ if }1<p<\frac{1+n}{2},
	\end{cases}
	\end{equation}
	which is optimal for the intrinsic Harnack's inequality as we prove in Section \ref{sec:range}. For the elliptic version of the inequality where we get both estimates without waiting time, the optimal range is to assume \eqref{eq:range} and $q<2$, as we proved in \cite{Kurkinen2022}. The notation used for space-time cylinders is defined in the next section.
	\begin{theorem}
		\label{thm:gbackharnack}
		Let $u \geq 0$ be a viscosity solution to \eqref{eq:rgnppar} in $Q_{1}^{-}(1)$ and let the range condition \eqref{eq:range} hold. Fix $\left(x_{0}, t_{0}\right) \in Q_{1}^{-}(1)$ such that $u(x_0,t_0)>0$. Then there exist $\gamma=\gamma(n, p, q)$, ${c}={c}(n, p, q)$ and ${\sigma=\sigma(n,p,q)>1}$ such that
		\begin{equation*}
		\gamma^{-1}\sup_{B_r(x_0)}u(\cdot,t_0-\theta r^{q})\leq u(x_0,t_0)\leq\gamma\inf_{B_{r}(x_0)}u(\cdot,t_0+\theta r^{q})
		\end{equation*}
		where
		\begin{equation*}
		\theta={c}u\left(x_{0}, t_{0}\right)^{2-q},
		\end{equation*}
		whenever $(x_0,t_0)+Q_{\sigma r}(\theta) \subset Q_{1}^{-}(1)$.
	\end{theorem}
	We prove this theorem in Sections \ref{sec:forward} and \ref{sec:backward} after first introducing prerequisites and proving auxiliary results in Sections \ref{sec:pre} and \ref{sec:comp}. The theorem is proven by first establishing the right inequality, from now on called the forward Harnack's inequality, and then using this result to prove the left inequality, henceforth called backward Harnack's inequality. These names are standard in the literature. We prove the forward inequality by first locating the local supremum of our function and establishing a positive lower bound in some small ball around the supremum point. This differs from the integral Harnack's inequality used by DiBenedetto for weak solutions at this step \cite[Chapter VII]{Dibenedetto1993}. The proof of this integral inequality uses the divergence form structure of the $p$-parabolic equation and thus is not available to us without a new proof. Next, we expand the positivity set around the obtained supremum point by using suitably constructed viscosity subsolutions and the comparison principle. In the singular case, we first expand the set in the time direction using one comparison function and then expand it sidewise for a specific time slice using another one. In the degenerate case, a single Barenblatt-type function is enough to get a similar result. Yet we need a different comparison function to handle exponents near $q=2$ if we wish to have stable constants as $q\to2$. We construct these viscosity subsolutions in Section \ref{sec:comp}. 
	
	For the backward Harnack's inequality, the singular case is proven as \cite[Theorem 5.2]{Kurkinen2022}, and we prove the remaining cases in Section \ref{sec:backward}. The case $q=2$ is a direct consequence of the forward inequality as we do not have to deal with intrinsic scaling. The proof of the degenerate case follows the proof of the similar result for the $p$-parabolic equation \cite[Section 5.2]{Dibenedetto2012} and uses the forward inequality and proceeds by contradiction that the backward inequality has to hold.
	
	In Section \ref{sec:cov} we prove covering arguments that take the intrinsic scaling into account. We do this by repeatedly iterating Harnack's inequality and choosing points and radii taking the intrinsic scaling into account. In the last Section \ref{sec:range}, we prove that if $q$ does not satisfy the range condition \eqref{eq:range}, it must vanish in finite time and thus cannot satisfy the intrinsic Harnack's inequality. Thus the range condition is optimal.

	\section{Prerequisites}
	\label{sec:pre}
	When $\nabla u\not =0$, we denote
	\begin{equation}
	\Delta_p^qu:=\abs{\nabla u}^{q-p}\div\left(\abs{\nabla u}^{p-2}\nabla u\right)=\abs{\nabla u}^{q-2}(\Delta u + (p-2)\Delta_\infty^Nu),
	\end{equation}
	where $p>1$ and $q>1$ are real parameters and the normalized or game theoretic infinity Laplace operator is given by
	\begin{equation*}
	\Delta_\infty^Nu:=\sum_{i,j=1}^{n}\frac{\partial_{x_i}u \, \partial_{x_j}u \, \partial_{x_i x_j}u}{\abs{\nabla u}^2}.
	\end{equation*}
	Thus the equation \eqref{eq:rgnppar} can be written as
	\begin{equation*}
	\partial_tu=\Delta_p^qu.
	\end{equation*}
	Let $\Om\subset\Rn$ be a domain and denote $\Om_T=\Om\times(0,T)$ the space-time cylinder and
	\begin{equation*}
	\partial_{p}\Om:=\left(\Om\times\{0\}\right)\cup\left({\partial\Om\times[0,T]}\right)
	\end{equation*}
	its parabolic boundary. We will mainly work with the following type of cylinders
	\begin{align*}
	Q_{r}^{-}(\theta)&:=B_r(0)\times(-\theta r^q,0],\\
	Q_{r}^{+}(\theta)&:=B_r(0)\times(0,\theta r^q)
	\end{align*}
	where $\theta$ is a positive parameter that determines the time-wise length of the cylinder relative to $r^q$. We denote the union of these cylinders as
	\begin{equation}
	Q_{r}(\theta):=Q_{r}^{+}(\theta)\cup Q_{r}^{-}(\theta) \label{def:cylinders}
	\end{equation} and when not located at the origin, we denote
	\begin{align*}
	(x_0,t_0)+Q_{r}^{-}(\theta)&:=B_r(x_0)\times(t_0-\theta r^q,t_0],\\
	(x_0,t_0)+Q_{r}^{+}(\theta)&:=B_r(x_0)\times(t_0,t_0+\theta r^q),\\
	(x_0,t_0)+Q_{r}(\theta)&:=B_r(x_0)\times(t_0-\theta r^q,t_0+\theta r^q).
	\end{align*}
	Apart from the case $p=q$, the equation \eqref{eq:rgnppar} is in non-divergence form and thus the standard theory of weak solutions is not available, and we will use the concept of viscosity solutions instead. Moreover, the equation is singular for $2>q>1$, and thus we need to restrict the class of test function in the definition to retain good a priori control on the behavior of solutions near the singularities. We use the definition first introduced in \cite{Ishii1995} for a different class of equations and in \cite{Ohnuma1997} for our setting. This is the standard definition in this context and it naturally lines up with the p-parabolic equation ($p=q$), where notions of weak and viscosity solutions are equivalent for all $p\in(1,\infty)$\cite{Juutinen2001,Parviainen2020,Siltakoski2021}. See also \cite{Julin2011}.
	
	Denote
	\begin{equation}
	\label{eq:gnnpparf}
	F(\eta,X)=\abs{\eta}^{q-2}\Tr\left(X+(p-2)\frac{\eta\otimes \eta}{\abs{\eta}^2}X\right)
	\end{equation}
	where $(a\otimes b)_{ij}=a_ib_j$, so that
	\begin{align*}
	F(\nabla u,D^2u)&=\abs{\nabla u}^{q-2}(\Delta u+(p-2)\Delta_\infty^Nu)=\Delta_p^qu
	\end{align*}
	whenever $\nabla u\not=0$.
	Let $\mathcal{F}(F)$ be the set of functions $f\in C^2([0,\infty))$ such that
	\begin{equation*}
	f(0)=f'(0)=f''(0)=0 \text{ and } f''(r)>0 \text{ for all }r>0,
	\end{equation*}
	and also require that for $g(x):=f(\abs{x})$, it holds that
	\begin{equation*}
	\lim_{\substack{x\to0\\x\not=0}}F(\nabla g(x),D^2g(x))=0.
	\end{equation*}
	This set $\mathcal{F}(F)$ is never empty because it is easy to see that $f(r)=r^\beta\in\mathcal{F}(F)$ for any $\beta>\max(q/(q-1),2)$. Note also that if $f\in\mathcal{F}(F)$, then $\lambda f\in\mathcal{F}(F)$ for all $\lambda>0$.
	
	Define also the set
	\begin{equation*}
	\Sigma=\{\sigma\in C^1(\R)\mid \sigma \text{ is even}, \sigma(0)=\sigma'(0)=0, \text{ and }\sigma(r)>0 \text{ for all }r>0\}.
	\end{equation*}
	We use $\mathcal{F}(F)$ and $\Sigma$ to define an admissible set of test functions for viscosity solutions.
	\begin{definition}\sloppy
		A function $\vp\in C^2(\Om_T)$ is admissible at a point $(x_0,t_0)\in\Om_T$ if either ${\nabla \vp(x_0,t_0)\not=0}$ or there are $\delta>0$, $f\in\mathcal{F}(F)$ and $\sigma\in\Sigma$ such that
		\begin{equation*}
		\abs{\vp(x,t)-\vp(x_0,t_0)-\partial_{t}\vp(x_0,t_0)(t-t_0)}\leq f(\abs{x-x_0})+\sigma(t-t_0),
		\end{equation*}
		for all $(x,t)\in B_\delta(x_0)\times (t_0-\delta, t_0+\delta)$. A function is admissible in a set if it is admissible at every point of the set.
	\end{definition}
	Note that by definition a function $\vp$ is automatically admissible in $\Om_T$ if either $\nabla\vp(x,t)\not=0$ in $\Om_T$ or the function $-\vp$ is admissible in $\Om_T$.
	\begin{definition}
		A function $u:\Om_T\rightarrow\mathbb{R}\cup\left\{ \infty\right\} $
		is a viscosity supersolution to
		\[
		\partial_{t}u=\Delta_{p}^{q}u\quad\text{in }\Om_T
		\]
		if the following three conditions hold.
		\begin{enumerate}
			\label{def:super}
			\item $u$ is lower semicontinuous,
			\item $u$ is finite in a dense subset of $\Om_T$,
			\item whenever an admissible $\vp\in C^{2}(\Om_T)$ touches $u$ at $(x,t)\in\Om_T$
			from below, we have
			\[
			\begin{cases}
			\partial_{t}\vp(x,t)-\Delta_{p}^{q}\vp(x,t)\geq0 & \text{if }\nabla \vp(x,t)\not=0,\\
			\partial_{t}\vp(x,t)\geq0 & \text{if }\nabla \vp(x,t)=0.
			\end{cases}
			\]
		\end{enumerate}
		A function $u:\Om_T\rightarrow\mathbb{R}\cup\left\{ -\infty\right\} $
		is a viscosity subsolution if $-u$ is a viscosity supersolution. A function $u:\Om_T\rightarrow\mathbb{R} $
		is a viscosity solution if it is a supersolution and a subsolution.
	\end{definition}
	
	In our proof of the forward Harnack's inequality for the singular range, we need an oscillation to obtain our initial lower bound because we do not have access to weak Harnack-type estimates used by DiBenedetto. We prove this result using a scaling argument and a singular range version of the oscillation estimate used in the degenerate case. This result is proven in \cite[Corollary 7.2]{Parviainen2020} for the case $q<2$. The proofs remain largely the same but we will present the proof here for the convenience of the reader.
	\newpage
\begin{lemma}
	\label{le:soscil2}
	Let $u$ be a viscosity solution to \eqref{eq:rgnppar} in $\Om_T$ and assume $q<2$. For any $\delta\in(0,1)$, there exists $C:=C(n,p,q,\delta)>1$ such that the following holds. Suppose that $\omega_0$ is such that for $a_0 := \om_0^{2-q}$, we have $Q_{4r}(a_0)\subset\Om_T$ and
	\begin{equation*}
		\osc_{Q_r(a_0)}u\leq\om_0,
	\end{equation*}
	and define the sequences
	\begin{equation*}
		r_i:=C^{-i}r, \quad \om_i:=\delta\om_{i-1}, \quad a_i:=\om_i^{2-q}
	\end{equation*}
	where $i=1,2,\dots$. 
	Then it holds that
	\begin{equation*}
		Q_{r_{i+1}}(a_{i+1})\subset Q_{r_i}(a_i) \quad \text{ and }\quad \osc_{Q_{r_i}(a_{i})}u\leq\om_i.
	\end{equation*}
\end{lemma}
\begin{proof}
	Observe that $Q_{r_{i+1}}(a_{i+1})\subset Q_{r_i}(a_i)$ holds as long as in the time-direction we have
	\begin{equation*}
		a_{i+1}r_{i+1}^q=\left(\delta\om_i\right)^{2-q}C^{-(i+1)q}r^q=\delta^{2-q}C^{-q}\left(\om_i\right)^{2-q}(C^{-i}r)^q\leq a_ir_i^q
	\end{equation*}
	where the last inequality holds if we choose $C$ to satisfy $C^q\delta^{q-2}\geq1$. To prove the second claim we will use induction.
	
	The case $i=0$ holds by assumption. Suppose that the claim holds for some $i=k$ meaning
	\begin{equation*}
		\osc_{Q_{r_k}(a_{k})}u\leq\om_k
	\end{equation*}
	and define
	\begin{equation*}
		u_k(x,t):=\frac{u(r_kx,a_k r_k^qt)-\inf_{Q_{r_k}(a_{k})}u}{\om_k}.
	\end{equation*}
	By induction assumption $\sup_{Q_1(1)}u_k\leq 1$.
	By change of variables, we can rewrite
	\begin{align*}
		\label{eq:oschold3b}
		\osc_{Q_{r_{k+1}}(a_{k+1})}\frac{u}{\om_k}&=\osc_{(x,t)\in Q_1(1)}\frac{u(r_{k+1}x,a_{k+1}r_{k+1}^qt)}{\om_k}\\&=\osc_{(x,t)\in Q_1(1)}\frac{u(C^{-1}r_{k}x,\delta^{2-q}C^{-q}a_{k}r_{k}^qt)-\inf_{Q_{r_k}(a_{k})}u}{\om_k} \\&=\osc_{Q_{C^{-1}}(\delta^{2-q}C^{-q})}u_k \numberthis
	\end{align*}
	Next, we will use the Hölder estimates proved in \cite{Imbert2019} to estimate the oscillation. By \cite[Lemma 3.1]{Imbert2019}, there exists a constant $C_1:=C_1(n,p,q,\norm{u_k}_{L^\infty(Q_{4r}(1))})$ such that
	\begin{equation}
		\label{eq:oschold1b}
		\sup_{\substack
			{t,s\in[-1,1]\\
				t\not=s}}\frac{\abs{u_k(x,t)-u_k(x,s)}}{\abs{t-s}^{\frac12}}\leq C_1
	\end{equation}
	and by using \cite[Lemma 2.3]{Imbert2019} for $y=x_0$ and $t=t_0$, there exists a constant $C_2:=C_2(n,p,q,\norm{u_k}_{L^\infty(Q_{16r}(1))})$ such that
	\begin{equation}
		\label{eq:oschold2b}
		u_k(x,t)-u_k(y,t)\leq C_2\left(\abs{x-y}+\abs{x-y}^2\right).
	\end{equation}
	By our induction assumption and the definition of $u_k$, $\norm{u_k}_{L^\infty(Q_{4r})(1)}\leq1$ and thus $C_1$ and $C_2$ can be chosen independent of the solution.

	Now \eqref{eq:oschold3b} can be estimated with  \eqref{eq:oschold1b} and \eqref{eq:oschold2b} in the following way: Denote ${G:=Q_{C^{-1}}(\delta^{2-q}C^{-q})}$ and let $(\bar{x},\bar{t})\in G$ be the point where $\sup_{G}u_k$ is obtained and $(\bar{y},\bar{s})\in G$ be the point where $\inf_{G}u_k$ is obtained. Now for $C_3=\max\{C_1,C_2\}$, we have
	\begin{align*}
		\label{eq:oschold4b}
		\osc_{G}u_k&\leq u_k(\bar{x},\bar{t})-u_k(\bar{y},\bar{s})+u_k(\bar{y},\bar{t})-u_k(\bar{y},\bar{t})\\&\leq C_1\abs{\bar{t}-\bar{s}}^{\frac{1}{2}}+C_2\left(\abs{\bar{x}-\bar{y}}+\abs{\bar{x}-\bar{y}}^2\right)
		\\&\leq C_3\left(\left[\delta^{2-q}C^{-q}\right]^{\frac{1}{2}}+C^{-1}+C^{-2}\right)
		\\&\leq C_3\left(\frac{\delta}{3C_3}+\frac{\delta}{3C_3}+\frac{\delta}{3C_3}\right)=\delta \numberthis
	\end{align*}
	where the last inequality holds if we choose
	\begin{equation*}
		C=\max\left\{\frac{3C_3}{\delta},\frac{(3C_3)^{\frac{2}{q}}}{\delta}, \delta^{-1}\right\}.
	\end{equation*}
	Thus by combining \eqref{eq:oschold3b} and \eqref{eq:oschold4b}, we get
	\begin{equation*}
		\osc_{Q_{r_{k+1}}(a_{k+1})}u\leq\delta\om_k=\om_{k+1}
	\end{equation*}
	as desired. Assumption $C\geq\delta^{-1}$ assures that $C^q\delta^{q-2}\geq1$.
\end{proof}
Because the situation is translation invariant, we can move all the cylinders to a point $(\hat{x},\hat{t})$ as long as $u$ remains a solution around the translated cylinder. The lemma also holds for cut cylinders $Q_r^-$ and $Q_r^+$ with the same proof. As a consequence we get the following specific form of oscillation estimate for a fixed time.
\begin{corollary}
	\label{cor:oscold}
	Let $u$ be a viscosity solution to \eqref{eq:rgnppar} in $\Om_T$ and assume $q<2$. Let $(\hat{x},\hat{t})\in\Om_T$ and suppose that $\omega_0$ is such that for $a_0 := \om_0^{2-q}$, we have $(\hat{x},\hat{t})+Q_{4r}^{-}(a_0)\subset\Om_T$ and
	\begin{equation*}
		\osc_{(\hat{x},\hat{t})+Q_r^-(a_0)}u\leq\om_0.
	\end{equation*}
	Then there exists constants $\hat{C}(n,p,q)>1$ and $\nu=\nu(n,p,q)\in(0,1)$ such that for any $0<\rho<r$ it holds
	\begin{equation*}
		\osc_{B_\rho(\hat{x})} u(\cdot,\hat{t})\leq \hat{C}\om_0\left(\frac{\rho}{R}\right)^\nu.
	\end{equation*}
\end{corollary}
\begin{proof}
	Let $\delta$, $C$, $a_k$ and $r_k$ be as in Lemma \ref{le:soscil2} and pick any $\rho\in(0,r]$. Choose an integer $k$ such that
	\begin{equation*}
		C^{-(k+1)}r\leq\rho\leq C^{k}r=r_k
	\end{equation*}
	so that we have
	\begin{equation*}
		k+1\geq-\frac{\log(\frac{\rho}{r})}{\log(C)}.
	\end{equation*}
	Thus by using the definition of $\om_k$, we have
	\begin{equation}
		\label{eq:osccor1}
		\om_k=\delta^k\om_0=\delta^{-1}\delta^{k+1}\om_0\leq\delta^{-1}\delta^{-\frac{\log(\frac{\rho}{r})}{\log(C)}}\om_0=\delta^{-1}\left(\frac{\rho}{r}\right)^{-\frac{\log(\delta)}{\log(C)}}\om_0.
	\end{equation}
	We have $\rho\leq r_k$ and thus by Lemma \ref{le:soscil2}
	\begin{equation*}
		\osc_{B_\rho(\hat{x})} u(\cdot,\bar{t})\leq\osc_{(\hat{x},\hat{t})+Q_\rho^-(a_k)} u\leq\osc_{(\hat{x},\hat{t})+Q_{r_k}^-(a_k)} u\leq\om_k.
	\end{equation*}
	Choosing $\hat{C}=\delta^{-1}$, we can use \eqref{eq:osccor1} to conclude
	\begin{equation*}
		\osc_{B_\rho(\hat{x})} u(\cdot,\bar{t})\leq\hat{C}\left(\frac{\rho}{r}\right)^{-\frac{\log(\delta)}{\log(C)}}\om_0,
	\end{equation*}
	which is strong enough because
	\begin{equation*}
		\nu:=-\frac{\log(\delta)}{\log(C)}\in(0,1),
	\end{equation*}
	because we assumed $C>\delta^{-1}$, when proving Lemma \ref{le:soscil2}.
\end{proof}
This corollary is used to prove the initial estimate for Harnack's inequality in the degenerate case. For $q<2$, we use cylinders with different time scaling. Specifically, for the proof, we need to obtain an estimate on the top of a time cylinder that does not depend on the height of the cylinder. We prove the needed estimate using a scaling argument and the previous corollary.
\begin{corollary}
	\label{cor:oscfix}
	 Let $u$ be a viscosity solution to \eqref{eq:rgnppar} in $\Om_T$ and assume $q<2$. Let $(\hat{x},\hat{t})\in\Om_T$, $R>0$ such that $(\hat{x},\hat{t})+Q_{4R}^-(1)\subset\Om_T$, and suppose that for some $\om>1$ and $0<\eps<\om^{q-2}R^q$, we have
	 \begin{equation}
	 	\label{eq:osclemma1}
	 	\osc_{B_R(\hat{x})\times(\hat{t}-\eps,\hat{t}]}u\leq\om.
	 \end{equation}
	 Then there exits constants $\nu$ and $\hat{C}$ such that for any $0<\rho<R$, we have
	 \begin{equation*}
	 	\osc_{B_\rho(\hat{x})\times(\hat{t}-\left(\frac{\rho}{R}\right)^q\eps,\hat{t}]}u\leq\hat{C}\om\left(\frac{\rho}{R}\right)^\nu,
	 \end{equation*}
	 so that in particular
	 \begin{equation*}
	 	\osc_{B_\rho(\hat{x})}u(\cdot,\hat{t})\leq\hat{C}\om\left(\frac{\rho}{R}\right)^\nu.
	 \end{equation*}
\end{corollary}
\begin{proof}
	Without loss of generality, assume $(\hat{x},\hat{t})=(0,0)$. Set $\ell=\left(\frac{\eps}{\om^{q-2}R^q}\right)^{\frac{1}{2(q-2)}}\geq1$ and
	\begin{equation*}
		v(x,t):=\ell u(x,\ell^{q-2}t).
	\end{equation*}
	Then $v$ is a viscosity solution in $Q_{4R}^{-}(\ell^{2-q})\supset Q_{4R}^{-}(1)$ since by a formal calculation for smooth solutions, we have
	\begin{align*}
		\partial_t v(x,t)&=\ell^{q-1}\partial_tu(x,\ell^{q-2}t)\\
		&=\left[\ell^{q-1}\abs{\nabla u}^{q-2}\left(\Delta u+(p-2)\frac{\nabla u\otimes\nabla u}{\abs{\nabla u}^2}D^2u\right)\right](x,\ell^{q-2}t) \\
		&=\left[\abs{\nabla (\ell u)}^{q-2}\left(\Delta (\ell u)+(p-2)\frac{\nabla (\ell u)\otimes\nabla (\ell u)}{\abs{\nabla (\ell u)}^2}D^2(\ell u)\right)\right](x,\ell^{q-2}t) \\
		&=\left[\abs{\nabla v}^{q-2}\left(\Delta v+(p-2)\frac{\nabla v\otimes\nabla v}{\abs{\nabla v}^2}D^2v\right)\right](x,t).
	\end{align*}
	On the other hand, observe that
	\begin{equation}
		\label{eq:osclemma2}
		R^q(\om^{q-2}R^{-q}\eps)^{\half}=(\om^{q-2}R^{q}\eps)^{\half}=\eps\left(\frac{\eps}{\om^{q-2}R^{q}}\right)^{-\half}=\eps\ell^{2-q}.
	\end{equation}
	Using this and our assumption \eqref{eq:osclemma1}, we can estimate
	\begin{align*}
	\osc_{Q_R^{-}(R^q(\om^{q-2}R^{-q}\eps)^{\half})}v&=\osc_{B_R(0)\times(-\eps \ell^{2-q},0]}\ell u(x,\ell^{q-2}t)\\
	&=\ell\osc_{B_R(0)\times(-\eps,0]} u(x,t) \\
	&\leq\left(\frac{\eps}{\om^{q-2}R^q}\right)^{\frac{1}{2(q-2)}}\om \\
	&=\left(\om^{q-2}R^{-q}\eps\right)^{\frac{1}{2(q-2)}}.
	\end{align*}
	Denoting $\om':=\left(\om^{q-2}R^{-q}\eps\right)^{\frac{1}{2(q-2)}}$ and $a':=(\om')^{q-2}=\left(\om^{q-2}R^{-q}\eps\right)^{\frac{1}{2}}$, this reads as
	\begin{equation*}
		\osc_{Q_R^-(a')}v\leq \om'.
	\end{equation*}
	Since $\om'=l\om\geq\om>1$ and $v$ is a viscosity solution in $Q_{4R}^{-}(1)$, the oscillation estimate Corollary \ref{cor:oscold} yields for any $\rho\in(0,R)$ that
	\begin{equation}
		\label{eq:osclemma3}
		\osc_{Q_\rho^-(a')}v\leq \hat{C}\om'\left(\frac{\rho}{R}\right)^\nu.
	\end{equation}
	Using the definition of $v$ and $a'$ as well as the computation \eqref{eq:osclemma2} on the left side, we get
	\begin{equation}
		\osc_{Q_\rho^-(a')}v(x,t)=\osc_{B_\rho(0)\times(-\eps \ell^{2-q}\left(\frac{\rho}{R}\right)^q,0]}\ell u(x,\ell^{q-2}t)=\ell\osc_{B_\rho(0)\times(-\eps \left(\frac{\rho}{R}\right)^q,0]} u(x,t).
	\end{equation}
	Thus dividing both sides of \eqref{eq:osclemma3} by $\ell$, we get the desired estimate.
\end{proof}

	Our proofs use the following comparison principle, which is Theorem 3.1 in \cite{Ohnuma1997}.
	\begin{theorem}
		\label{thm:comp}
		Let $\Omega\subset\Rn$ be a bounded domain. Suppose that u is a viscosity supersolution and $v$ is a viscosity subsolution to \eqref{eq:rgnppar} in $\Omega_{T}$. If
		$$
		\infty \neq \limsup _{\Omega_{T} \ni(y, s) \rightarrow(x, t)} v(y, s) \leq \liminf _{\Omega_{T} \ni(y, s) \rightarrow(x, t)} u(y, s) \neq-\infty
		$$
		for all $(x, t) \in \partial_{p} \Omega_{T},$ then $v \leq u$ in $\Omega_{T}$.
	\end{theorem}

	\section{Comparison functions}
	\label{sec:comp}
	Comparison functions are used in the standard proof for the intrinsic Harnack's inequality for the divergence form equation to expand the positivity set around the supremum point using the comparison principle. In the degenerate case, a single Barenblatt-type solution is enough to get the estimate but in the singular case, we need two separate subsolutions. The Barenblatt solutions do not have compact support in the singular range and thus we need to find another type of comparison function. Because of the connection of the equation \eqref{eq:rgnppar} and the usual $p$-parabolic equation examined in \cite{Parviainen2020}, we can use similar comparison functions as DiBenedetto in his proof for the singular range. We will need three different subsolutions to handle the singular case and values of $q$ near $q=2$.
	We denote throughout this section
	\begin{equation*}
	\eta:=\frac{p-1}{q-1}
	\end{equation*}
	which is the time-scaling constant connecting \eqref{eq:rgnppar} to the usual $q$-parabolic equation in the radial case.
	
	Assume $q<2$. We will use the following subsolution which is a time-rescaled version of the solution used in the $p$-parabolic case by DiBenedetto \cite[VII.7]{Dibenedetto1993}. 
	Let
	\begin{equation}
	\label{eq:compfunc}
	\Phi(x,t):=\frac{\kappa\rho^{q\xi}}{R(t)^\xi}\left(1-\left(\frac{\abs{x}^q}{R(t)}\right)^{\frac{1}{q-1}}\right)_+^2,
	\end{equation}
	where
	\begin{equation*}
	R(t):=\eta\kappa^{q-2}t+\rho^q
	\end{equation*} and $\kappa$ and $\rho$ are positive parameters and $\xi>1$ is chosen independent of $\kappa$ and $\rho$. By $(\cdot)_+$ we denote the positive part of the function inside the bracket. 
	
	By construction $\supp \Phi(\cdot,0)=B_\rho(0)$ and for $t\geq0$, we get the expanding balls
	\begin{equation*}
	\supp \Phi(\cdot,t)=B_{R(t)^{\frac{1}{q}}}(0)
	\end{equation*}
	and the estimate
	\begin{equation}
	\label{eq:compupperbound}
	\Phi(x,0)\leq\Phi(x,t)\leq\kappa \quad \text{ for } t\geq 0.
	\end{equation}
	We examine $\Phi$ in the domains
	\begin{equation}
	\label{eq:compdomain}
	\mathcal{P}_{\kappa,\xi}:=B_{R(t)^\frac{1}{q}}(0)\times\left(0,\frac{\kappa^{2-q}\rho^q}{\eta\xi}\right).
	\end{equation}
	We have $\Phi\in C^\infty(\mathcal{P}_{\kappa,\xi})\cap C(\overline{\mathcal{P}_{\kappa,\xi}})$ and as we see in the following lemma, we can choose the constant $\xi$ to make $\Phi$ a viscosity subsolution to \eqref{eq:rgnppar} in this set. 
	\begin{lemma}
		\label{le:compfunc}
		Let the range condition \eqref{eq:range} hold and $q<2$. There exists a constant $\xi:=\xi(n,p,q)$ so that $\Phi$ is a viscosity subsolution to \eqref{eq:rgnppar} in  $\Rn\times\left(0,\frac{\kappa^{2-q}\rho^q}{\eta\xi}\right)$.
	\end{lemma}
	\begin{proof}
		The function $\Phi\equiv0$ outside $\mathcal{P}_{\kappa,\xi}$, so it is enough for us to check that $\Phi$ is a viscosity subsolution on the boundary and inside this set.
		Let us first look at the points where $\nabla\Phi\not=0$, because here we can use the radiality of $\Phi$ in spatial coordinates and a simple calculation to simplify our statement to the form
		\begin{equation}
			\label{eq:radcompnot1}
		\partial_t\phi-\abs{\phi'}^{q-2}\left((p-1)\phi'' +\phi'\frac{n-1}{r}\right)\leq 0 \quad \text{ in } \mathcal{P}_{\kappa,\xi}':=\left(0,R(t)^\frac{1}{q}\right)\times\left(0,\frac{\kappa^{2-q}\rho^q}{\eta\xi}\right)
		\end{equation}
		where
		\begin{equation}
		\label{eq:radcompfunc}
		\phi(r,t):=\frac{\kappa\rho^{q\xi}}{R(t)^\xi}\left(1-\left(\frac{r^q}{R(t)}\right)^{\frac{1}{q-1}}\right)_+^2.
		\end{equation}
		We use the following notation during the calculation
		\begin{equation}
		\label{eq:radcompnot}
		R(t):=\eta\kappa^{q-2}t+\rho^q,\quad \mathcal{F}:=1-z^{\frac{1}{q-1}},\quad z:=\frac{r^q}{R(t)},\quad A:=\frac{\kappa\rho^{q\xi}}{R(t)^\xi}.
		\end{equation}
		By direct calculation inside $\mathcal{P}_{\kappa,\xi}'$, we have
		\begin{align*}
		\mathcal{F}'&=-\frac{1}{q-1}z^{\frac{1}{q-1}-1}\frac{qr^{q-1}}{R(t)}=-\frac{q}{q-1}\frac{z^{\frac{1}{q-1}}}{r}\\
		\mathcal{F}''&=-\frac{q}{q-1}\left(\frac{q}{q-1}\frac{z^{\frac{1}{q-1}}}{r^2}-\frac{z^{\frac{1}{q-1}}}{r^2}\right)=-\frac{q}{(q-1)^2}\frac{z^{\frac{1}{q-1}}}{r^2}\\
		\phi'&=2A\mathcal{F}\mathcal{F}'=-2A\mathcal{F}\frac{q}{q-1}\frac{z^{\frac{1}{q-1}}}{r} \\
		\phi''&=2A\left((\mathcal{F}')^2+\mathcal{F}\mathcal{F}''\right)=2A\left(\frac{q^2}{(q-1)^2}\frac{z^{\frac{2}{q-1}}}{r^2}-\mathcal{F}\frac{q}{(q-1)^2}\frac{z^{\frac{1}{q-1}}}{r^2}\right)\\
		&=2A\frac{q}{(q-1)^2}\left(qz^{\frac{1}{q-1}}-\mathcal{F}\right)\frac{z^{\frac{1}{q-1}}}{r^2}
		\end{align*}
		Moreover,
		\begin{align*}
		\label{eq:compfunc3}
		\partial_t\phi&=-\frac{\xi\kappa\rho^{q\xi}}{R(t)^{\xi+1}}\mathcal{F}^2\eta\kappa^{q-2}+2A\mathcal{F}\frac{1}{q-1}z^{\frac{1}{q-1}-1}\frac{r^q}{R(t)^2}\eta\kappa^{q-2} \\
		&=-\frac{\xi\eta\kappa^{q-1}\rho^{q\xi}}{R(t)^{\xi+1}}\mathcal{F}^2+\frac{\eta\kappa^{q-1}\rho^{q\xi}}{R(t)^{\xi+1}}\mathcal{F}\frac{2}{q-1}z^{\frac{1}{q-1}}. \numberthis
		\end{align*}
		Define an operator $\mathcal{L}:C^2(\R)\to\R$ by
		\begin{equation*}
		\mathcal{L}(\phi):=\frac{R(t)^{\xi+1}}{\eta\kappa^{q-1}\rho^{q\xi}\mathcal{F}}\left(\partial_{t}\phi-\abs{\phi'}^{q-2}((p-1)\phi'' +\phi'\frac{n-1}{r})\right).
		\end{equation*}
		By the calculations above, we have
		\begin{align*}
		\mathcal{L}(\phi)&=-\xi\mathcal{F}+\frac{2}{q-1}z^{\frac{1}{q-1}}-\abs{\phi'}^{q-2}2\frac{\kappa^{2-q} R(t)}{\eta\mathcal{F}}\frac{q}{q-1}\left(\eta\left(qz^{\frac{1}{q-1}}-\mathcal{F}\right)\frac{z^{\frac{1}{q-1}}}{r^2}-\mathcal{F}\frac{z^{\frac{1}{q-1}}}{r}\frac{n-1}{r}\right) \\
		&=-\xi\mathcal{F}+\frac{2}{q-1}z^{\frac{1}{q-1}}+\abs{\phi'}^{q-2}2\frac{\kappa^{2-q} R(t)}{\eta}\frac{q}{q-1}\frac{z^{\frac{1}{q-1}}}{r^2}\left(\eta\left(1-\frac{qz^{\frac{1}{q-1}}}{\mathcal{F}}\right)+n-1\right)\\
		&=-\xi\mathcal{F}+\frac{2}{q-1}z^{\frac{1}{q-1}}+\abs{2A\mathcal{F}\frac{q}{q-1}\frac{z^{\frac{1}{q-1}}}{r}}^{q-2}2\kappa^{2-q} R(t)\frac{q}{q-1}\frac{z^{\frac{1}{q-1}}}{r^2}\left(\frac{n-1}{\eta}+1-\frac{qz^{\frac{1}{q-1}}}{\mathcal{F}}\right)\\
		&=-\xi\mathcal{F}+\frac{2}{q-1}z^{\frac{1}{q-1}}+(2A\mathcal{F})^{q-2}\left(\frac{q}{q-1}\right)^{q-1}2\kappa^{2-q} R(t)\frac{z^{\frac{q-2}{q-1}}}{r^{q-2}}\frac{z^{\frac{1}{q-1}}}{r^2}\left(C_1-\frac{qz^{\frac{1}{q-1}}}{\mathcal{F}}\right)\\
		&=-\xi\mathcal{F}+\frac{2}{q-1}z^{\frac{1}{q-1}}+\left(\frac{2q}{q-1}\right)^{q-1}\left(\frac{\rho^{q\xi}}{R(t)^\xi}\mathcal{F}\right)^{q-2}\left(\frac{n-1}{\eta}+1-\frac{qz^{\frac{1}{q-1}}}{\mathcal{F}}\right).
		\end{align*}
		
		Introduce the two sets
		\begin{equation*}
		\mathcal{E}_1:=\left\{(r,t)\in\mathcal{P}_{k,\xi}'\mid \mathcal{F}<\delta \right\}, \qquad \mathcal{E}_2:=\left\{(r,t)\in\mathcal{P}_{k,\xi}'\mid \mathcal{F}\geq\delta \right\}
		\end{equation*}
		where $\delta>0$ is a constant to be chosen. Now inside $\mathcal{E}_1$ we can estimate what is inside the last brackets from above
		\begin{equation*}
		\left(\frac{n-1}{\eta}+1-\frac{qz^{\frac{1}{q-1}}}{\mathcal{F}}\right)\leq\frac{n-1}{\eta}+1-\frac{q}{\mathcal{F}}\leq\frac{n-1}{\eta}+1-\frac{q}{\delta}<0
		\end{equation*}
		if we choose $\delta$ small enough that the last inequality holds. Notice also that both $\mathcal{F}\in[0,1]$ and $\frac{\rho^{q\xi}}{R(t)^\xi}\in[0,1]$ and thus
		\begin{equation*}
		\left(\frac{\rho^{q\xi}}{R(t)^\xi}\mathcal{F}\right)^{q-2}\geq1
		\end{equation*} by our assumption $q<2$. Thus inside $\mathcal{E}_1$ 
		\begin{align*}
		\label{eq:compfunc1}
		\mathcal{L}(\phi)&\leq\frac{2}{q-1}+\left(\frac{2q}{q-1}\right)^{q-1}\left(\frac{n-1}{\eta}+1-\frac{q}{\delta}\right)<0.
		\numberthis
		\end{align*}
		Here we can choose $\delta$ to be small enough to guarantee that the right side of the equation is negative and this can be done without dependence on $\xi$.
		
		Let us next focus on $\mathcal{E}_2$. By the range of $t$, we have
		\begin{equation*}
		\left(\frac{R(t)^\xi}{\rho^{q\xi}}\frac{1}{\mathcal{F}}\right)^{2-q}\leq\left(\frac{\left(\eta\kappa^{q-2}\frac{\kappa^{2-q}\rho^q}{\eta\xi}+\rho^q\right)^\xi}{\rho^{q\xi}}\frac{1}{\mathcal{F}}\right)^{2-q}\leq\left(\frac{\xi+1}{\xi}\right)^{\xi(2-q)}\delta^{q-2}\leq\left(\frac{e}{\delta}\right)^{q-2}
		\end{equation*}
		and thus for $\delta$ we chose above, we have 
		\begin{align*}
		\label{eq:compfunc2}
		\mathcal{L}(\phi)&\leq-\xi\mathcal{F}+\frac{2}{q-1}\left(\frac{R(t)^\xi}{\rho^{q\xi}}\frac{1}{\mathcal{F}}\right)^{2-q}\left(\frac{n-1}{\eta}+1\right)\\&\leq-\xi\mathcal{F}+\frac{2}{q-1}\left(\frac{e}{\delta}\right)^{q-2}\left(\frac{n-1}{\eta}+1\right)\numberthis
		\end{align*}
		in $\mathcal{E}_2$. We can now choose $\xi$ to be large enough to guarantee that the right side is negative. Thus combining the estimates \eqref{eq:compfunc1} and \eqref{eq:compfunc2}, we have that $\mathcal{L}(\phi)\leq 0$ in the entire $\mathcal{P}_{k,\xi}$ and thus by just multiplying the positive scaling factor in the definition of $\mathcal{L}$ away, we have that $\phi$ is a classical subsolution.
		
		We still need to check the points where $\nabla\Phi=0$, because there the simplification we did earlier in \eqref{eq:radcompnot1} does not hold. Also because of the singular nature of the equation \eqref{eq:rgnppar}, the concept of a classical solution does not really make sense at these points and we need to use the definition of viscosity solutions.
		By similar calculation to the radial case using the same notation, we have
		\begin{equation*}
		\abs{\nabla\Phi(x,t)}=\abs{-2A\mathcal{F}\left(\frac{q}{q-1}\right)\frac{\abs{x}^{\frac{q}{q-1}-2}}{R(t)^{\frac{1}{q-1}}}x}=\frac{2A}{R(t)^{\frac{1}{q-1}}}\left(\frac{q}{q-1}\right)\mathcal{F}\abs{x}^{\frac{1}{q-1}}
		\end{equation*}
		and thus the gradient vanishes at the origin as $\frac{1}{q-1}>0$ and in the set $\partial B_{R(t)^{\frac{1}{q}}}(0)\times\left(0,\frac{\kappa^{2-q}\rho^q}{\xi}\right)$ as there $\mathcal{F}=0$. This latter set happens to be the lateral boundary of the support of $\Phi$. By our previous calculation \eqref{eq:compfunc3}, the time derivative of $\Phi$ is
		\begin{equation}
		\partial_t\Phi=-\frac{\xi\kappa^{q-1}\rho^{q\xi}}{R(t)^{\xi+1}}\mathcal{F}^2+\frac{\kappa^{q-1}\rho^{q\xi}}{R(t)^{\xi+1}}\mathcal{F}\frac{2}{q-1}z^{\frac{1}{q-1}},
		\end{equation}
		which clearly satisfies $\partial_t\Phi\leq0$ at the critical points as the first term is negative and the second is zero if either $z=0$ or $\mathcal{F}=0$.
		Let $\vp\in C^2$ be an admissible test function touching $\Phi$ at a critical point $(x,t)$ from above. For any such function $\partial_t\vp(x,t)=\partial_t\Phi(x,t)\leq0$ and thus $\Phi$ is a viscosity subsolution in $\mathcal{P}_{\kappa,\xi}$. The zero function is also a viscosity subsolution so $\Phi$ is a viscosity subsolution in the entire $\Rn\times\left(0,\frac{\kappa^{2-q}\rho^q}{\eta\xi}\right)$ as we already verified the boundary.
	\end{proof}
	The comparison function $\Phi$ defined in \eqref{eq:compfunc} does not give us stable constants as $q\to2$ because the radius we use it for blows up. We can extend the proof of the degenerate case slightly below $q=2$ with a different comparison function and use this to get stable constants in our inequality for the whole range \eqref{eq:range}. Let $\rho$ and $\kappa$ be positive parameters and define the function
	\begin{equation}
	\label{eq:nearcompfunc}
	\mathcal{G}(x,t):=\frac{\kappa \rho^{\frac{\nu}{\lambda(\nu)}}}{\Sigma(t)^\nu}\left(1-\left(\frac{\abs{x}}{\Sigma(t)^{\lambda(\nu)}}\right)^{\frac{q}{q-1}}\right)_+^{\frac{q}{q-1}},
	\end{equation}
	where 
	\begin{equation*}
	\Sigma(t):=\eta\kappa^{q-2}\rho^{(q-2)\frac{\nu}{\lambda(\nu)}}t+\rho^{\frac{1}{\lambda(\nu)}}, \quad t\geq0.
	\end{equation*}Here $\nu>1$ is a constant and
	\begin{equation}
	\label{eq:numlambda}
	\lambda(\nu):=\frac{1-\nu(q-2)}{q}.
	\end{equation}
	The function \eqref{eq:nearcompfunc} is a time-rescaled version of the comparison function introduced by DiBenedetto in \cite[VII 3(i)]{Dibenedetto1993}.
	We also introduce a number
	\begin{equation}
	\label{eq:numq}
	q(\nu):=\frac{4(1+2\nu)}{1+4\nu}.
	\end{equation}
	This number $q(\nu)$ will define the size of the interval around $q=2$, where $\mathcal{G}$ is a viscosity subsolution.
	\begin{lemma}
		\label{le:compnear2}
		Let  $q\in(4-q(\nu),7/3)$. There exists a $\nu:=\nu(n,p)>1$ independent of $q$ such that $\mathcal{G}$ is a viscosity subsolution to \eqref{eq:rgnppar} in $\Rn\times\R^+$.
	\end{lemma}
	\begin{proof} We prove this statement by first showing that $\mathcal{G}$ is a classical subsolution in the support of this function
		\begin{equation*}
		\mathcal{S}:=\supp\mathcal{G}=\left\{(x,t)\in\Rn\times\R^+\Big|\abs{x}<\Sigma(t)^{\lambda(\nu)}, t>0\right\}
		\end{equation*} apart from the points where $\nabla\mathcal{G}=0$ and dealing with the boundary and rest of the space afterward. The function $\mathcal{G}$ is radial with respect to space and thus we can perform our calculations in radial coordinates. Define
		\begin{equation}
		\label{eq:rnearcompfunc}
		g(r,t):=\frac{\kappa \rho^{\frac{\nu}{\lambda(\nu)}}}{\Sigma(t)^\nu}\left(1-\left(\frac{r}{\Sigma(t)^{\lambda(\nu)}}\right)^{\frac{q}{q-1}}\right)_+^{\frac{q}{q-1}},
		\end{equation}
		and 
		\begin{equation*}
		z:=\frac{r}{\Sigma(t)^{\lambda(\nu)}}, \quad \mathcal{F}:=(1-z^{\frac{q}{q-1}})_+, \quad a:=\left(\frac{q}{q-1}\right)^2,  \quad A:=\frac{\kappa \rho^{\frac{\nu}{\lambda(\nu)}}}{\Sigma(t)^\nu},\quad \mathcal{S'}:=\supp g.
		\end{equation*}
		Again whenever $g'\not=0$, we can use the radiality and a quick calculation to simplify our statement to the form
		\begin{equation}
		\partial_tg-\abs{g'}^{q-2}\left((p-1)g'' +g'\frac{n-1}{r}\right)\leq 0 \quad \text{ in } \R\times(0,\infty).
		\end{equation}
		
		Inside $\mathcal{S'}$, we have
		\begin{align*}
		\partial_{t}A&=-\nu\frac{\kappa \rho^{\frac{\nu}{\lambda(\nu)}}}{\Sigma(t)^{\nu+1}}\eta\kappa^{q-2}\rho^{(q-2)\frac{\nu}{\lambda(\nu)}}=-\nu\frac{ \rho^{(q-1)\frac{\nu}{\lambda(\nu)}}}{\Sigma(t)^{\nu+1}}\eta\kappa^{q-1}\\
		\partial_{t}\mathcal{F}&=-\frac{q}{q-1}z^{\frac{1}{q-1}}r\frac{-\lambda(\nu)}{\Sigma(t)^{\lambda(\nu)+1}}\eta\kappa^{q-2}\rho^{(q-2)\frac{\nu}{\lambda(\nu)}}
		\end{align*}
		and thus
		\begin{align*}
		\label{eq:nearcompfunc4}
		\partial_{t}g&=-\nu\frac{ \rho^{(q-1)\frac{\nu}{\lambda(\nu)}}}{\Sigma(t)^{\nu+1}}\kappa^{q-1}\eta\mathcal{F}^{\frac{q}{q-1}}+A\frac{q}{q-1}\mathcal{F}^{\frac{1}{q-1}}\left(-\frac{q}{q-1}z^{\frac{1}{q-1}}r\frac{-\lambda(\nu)}{\Sigma(t)^{\lambda(\nu)+1}}\eta\kappa^{q-2}\rho^{(q-2)\frac{\nu}{\lambda(\nu)}}\right)
		\\
		&=-\nu\frac{ \left(\kappa \rho^{\frac{\nu}{\lambda(\nu)}}\right)^{q-1}}{\Sigma(t)^{\nu+1}}\eta\mathcal{F}^{\frac{q}{q-1}}+Aa\mathcal{F}^{\frac{1}{q-1}}z^{\frac{q}{q-1}}\frac{\lambda(\nu)}{\Sigma(t)}\eta\left(\kappa \rho^{\frac{\nu}{\lambda(\nu)}}\right)^{q-2}\\
		&=\frac{ \left(\kappa \rho^{\frac{\nu}{\lambda(\nu)}}\right)^{q-1}}{\Sigma(t)^{\nu+1}}\left(-\nu \eta\mathcal{F}+a\lambda(\nu)\eta\mathcal{F}^{\frac{1}{q-1}}z^{\frac{q}{q-1}}\right). \numberthis
		\end{align*}
		For the spatial derivatives, we have
		\begin{align*}
		g'&=A\frac{q}{q-1}\mathcal{F}^{\frac{1}{q-1}}\mathcal{F}'=-Aa\mathcal{F}^{\frac{1}{q-1}}\frac{z^{\frac{1}{q-1}}}{\Sigma(t)^{\lambda(\nu)}}\\
		g''&=A\frac{q}{q-1}\left(\frac{1}{q-1}\mathcal{F}^{\frac{1}{q-1}-1}(\mathcal{F}')^2+\mathcal{F}^{\frac{1}{q-1}}\mathcal{F}''\right)\\
		&=A\frac{q}{q-1}\left(\frac{q^2}{(q-1)^3}\mathcal{F}^{\frac{1}{q-1}-1}\frac{z^{\frac{2}{q-1}}}{\Sigma(t)^{2\lambda(\nu)}}-\mathcal{F}^{\frac{1}{q-1}}\frac{q}{(q-1)^2}\frac{z^{\frac{1}{q-1}-1}}{\Sigma(t)^{2\lambda(\nu)}}\right)\\
		&=\frac{Aa}{q-1}\left(\frac{q}{q-1}z^{\frac{q}{q-1}}-\mathcal{F}\right)z^{\frac{1}{q-1}-1}\frac{\mathcal{F}^{\frac{1}{q-1}-1}}{\Sigma(t)^{2\lambda(\nu)}}
		\end{align*}
		so finally
		\begin{align*}
		\abs{g'}^{q-2}&\left((p-1)g'' +g'\frac{n-1}{r}\right)	\\&=\left(\frac{Aa}{\Sigma(t)^{\lambda(\nu)}}\right)^{q-2}\left(\mathcal{F}z\right)^{\frac{q-2}{q-1}}\left(Aa\frac{p-1}{q-1}\left(\frac{q}{q-1}z^{\frac{q}{q-1}}-\mathcal{F}\right)z^{\frac{1}{q-1}-1}\frac{\mathcal{F}^{\frac{1}{q-1}-1}}{\Sigma(t)^{2\lambda(\nu)}}-Aa\mathcal{F}^{\frac{1}{q-1}}\frac{z^{\frac{1}{q-1}}}{\Sigma(t)^{\lambda(\nu)}}\frac{n-1}{r}\right)\\
		&=\frac{(Aa)^{q-1}}{\Sigma(t)^{(q-2)\lambda(\nu)}}\left(\frac{p-1}{q-1}\left(\frac{q}{q-1}z^{\frac{q}{q-1}}-\mathcal{F}\right)\frac{1}{\Sigma(t)^{2\lambda(\nu)}}-\mathcal{F}\frac{r}{\Sigma(t)^{2\lambda(\nu)}}\frac{n-1}{r}\right)\\
		&=\frac{(Aa)^{q-1}}{\Sigma(t)^{q\lambda(\nu)}}\left(C_2z^{\frac{q}{q-1}}-C_1\mathcal{F}\right)
		\end{align*}
		for constants $C_1=\frac{(n-1)(q-1)+p-1}{q-1}$ and $C_2=\frac{q(p-1)}{(q-1)^2}$.
		We define an operator $\mathcal{L}:C^2(\R)\to\R$ by
		\begin{equation}
		\mathcal{L}(g):=\frac{\Sigma(t)^{\nu+1}}{\left(\kappa \rho^{\frac{\nu}{\lambda(\nu)}}\right)^{q-1}}\left(\partial_tg-\abs{g'}^{q-2}\left((p-1)g'' +g'\frac{n-1}{r}\right)\right).
		\end{equation}
		Therefore by the calculation above, we have
		\begin{align*}
		\label{eq:nearcompfunc1}
		\mathcal{L}(g)&=-\nu \eta\mathcal{F}^{\frac{q}{q-1}}+a\lambda(\nu)\eta\mathcal{F}^{\frac{1}{q-1}}z^{\frac{q}{q-1}}-\frac{\Sigma(t)^{\nu+1}}{\left(\kappa \rho^{\frac{\nu}{\lambda(\nu)}}\right)^{q-1}}\frac{a^{q-1}}{\Sigma(t)^{q\lambda(\nu)}}\left(\frac{\kappa \rho^{\frac{\nu}{\lambda(\nu)}}}{\Sigma(t)^{\nu}}\right)^{q-1}\left(C_2z^{\frac{q}{q-1}}-C_1\mathcal{F}\right)\\
		&=-\nu \eta\mathcal{F}^{\frac{q}{q-1}}+a\lambda(\nu)\eta\mathcal{F}^{\frac{1}{q-1}}z^{\frac{q}{q-1}}+\Sigma(t)^{(2-q)\nu+1-q\lambda(\nu)}a^{q-1}\left(C_1\mathcal{F}-C_2z^{\frac{q}{q-1}}\right)\\
		&=-\nu \eta\mathcal{F}^{\frac{q}{q-1}}+a\lambda(\nu)b\mathcal{F}^{\frac{1}{q-1}}z^{\frac{q}{q-1}}+a^{q-1}\left(C_1\mathcal{F}-C_2z^{\frac{q}{q-1}}\right)\numberthis
		\end{align*}
		where the exponent of $\Sigma(t)$ is zero because of \eqref{eq:numlambda}.
		We introduce two sets
		\begin{align*}
		\mathcal{E}_1:&=\left\{(x,t)\in\R\times(0,\infty)\mid z^{\frac{q}{q-1}}\geq \frac{1}{2}\left(1+\frac{C_1}{C_1+C_2}\right) \right\} \\ \mathcal{E}_2:&=\left\{(x,t)\in\R\times(0,\infty)\mid z^{\frac{q}{q-1}}<\frac{1}{2}\left(1+\frac{C_1}{C_1+C_2}\right) \right\}
		\end{align*}
		and note that as $\lambda$ is decreasing with respect to $q$, we have
		\begin{equation}
		\label{eq:lambdaest}
		\frac{1}{4}=\lambda(q(\nu))\leq\lambda(\nu)\leq\lambda(4-q(\nu))=\frac{6\nu+1}{8\nu}\leq \frac{7}{8} \quad \text{for}\quad q\in[4-q(\nu),q(\nu)].
		\end{equation}

		Inside $\mathcal{E}_1$, the first term can be small depending on the data but the lower bound we chose for $z^{\frac{q}{q-1}}$ ensures that the rest of the terms are negative on their own without dependence on $\nu$.
		Using \eqref{eq:nearcompfunc1} and estimate \eqref{eq:lambdaest}, it follows that in $\mathcal{E}_1$ it holds 
		\begin{align*}
		\label{eq:nearcompfunc2}
		\mathcal{L}(g)&\leq -\nu \eta\mathcal{F}^{\frac{q}{q-1}}+ a\lambda(\nu)\eta\mathcal{F}^{\frac{1}{q-1}}z^{\frac{q}{q-1}}+a^{q-1}\left(C_1(1-z^{\frac{q}{q-1}})_+-C_2z^{\frac{q}{q-1}}\right) \\
		&\leq a\lambda(\nu)\eta\mathcal{F}^{\frac{1}{q-1}}+a^{q-1}\left(C_1-(C_1+C_2)z^{\frac{q}{q-1}}\right)\\
		&\leq \hat{a}\left(\lambda(\nu)\eta+C_1-\frac{1}{2}\left(C_1+C_2\right)\left(1+\frac{C_1}{C_1+C_2}\right)\right)\\
		&= \hat{a}\left(\lambda(\nu)\eta-\frac{q(p-1)}{2(q-1)^2}\right)\\
		&\leq \hat{a}\eta\left(\frac{7}{8}-\frac{q}{2(q-1)}\right)\leq 0 \numberthis
		\end{align*}
		where $\hat{a}=\max\{a,a^{q-1}\}$. The last inequality holds because we assumed that $q<\frac{7}{3}$. Notice that this estimate holds for all $\nu$ but only for $q\in(4-q(\nu),\frac{7}{3})$ depending on the $\nu$ we pick.
		
		In $\mathcal{E}_2$, we have
		\begin{equation*}
		\mathcal{F}\geq\frac{C_2}{2(C_1+C_2)}
		\end{equation*}
		and we can ensure that $\mathcal{L}(g)$ is negative by choosing a suitably large $\nu$.
		We again estimate using \eqref{eq:nearcompfunc1} and \eqref{eq:lambdaest} that inside $\mathcal{E}_2$, it holds
		\begin{align*}
		\label{eq:nearcompfunc3}
		\mathcal{L}(g)&\leq -\nu \eta\mathcal{F}^{\frac{q}{q-1}}+a\lambda(\nu)\eta\mathcal{F}^{\frac{1}{q-1}}z^{\frac{q}{q-1}}+a^{q-1}\left(C_1\mathcal{F}-C_2z^{\frac{q}{q-1}}\right)\\
		\leq&-\nu \eta\left(\frac{C_2}{2(C_1+C_2)}\right)^{\frac{q}{q-1}}+\hat{a}\left(\frac{7}{8}b+C_1\right). \numberthis
		\end{align*}
		Choose
		\begin{equation*}
		\nu:=\max_{q\in[8/5,7/3]}\hat{a}\left(\frac{7}{8}+\frac{C_1}{\eta}\right)\left(\frac{C_2}{2(C_1+C_2)}\right)^{-\frac{q}{q-1}}
		\end{equation*}
		so that for this $\nu$, we have $\mathcal{L}(g)\leq0$ in $\mathcal{E}_2$ by \eqref{eq:nearcompfunc3}. Notice that this choice of $\nu$ depends on $n$ and $p$ but not $q$ and that $4-q(\nu)>\frac{8}{5}$ for all $\nu\geq1$. Thus for this choise of $\nu$, we get that by \eqref{eq:nearcompfunc2} and \eqref{eq:nearcompfunc3}, we have $\mathcal{L}(g)\leq0$ in the classical sense in $\mathcal{S'}$ for all $q\in(4-q(\nu),\frac{7}{3})$.
		
		We still have to check the points where $\nabla\mathcal{G}=0$. The gradient for the original function \eqref{eq:nearcompfunc} is
		\begin{equation*}
		\nabla\mathcal{G}(x,t)=A\frac{q}{q-1}\mathcal{F}^{\frac{1}{q-1}}\mathcal{F}'=-Aa\mathcal{F}^{\frac{1}{q-1}}z^{\frac{1}{q-1}}\frac{x}{\abs{x}\Sigma(t)^{\lambda(\nu)}}=-Aa\mathcal{F}^{\frac{1}{q-1}}\abs{x}^{\frac{1}{q-1}-1}\frac{x}{\Sigma(t)^{q\lambda(\nu)}}
		\end{equation*}
		which exists and vanishes at the origin as $\frac{1}{q-1}>0$ and also vanishes when $\mathcal{F}=0$, that is when $x\in\partial\mathcal{R}$. Using the time derivative we calculated in \eqref{eq:nearcompfunc4}, we have that for $x\in\partial\mathcal{R}$ it holds $\partial_t\mathcal{G}=0$ and for $x=0$, we have
		\begin{equation*}
		\partial_t\mathcal{G}=-\nu\frac{ \left(\kappa \rho^{\frac{\nu}{\lambda(\nu)}}\right)^{q-1}}{\Sigma(t)^{\nu+1}}\leq0.
		\end{equation*}
		Let $\vp\in C^2$ be an admissible test function touching $\Phi$ at a critical point $(x,t)$ from above. For any such function $\partial_t\vp(x,t)=\partial_t\mathcal{G}(x,t)\leq0$ and thus $\Phi$ is a viscosity subsolution in $\mathcal{S}$. In $\left(\Rn\times\R^+\right)\setminus\mathcal{S}$, any admissible test function touching $\mathcal{G}$ from above must have zero time-derivative and thus $\mathcal{G}$ is a viscosity subsolution in the entire $\Rn\times\R^+$.
	\end{proof}
	We need one more comparison function to handle expanding the sidewise positivity set in our proof of the singular forward Harnack's inequality in Theorem \ref{thm:parharnack}. This differs from the degenerate case where only one Barenblatt type comparison function is used \cite[Theorem 7.3]{Parviainen2020}.
	
	Let $k$ and $\nu$ be positive parameters and consider cylindrical domains with annular cross-section
	\begin{equation}
	\label{eq:cylcompset}
	\mathcal{C}(\theta):=\left\{\nu<\abs{x}<1\right\}\times(0,\theta).
	\end{equation}
	For these parameters and a constant $\zeta$, we define
	\begin{equation}
	\label{eq:cylcompfunc}
	\Psi(x,t):=k\left(1-\abs{x}^2\right)_+^{\frac{q}{q-1}}\left(1+k^{\frac{2-q}{q-1}}\zeta\left(\frac{\abs{x}^q}{\eta t}\right)^{\frac{1}{q-1}}\right)^{-\frac{q-1}{2-q}}.
	\end{equation}
	This is a rescaled version of the comparison function introduced by DiBenedetto in \cite[VII 6]{Dibenedetto1993}. Our set \eqref{eq:cylcompset} has different scaling compared to DiBenedetto's as we feel this slightly simplifies the roles of parameters. After finding a suitable $\zeta$ to ensure that $\Psi$ is a subsolution, we can pick $k$ to set what value $\Psi$ attains on the inner lateral boundary and finally pick $\nu$ to set the size of the hole in the annular cross-section of our cylinder to be of suitable radius. In our proof of the forward inequality these are picked in \eqref{eq:pickk} and \eqref{eq:picknu}. We present the proof in detail for the ease of the reader and to fix some typos in the literature.
	\begin{lemma}
		\label{le:cylcomp}
		Let the range condition \eqref{eq:range} hold and $q<2$. There exist constants $\zeta:=\zeta(n,p,q)$ and $\Theta:=\Theta(n,p,q)$ such that for every $0<\nu<1$ and $k>0$, \eqref{eq:cylcompfunc} is a viscosity subsolution to the equation \eqref{eq:rgnppar} in $C(\theta)$ for
		\begin{equation}
		\label{eq:cylcompconst}
		\theta=\nu^qk^{2-q}\Theta.
		\end{equation}
	\end{lemma}
	\begin{proof}
		The function $\Psi$ is radial and thus we will again do our calculations in radial coordinates. Define
		\begin{equation*}
		\psi(r,t):=k\left(1-r^2\right)_+^{\frac{q}{q-1}}\left(1+k^{\frac{2-q}{q-1}}\zeta\left(\frac{r^q}{\eta t}\right)^{\frac{1}{q-1}}\right)^{-\frac{q-1}{2-q}}
		\end{equation*}
		and denote
		\begin{equation*}
		z:=k^{\frac{2-q}{q-1}}\zeta\left(\frac{r^q}{\eta t}\right)^{\frac{1}{q-1}}, \quad \mathcal{F}:=1+z,\quad w:=\frac{k}{\mathcal{F}^{\frac{q-1}{2-q}}},\quad v:=(1-r^2)^{\frac{q}{q-1}},
		\end{equation*}
		so that $\psi=vw$.
		Whenever $\psi'\not=0$, we can simplify our statement to the form
		\begin{equation*}
		\partial_t\psi-\abs{\psi'}^{q-2}\left((p-1)\psi'' +\psi'\frac{n-1}{r}\right)\leq 0 \quad \text{ in } \mathcal{C'}(\theta):=\left\{\nu<r <1\right\}\times(0,\theta).
		\end{equation*}
		We have
		\begin{align*}
		v'&=-\frac{2rq}{q-1}\left(1-r^2\right)^{\frac{1}{q-1}}=-\frac{2rq}{q-1}v^{\frac{1}{q}} \\
		v''&=\frac{4r^2q}{(q-1)^2}\left(1-r^2\right)^{\frac{2-q}{q-1}}-\frac{2q}{q-1}\left(1-r^2\right)^{\frac{1}{q-1}}=\frac{4r^2q}{(q-1)^2}v^{\frac{2-q}{q}}-\frac{2q}{q-1}v^{\frac{1}{q}} \\
		w'&=k\frac{1-q}{2-q}\mathcal{F}^{-\frac{1}{2-q}}\frac{q}{q-1}k^{\frac{2-q}{q-1}}\zeta\left(\frac{r}{\eta t}\right)^{\frac{1}{q-1}}=-\frac{q}{2-q}\frac{w}{\mathcal{F}}\frac{z}{r} \\
		w''&=-\left(\frac{q}{2-q}\right)\left[-\left(\frac{q}{2-q}\right)\frac{w}{\mathcal{F}^2}\frac{z^2}{r^2}+\left(\frac{q}{q-1}\right)\frac{w}{\mathcal{F}}\frac{z}{r^2}-\left(\frac{q}{q-1}\right)\frac{w}{\mathcal{F}}\frac{z^2}{r^2}-\frac{w}{\mathcal{F}}\frac{z}{r^2}\right] \\
		&=\left(\frac{q}{2-q}\right)\left[\left(\frac{q}{(2-q)(q-1)}\right)\frac{w}{\mathcal{F}^2}\frac{z^2}{r^2}-\frac{1}{q-1}\frac{w}{\mathcal{F}}\frac{z}{r^2}\right] \\
		&=\frac{q^2}{(2-q)^2(q-1)}\frac{w}{\mathcal{F}^2}\frac{z^2}{r^2}-\frac{q}{(2-q)(q-1)}\frac{w}{\mathcal{F}}\frac{z}{r^2}.
		\end{align*}
		Define operators $\mathcal{Q}:C^2(\R)\to\R$ and $\mathcal{R}:C^2(\R)\to\R$ by
		\begin{equation*}
		\mathcal{Q}(\psi):=\partial_{t}\psi-\abs{\psi'}^{q-2}\left((p-1)\psi'' +\psi'\frac{n-1}{r}\right)
		\end{equation*}
		and
		\begin{equation*}
		\mathcal{R}(\psi):=-(p-1)\psi''-\frac{n-1}{r}\psi'
		\end{equation*}
		so that $\mathcal{Q}(\psi)=\partial_t\psi+\abs{\psi'}^{q-2}\mathcal{R}(\psi)$. Using $\psi'=w'v+wv'$ and $\psi''=w''v+2w'v'+wv''$, we can estimate $\mathcal{R}(\psi)$ to obtain
		
		\begin{align*}
		\label{eq:cylcompfunc5}
		\mathcal{R}(\psi)= & -(p-1)(w^{\prime\prime}v+2w^{\prime}v^{\prime}+wv^{\prime\prime})-\frac{n-1}{r}(w^{\prime}v+wv^{\prime})\\
		&=-(p-1)\Bigg[\left(\frac{q^{2}}{(2-q)^{2}(q-1)}\frac{w}{\mathcal{F}^{2}}\frac{z^{2}}{r^{2}}-\frac{q}{(2-q)(q-1)}\frac{w}{\mathcal{F}}\frac{z}{r^{2}}\right)v+\frac{4q^{2}}{(2-q)(q-1)}\frac{w}{\mathcal{F}}\frac{z}{r}v^{\frac{1}{q}}\\ &\quad\quad\quad\quad\quad+w\left(\frac{4r^{2}q}{(q-1)^{2}}v^{\frac{2-q}{q}}-\frac{2q}{q-1}v^{\frac{1}{q}}\right)\Bigg] -\frac{n-1}{r}\left(-\frac{q}{2-q}\frac{w}{\mathcal{F}}\frac{z}{r}v-\frac{2rq}{q-1}v^{\frac{1}{q}}w\right)\\
		= & -\left[\frac{(p-1)q^{2}}{(2-q)^{2}(q-1)}\frac{z}{\mathcal{F}}-\frac{(p-1)q}{(2-q)(q-1)}-\frac{(n-1)q}{2-q}\right]\frac{w}{\mathcal{F}}\frac{z}{r^{2}}v\\
		& -\frac{4q^{2}(p-1)}{(2-q)(q-1)}\frac{z}{\mathcal{F}r}wv^{\frac{1}{q}}-\frac{4r^{2}q(p-1)}{(q-1)^{2}}wv^{\frac{2-q}{q}}+\frac{2q(p-1)}{q-1}wv^{\frac{1}{q}}+\frac{2(n-1)q}{q-1}wv^{\frac{1}{q}}\\
		=: & \frac{q}{2-q}\eta\left[d-\frac{q}{(2-q)}\frac{z}{\mathcal{F}}\right]\frac{wz}{\mathcal{F}r^{2}}v+A.\numberthis
		\end{align*}
		Here $d=\frac{(n-1)(q-1)}{(p-1)}+1$ and $A$ consists of four latter terms.
		Next, we will prove that $A$ is negative for suitably large $z$ and we prove the technical part of this as a separate lemma after finishing this proof.
		Let $Z(p,q,N)$ be the positive constant given by the Lemma \ref{le:alemma} proven below and note that to use this lemma, we will need to restrict $\theta$ to make sure that
		\begin{equation*}
		z\geq Z \text{ for all }(r,t)\in C(\theta).
		\end{equation*}
		The correct choice turns out to be 
		\begin{equation}
		\label{eq:cylcompfunc4}
		\theta\leq\frac{\zeta^{q-1}\nu^qk^{2-q}}{\eta}\frac{1}{Z^{q-1}}
		\end{equation}
		as plugging this into the definition of $z$, we get
		\begin{equation*}
		z=k^{\frac{2-q}{q-1}}\zeta\left(\frac{r^q}{\eta t}\right)^{\frac{1}{q-1}}\geq k^{\frac{2-q}{q-1}}\zeta\left(\frac{\nu^q}{\eta t}\right)^{\frac{1}{q-1}}\geq k^{\frac{2-q}{q-1}}\zeta\left(\frac{\nu^q}{\zeta^{q-1}\nu^qk^{2-q}\frac{1}{Z^{q-1}}}\right)^{\frac{1}{q-1}}=Z.
		\end{equation*}
		Thus by Lemma \ref{le:alemma}
		\begin{align*}
		A& =  -wv^{\frac{1}{q}}\frac{4q^{2}(p-1)}{(2-q)(q-1)}\frac{z}{\mathcal{F}r}-wv^{\frac{2-q}{q}}\frac{4r^{2}q(p-1)}{(q-1)^{2}}+wv^{\frac{1}{q}}\frac{2q(p-1)}{q-1}+wv^{\frac{1}{q}}\frac{2(n-1)q}{q-1}\\
		&=  \frac{2q}{q-1}wv^{\frac{1}{q}}\left(-2\frac{q(p-1)}{2-q}\frac{z}{\mathcal{F}r}-2\frac{p-1}{q-1}v^{\frac{1-q}{q}}r^{2}+p+n-2\right)\\
		&=  \frac{2q}{q-1}wv^{\frac{1}{q}}\left(-2\frac{q(p-1)}{2-q}\frac{z}{(1+z)}\frac{1}{r}-2\frac{p-1}{q-1}\frac{r^{2}}{1-r^2}+p+n-2\right)\\
		&\leq 0
		\end{align*}
		and thus combining this with \eqref{eq:cylcompfunc5}, we get
		\begin{equation}
		\label{eq:cylcompfunc69}
		\mathcal{R}(\psi)\leq\frac{q}{2-q}\eta v\frac{w}{\mathcal{F}}\frac{z}{r^2}\left[d-\frac{q}{2-q}\frac{z}{\mathcal{F}}\right].
		\end{equation}
		Next, we estimate
		\begin{align*}
		\abs{\psi'}&=-\psi'=-w'v-wv'=\frac{q}{2-q}\frac{w}{\mathcal{F}}\frac{z}{r}v+w\frac{2rq}{q-1}v^{\frac{1}{q}}\\
		&=\frac{w}{r}\left(\frac{q}{2-q}\frac{z}{\mathcal{F}}(1-r^2)^{\frac{q}{q-1}}+\frac{2q}{q-1}r^2(1-r^2)^{\frac{1}{q-1}}\right)\\
		&\leq\frac{w}{r}\left(\frac{q}{2-q}+\frac{2q}{q-1}\right)=:C_1\frac{w}{r}
		\end{align*}
		and by direct calculation
		\begin{equation*}
		\partial_t\psi=v\left(\frac{1-q}{2-q}\frac{w}{\mathcal{F}}\left(-\frac{1}{q-1}\frac{z}{t}\right)\right)=\frac{1}{2-q}v\frac{w}{\mathcal{F}}\frac{z}{t}
		\end{equation*}
		and thus we get
		\begin{align}
		\label{eq:cylcompfunc70}
		\abs{\psi'}^{2-q}\partial_{t}\psi\leq\frac{C_1^{2-q}}{2-q}\left(\frac{w}{r}\right)^{2-q}v\frac{w}{\mathcal{F}}\frac{z}{t}.
		\end{align}
		Set
		\begin{equation*}
		\mathcal{L}(\psi)=\frac{(2-q)\mathcal{F}r^2}{vwz}\abs{\psi'}^{2-q}\mathcal{Q}(\psi)=\frac{(2-q)\mathcal{F}r^2}{vwz}\left(\abs{\psi'}^{2-q}\partial_t\psi+\mathcal{R}(\psi)\right)
		\end{equation*}
		and plug in our estimates \eqref{eq:cylcompfunc69} and \eqref{eq:cylcompfunc70} to get
		\begin{align*}
		\label{eq:cylcompfunc3}
		\mathcal{L}(\psi)&\leq\frac{(2-q)\mathcal{F}r^2}{vwz}\left(\frac{C_1^{2-q}}{2-q}\left(\frac{w}{r}\right)^{2-q}v\frac{w}{\mathcal{F}}\frac{z}{t}+\frac{q}{2-q}\eta v\frac{w}{\mathcal{F}}\frac{z}{r^2}\left[d-\frac{q}{2-q}\frac{z}{\mathcal{F}}\right]\right)\\
		&=C_1^{2-q}w^{2-q}\frac{r^q}{t}+\eta q\left[d-\frac{q}{2-q}\frac{z}{\mathcal{F}}\right].\numberthis
		\end{align*}
		By our definition of $w$ and $z$
		\begin{equation*}
		w^{2-q}\frac{r^q}{t}=\left(\frac{z}{1+z}\right)^{q-1}\frac{k^{2-q}}{\left(k^{\frac{2-q}{q-1}}\zeta\left(\frac{r^q}{\eta t}\right)^{\frac{1}{q-1}}\right)^{q-1}}\frac{r^q}{t}\leq\frac{\eta}{\zeta^{q-1}}
		\end{equation*}
		and
		\begin{align*}
		\eta q\left[d-\frac{q}{2-q}\frac{z}{\mathcal{F}}\right]&=\eta \frac{q}{2-q}\left[d(2-q)-\frac{qz}{1+z}\right]=\eta \frac{q}{2-q}\left[d(2-q)-q+\frac{q}{1+z}\right]\\&=:\eta \frac{q}{2-q}\left[-\lambda+\frac{q}{\mathcal{F}}\right].
		\end{align*}
		Using these we can further estimate \eqref{eq:cylcompfunc3} to get
		\begin{equation}
		\label{eq:cylcompfunc6}
		\mathcal{L}(\psi)\leq\frac{C_1^{2-q}\eta}{\zeta^{q-1}}+\eta \frac{q}{2-q}\left[-\lambda+\frac{q}{\mathcal{F}}\right].
		\end{equation}
		Now finally if we assume
		\begin{equation}
		\label{eq:cylcompfunc7}
		z>\frac{2q}{\lambda}
		\end{equation} we have $\frac{q}{\mathcal{F}}=\frac{q}{1+z}<\frac{\lambda}{2}$ and can choose the $\zeta$ that satisfies
		\begin{equation*}
		\frac{C_1^{2-q}\eta}{\zeta^{q-1}}-\eta\frac{q}{2-q}\frac{\lambda}{2}\leq0.
		\end{equation*}
		For this $\zeta$, the estimate \eqref{eq:cylcompfunc6} becomes $\mathcal{L}(\psi)\leq 0$ and
		we have that $\psi$ is a classical subsolution. To ensure that only $z$ satisfying both $z\geq Z$ and \eqref{eq:cylcompfunc7} are in our annulus $\mathcal{C'}(\theta)$, we need to further restrict $\theta$ we picked in \eqref{eq:cylcompfunc4} to make sure that
		\begin{equation*}
		t<k^{2-q}\zeta^{q-1}r^q\left(\frac{\lambda}{2q}\right)^{q-1}\eta^{-1}
		\end{equation*}
		in the set.
		By the definition of $\mathcal{C'}(\theta)$, we have $r\geq\nu$ so it suffices to choose
		\begin{equation*}
		\theta\leq\left(\frac{\lambda}{2q}\right)^{q-1}\frac{\zeta^{q-1}\nu^qk^{2-q}}{\eta}
		\end{equation*}
		so picking
		\begin{equation*}
		\theta:=\frac{\zeta^{q-1}}{\eta}\min\left\{\left(\frac{\lambda}{2q}\right)^{q-1},\frac{1}{Z^{q-1}}\right\}\nu^qk^{2-q}=:\Theta(n,p,q)\nu^qk^{2-q}
		\end{equation*}
		all estimates hold and $\psi$ is a classical subsolution in $\mathcal{C'}(\theta)$.
		We still have to check the points where $\nabla\Psi=0$. By direct calculation, denoting
		\begin{equation*}
		v=(1-\abs{x}^2)^{\frac{q}{q-1}} \quad \text{and}\quad w=k\left(1+k^{\frac{2-q}{q-1}}\zeta\left(\frac{\abs{x}^q}{\eta t}\right)^{\frac{1}{q-1}}\right)^{-\frac{q-1}{2-q}},
		\end{equation*}
		we have
		\begin{equation*}
		\nabla\Psi(x)=\frac{2q}{q-1}v^{\frac{1}{q}}xw+v\left(-\frac{q}{2-q}\right)k^{\frac{1}{q-1}}\left(1+k^{\frac{2-q}{q-1}}\zeta\left(\frac{\abs{x}^{\frac{q}{q-1}}}{\left(\eta t\right)^{\frac{1}{q-1}}}\right)\right)^{-\frac{1}{2-q}}\zeta\left(\frac{\abs{x}^{\frac{2-q}{q-1}}}{\left(\eta t\right)^{\frac{1}{q-1}}}\right)x
		\end{equation*}
		and it is easy to see that $\nabla\Psi(x)=0$ if and only $\abs{x}=1$ or $x=0$. The origin is outside our domain so let $(y,s)$ be an arbitrary point such that $\abs{y}=1$ and $s\in(0,\theta)$ and let $\vp\in C^2$ be an admissible test function touching $\Psi$ from above at $(y,s)$. At such point
		\begin{equation*}
		\partial_{t}\vp(y,s)=\partial_{t}\Psi(y,s)=\frac{1}{2-q}(1-\abs{y}^2)^{\frac{q}{q-1}}\frac{w}{\mathcal{F}}\frac{z}{t}=0
		\end{equation*}
		and same trivially holds when touching a point in $\left(\Rn\setminus B_1(0)\right)\times(0,\theta)$ and thus $\Psi$ is a viscosity subsolution in $\left(\Rn\setminus B_\nu(0)\right)\times(0,\theta)$. This finishes the proof of Lemma \ref{le:cylcomp}.
	\end{proof}
	Next, we will prove Lemma \ref{le:alemma} that we used in the above proof to show that $A$ was negative.
	\begin{lemma}
		\label{le:alemma}
		There exists a constant $Z=Z(p,q,n)$ such that for all $z\geq Z$ and all $r\in(0,1)$ we have
		\begin{equation*}
		E(r):=2\frac{q(p-1)}{2-q}\frac{z}{(1+z)}\frac{1}{r}+2\frac{p-1}{q-1}\frac{r^{2}}{1-r^2}-p-n+2\geq 0.
		\end{equation*}
	\end{lemma}
	\begin{proof}
		Let $K:=\max\{K_1,K_2\}$ where
		\begin{equation*}
		K_1:=\frac{n-1}{p-1}\frac{2-q}{2q}, \quad K_2:=\left(1-\frac{n-1}{p-1}\right)\frac{2-q}{3q}.
		\end{equation*}
		We begin by showing that $K<1$ using the range condition \eqref{eq:range}. 
		
		We first consider the case where $p<\frac{n+1}{2}$ and $q>\frac{2(n-p)}{n-1}$. Since the latter inequality implies
		\begin{equation*}
		\frac{q}{2-q}>\frac{\frac{2(n-p)}{n-1}}{2-2\frac{(n-p)}{n-1}}=\frac{n-p}{p-1},
		\end{equation*}
		we obtain
		\begin{equation*}
		K_1=\frac{n-1}{p-1}\frac{2-q}{2q}<\frac{1}{2}\frac{n-1}{n-p}<\frac{1}{2}\frac{n-1}{n-\frac{n+1}{2}}=\frac{n-1}{2n-n-1}=1
		\end{equation*}
		using the upper bound on $p$. Similarly, we estimate
		\begin{equation*}
		K_2\leq\frac{1}{3}\frac{p-1}{n-p}\left(1+\frac{n-1}{p-1}\right)=\frac{1}{3}\left(\frac{p+n-2}{n-p}\right)<\frac{1}{3}\frac{\frac{n+1}{2}+n-2}{n-\frac{n+1}{2}}=1.
		\end{equation*}
		In the case $p>\frac{n+1}{2}$, we have directly
		\begin{equation*}
		K_1=\frac{n-1}{p-1}\frac{2-q}{2q}\leq\frac{n-1}{\frac{n+1}{2}-1}\frac{2-q}{2q}=2\frac{n-1}{n-1}\frac{2-q}{2q}=\frac{2-q}{q}<1
		\end{equation*}
		and
		\begin{equation*}
		K_2=\frac{1}{3}\frac{2-q}{q}\left(1+\frac{n-1}{p-1}\right)\leq\frac{1}{3}\frac{2-q}{q}\left(1+\frac{n-1}{\frac{n+1}{2}-1}\right)=\frac{2-q}{q}<1
		\end{equation*}
		so hence we have $K<1$ for all exponents satisfying \eqref{eq:range}. Now observe that this implies
		\begin{equation*}
		\frac{z}{z+1}\geq K
		\end{equation*}
		if and only if
		\begin{equation*}
		z\geq\frac{K}{1-K}.
		\end{equation*}
		Denote $Z=\frac{K}{1-K}$ so that by above we have
		\begin{equation}
		\label{eq:alemma1}
		\frac{z}{z+1}\geq K \quad \text{ for all } z\geq Z.
		\end{equation}
		Now, we estimate $E(r)$ separately in the cases $r\geq\frac{2}{3}$ and $r<\frac{2}{3}$. If we first assume $r\geq\frac{2}{3}$, this implies $\frac{r^2}{1-r^2}\geq\frac{4}{5}=:a$ so using \eqref{eq:alemma1}, we can estimate
		\begin{align*}
		E(r)&\geq2\frac{q(p-1)}{2-q}K_1\frac{1}{r}+2\frac{p-1}{q-1}a-p-n+2\\
		&=(p-1)\left(\frac{2q}{2-q}K_1+\frac{2a}{q-1}-1-\frac{n-1}{p-1}\right)\\
		&=(p-1)\left(\frac{2q}{2-q}K_1+\frac{2a+1-q}{q-1}-\frac{n-1}{p-1}\right)\\
		&\geq(p-1)\left(\frac{2q}{2-q}K_1-\frac{n-1}{p-1}\right)\\
		&=0,
		\end{align*}
		where the last identity follows from the definition of $K_1$. If $r\leq\frac{2}{3}$, we discard the second term with $r$ and estimate again using \eqref{eq:alemma1} to get
		\begin{align*}
		E(r)&\leq 2\frac{q(p-1)}{2-q}\frac{3}{2}K_2-p-n+2\\
		&=(p-1)\left(\frac{1}{3}\frac{q}{2-q}K_2-1-\frac{n-1}{p-1}\right)\\
		&=0,
		\end{align*}
		where we used the definition of $K_2$.
	\end{proof}
	\section{Forward intrinsic Harnack's inequality}
	\label{sec:forward}
	In their paper \cite{Parviainen2020}, Parviainen and Vázquez prove the forward Harnack's inequality for viscosity solutions of \eqref{eq:rgnppar} in the degenerate case $q>2$. In this section, we prove the remaining singular case $q<2$ and the case of $q$ near 2. For the proof of the same results for the standard singular $p$-parabolic equation see \cite[VII.9]{Dibenedetto1993}.
	
	In the proof, we first rescale the equation into a simpler form, locate the local supremum of the function in a specific cylinder, and use oscillation estimates to show that there exists a small ball on a time slice where the function is strictly larger than a value depending on the singularity of the equation \eqref{eq:rgnppar}. A key difference compared to the degenerate case is that the cylinders we use have linear time scaling, but their height depends on a constant $\sigma$. We need our initial estimates to be independent of this constant so that we can later use it to ensure our comparison function is a subsolution for a sufficiently long time interval. Barenblatt-type solutions have infinite speed of propagation for $q < 2$ and hence do not work as comparison functions, unlike in the degenerate case.
	
	In the strictly singular case, we next choose $\sigma$ and use the comparison function constructed in Lemma \ref{le:compfunc} to expand the set of positivity in the time direction, thereby obtaining a similar lower bound extended from one time slice to a space-time cylinder. Finally, we use a second comparison function, constructed in Lemma \ref{le:cylcomp}, to widen the set of positivity in the spatial direction, filling the entire ball we are interested in and obtaining the final estimate.
	
	At the end of this section, we prove the inequality for values of $q$ near 2. This case is similar to the degenerate case and requires only one comparison function; however, we use the one constructed in Lemma \ref{le:compnear2} instead of the Barenblatt solution used in the degenerate case. This method yields stable constants as $q \to 2$ from either side. Our main result is the following.
	
%	
%	To get to a suitable setting to use our oscillation estimate, we need a way to estimate the supremum in the past time-slice. Our proof uses a supersolution with infinite boundary values which we constructed in \cite[Lemma 4.1]{Kurkinen2022}.
	\begin{theorem}
		\label{thm:parharnack}
		Let $u \geq 0$ be a viscosity solution to \eqref{eq:rgnppar} in $Q_{1}^{-}(1)$ and let the range condition \eqref{eq:range} hold. Fix $\left(x_{0}, t_{0}\right) \in Q_{1}^{-}(1)$ such that $u(x_0,t_0)>0$. Then there exist $\mu=\mu(n, p, q)$ and ${c}={c}(n, p, q)$ such that
		\begin{equation*}
		u\left(x_{0}, t_{0}\right) \leq \mu \inf _{B_{r}\left(x_{0}\right)} u\left(\cdot, t_{0}+\theta r^q\right)
		\end{equation*}
		where
		\begin{equation*}
		\theta={c}u\left(x_{0}, t_{0}\right)^{2-q},
		\end{equation*}
		whenever $(x_0,t_0)+Q_{4r}(\theta) \subset Q_{1}^{-}(1)$.
	\end{theorem}
	\begin{remark}
		The constants $\mu$ and $c$ can be picked to be stable as $q\to2$ from either side as we show in the proof. As $q$ approaches the lower bound in \eqref{eq:range}, $\mu$ tends to infinity, and $c$ tends to zero.
		As $q\to\infty$, both $\mu$ and $c$ tend to infinity.
	\end{remark}
	\begin{proof}[Proof of Theorem \ref{thm:parharnack}]
		The proof for the degenerate case $q>2$ is given as \cite[Theorem 7.3]{Parviainen2020} and thus we can focus on the singular case $q<2$. 
		Consider the rescaled equation
		\begin{equation}
		v(x,t)=\frac{1}{u(x_0,t_0)}u(x_0+rx,t_0+u(x_0,t_0)^{2-q}tr^q)
		\end{equation}
		which solves
		\begin{equation*}
		\begin{cases}
		\partial_t v=\abs{\nabla v}^{q-p}\div\left(\abs{\nabla v}^{p-2}\nabla v\right) \text{ in } Q \\
		v(0,0)=1,
		\end{cases}
		\end{equation*}
		where $Q=B_4(0)\times(-4^q,4^q)$. Now it is enough to show that there exists positive constants $c_0$ and $\mu_0$ such that
		\begin{equation}
		\label{eq:parharnack1}
		\inf_{B_{1}(0)}v(\cdot,c_0)\geq\mu_0
		\end{equation}
		because then by the definition of $v$, we have
		\begin{align*}
		\mu_0u(x_0,t_0)&\leq u(x_0,t_0) \inf_{x\in B_{1}(0)}\frac{1}{u(x_0,t_0)}u(x_0+rx,t_0+c_0u(x_0,t_0)^{2-q}r^q)\\&=\inf_{B_{r}(x_0)}u(\cdot,t_0+c_0u(x_0,t_0)^{2-q}r^q).
		\end{align*}
		For the first part of the proof, we make the extra assumption $q<2$ and deal with values near $q=2$ afterward. Proof for $q=2$ is easy but we need to deal with values near it separately to ensure that we get stable constants as $q\to2$ from either side. We will prove \eqref{eq:parharnack1} in the following steps.

		\textbf{Step 1: Locating the supremum.} First, we will need to locate the supremum of $v$ in $Q$ before time $t=0$ and establish a positive lower bound for $v$ in some small ball around the supremum point. We do this by using the oscillation result we proved earlier. Let $ \mathcal{Q}_0=\{(0,0)\}$ and for all $\tau\in(0,1)$ and $\sigma\in(0,1)$ to be chosen later define nested expanding cylinders 
		\begin{equation*}
		\mathcal{Q}_\tau:=\left\{(x,t)\in Q\mid \abs{x}<\tau, t\in(-\sigma\tau,0]\right\}
		\end{equation*}
		and the numbers
		\begin{equation*}
		M_\tau:=\sup_{\mathcal{Q}_\tau}v, \qquad N_\tau:=(1-\tau)^{-\frac{q}{2-q}}.
		\end{equation*}
		Notice that $M_0=1=N_0$ and
		\begin{equation*}
		\lim_{\tau\nearrow1}N_\tau=\infty \quad \text{ and } \quad  \lim_{\tau\nearrow1}M_\tau<\infty
		\end{equation*}
		as $v$ is bounded. Therefore by continuity, the equation $M_\tau=N_\tau$ must have a largest root  $\tau_0\in[0,1)$, which satisfies
		\begin{equation*}
		M_{\tau_0}=(1-\tau_0)^{-\frac{q}{2-q}} \quad \text{ and } \quad \sup_{\mathcal{Q}_\tau}v=M_\tau\leq N_\tau \text{ for all }  1>\tau>\tau_0.
		\end{equation*}
		Especially for $\hat{\tau}:=\frac{1+\tau_0}{2}$ we have
		\begin{equation*}
		 M_{\hat{\tau}}\leq N_{\hat{\tau}}= 2^{\frac{q}{2-q}}(1-\tau_0)^{-\frac{q}{2-q}}.
		\end{equation*}

		By continuity of $v$, it achieves the value $M_{\tau_0}$ at some point $(\hat{x},\hat{t})\in\overline{\mathcal{Q}}_{\tau_0}$. For the radius $R=\frac{1-\tau_0}{2}$, we have $(\hat{x},\hat{t})+\mathcal{Q}_{R}\subset \mathcal{Q}_{\hat{\tau}}$ because for any $(y,s)\in(\hat{x},\hat{t})+\mathcal{Q}_{R}$ in space
		\begin{equation*}
			\abs{y}\leq\tau_0+\frac{1-\tau_0}{2}=\hat{\tau}
		\end{equation*}
		and in time
		\begin{equation}
			s\geq-\sigma\tau_0-\sigma\frac{1-\tau_0}{2}=-\sigma\hat{\tau}.
		\end{equation}
		Thus
		\begin{equation}
			\label{eq:parharnack0}
		\osc_{(\hat{x},\hat{t})+\mathcal{Q}_{R}}v\leq\sup_{(\hat{x},\hat{t})+\mathcal{Q}_{R}}v\leq \sup_{\mathcal{Q}_{\hat{\tau}}}v\leq 2^{\frac{q}{2-q}}(1-\tau_0)^{-\frac{q}{2-q}}.	
		\end{equation}
		
		\textbf{Step 2: Establishing an initial lower bound.}
		Next we will prove that there exists a constant $\delta:=\delta(n,p,q)$, such that
			\begin{equation*}
				v(x,\hat{t})\geq\frac{1}{2}(1-\tau_0)^{-\frac{q}{2-q}}=:\kappa
			\end{equation*}
			for all $x\in B_\rho(\hat{x})$ where the radius $\rho=\delta 
			R$. We specifically need for later that there is no dependence on the constant $\sigma$ in this estimate.
			%Lisäsin kannen q:n määritelmään
			
			By equation \eqref{eq:parharnack0}, the function $v$ satisfies the assumptions of Corollary \ref{cor:oscfix} for ${\om=2^{\frac{q}{2-q}}(1-\tau_0)^{-\frac{q}{2-q}}>1}$ and any $\eps<\min\{\sigma R,\om^{q-2}R^q\}$.
			Therefore by the corollary, there exists $\hat{C}$ and $\nu$ such that
			\begin{equation*}
				\osc_{B_\rho(\hat{x})}v(\cdot,\hat{t})\leq \hat{C}\left(2^{\frac{q}{2-q}}(1-\tau_0)^{-\frac{q}{2-q}}\right)\left(\frac{\rho}{R}\right)^\nu
			\end{equation*}
			for any $0<\rho<R$. Pick $\rho=\delta R$ for some $\delta:=\delta(n,p,q)$ small enough to satisfy the inequality $1-\hat{C}\delta^\nu2^{\frac{q}{2-q}}\geq\frac12$ so that
			\begin{align*}
				\label{eq:spacepos}
				v(x,\hat{t})&\geq\inf_{B_{\delta R}(\hat{x})}v(\cdot,\hat{t})=\sup_{B_{\delta R}(\hat{x})}v(\cdot,\hat{t})-\osc_{B_{\delta R}(\hat{x})}v(\cdot,\hat{t})\\
				&\geq v(\hat{x},\hat{t})-\hat{C}\left(\frac{\delta R}{R}\right)^\nu2^{\frac{q}{2-q}}(1-\tau_0)^{-\frac{q}{2-q}}
				\\&=\left(1-\hat{C}\delta^\nu2^{\frac{q}{2-q}}\right)(1-\tau_0)^{-\frac{q}{2-q}}\geq\frac12(1-\tau_0)^{-\frac{q}{2-q}}=\kappa \numberthis
			\end{align*}
			for all $x\in B_{\rho}(\hat{x})$. \\
		\textbf{Step 3: Time expansion of positivity.}
		We have managed to prove the positivity of $v$ in a small ball for time $\hat{t}$ and now we intend to improve this estimate to get positivity in a time cylinder. Consider the translated comparison function $\Phi(x-\hat{x},t-\hat{t})$ introduced in \eqref{eq:compfunc} for choices $\kappa$ and $\rho$ introduced above. By Lemma \ref{le:compfunc}, we have thus that $\Phi$ is a viscosity subsolution to \eqref{eq:rgnppar} in
		\begin{equation}
		(\hat{x},\hat{t})+\mathcal{P}_{\kappa,\xi}:=B_{(\hat{R}(t-\hat{t}))^\frac{1}{q}}(\hat{x})\times\left(\hat{t},\hat{t}+\frac{\kappa^{2-q}\rho^q}{\eta\xi}\right)
		\end{equation}
		for time dependent radius $\hat{R}(t):=\eta\kappa^{q-2}t+\rho^q$. We choose
		\begin{equation*}
		3\sigma:=\frac{\kappa^{2-q}\rho^q}{\eta\xi}=\frac{(1-\tau_0)^{-q}\rho^q}{2^{2-q}\eta\xi}
		\end{equation*}
		where $\sigma$ is the constant we did not yet choose in the definition of our cylinders $\mathcal{Q}_\rho$. Now by \eqref{eq:spacepos} and \eqref{eq:compupperbound}, it holds
		\begin{equation*}
		v(x,\hat{t})\geq\kappa\geq\Phi(x-\hat{x},\hat{t}-\hat{t})
		\end{equation*}
		and by positivity $v\geq\Phi$ on the spatial boundary. Thus by the comparison principle Theorem \ref{thm:comp}
		\begin{equation*}
		v\geq\Phi \text{ in } \left\{\abs{x-\hat{x}}^q<\hat{R}(3\sigma-\hat{t})\right\}\times\left\{0<t-\hat{t}<3\sigma\right\}
		\end{equation*}
		so in particular as $\rho\leq(\hat{R}(3\sigma-\hat{t}))^{\frac{1}{q}}$, we get for $t-\hat{t}\in(\sigma,3\sigma)$ and $\abs{x}\leq\rho$
		\begin{align*}
		\label{eq:timepos}
		v(x,t)&\geq\frac{\kappa\rho^{q\xi}}{\left(\eta\kappa^{q-2}(3\sigma)+\rho^q\right)^\xi}\left(1-\left(\frac{\rho^q}{\eta\kappa^{q-2}\sigma+\rho^q}\right)^{\frac{1}{q-1}}\right)_+^2\\&= \frac{\frac12(1-\tau_0)^{-\frac{q}{2-q}}}{\left(\frac{1}{\xi}+1\right)^\xi}\left(1-\left(\frac{3\xi}{3\xi+1}\right)^{\frac{1}{q-1}}\right)_+^2\\
		&=:\hat{c}(n,p,q)(1-\tau_0)^{-\frac{q}{2-q}}. \numberthis
		\end{align*}
		We do not have a way to know the exact location of $\hat{t}$ inside $\mathcal{Q}_{\tau_0}$ but we know that $\hat{t}\in(-1,0)$ and $\sigma\in(0,1)$. Thus as 
		\begin{equation}
		(\sigma,2\sigma)\subset\bigcap_{\hat{t}\in(-1,0)}(\hat{t}+\sigma,\hat{t}+3\sigma),
		\end{equation}
		we have estimate \eqref{eq:timepos} for all $(x,t)\in B_\rho(\hat{x})\times(\sigma,2\sigma)$. As $q\nearrow2$, we have $\sigma\searrow0$ and hence the set converges towards an empty set. To get the estimate for values of $q$ near 2, we repeat a similar argument but with a different comparison function.\\
		
		\begin{figure}[h]\label{fig:timeexpansion}
		\includegraphics[scale=0.9]{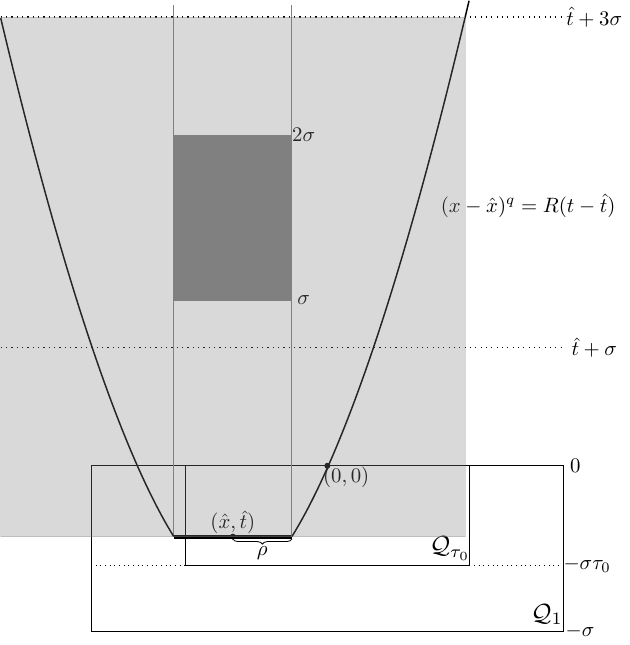}
			\caption{Illustration of the time expansion. We use the comparison principle over the light gray cylinder and get the final estimate over the gray cylinder.}
		\end{figure}

		\textbf{Step 4: Sidewise expansion of positivity.}
		We will next expand the positivity set of $v$ over $B_1(\hat{x})$ for a specific time slice using yet another comparison function to finally get the estimate \eqref{eq:parharnack1}.
		Choose
		\begin{equation}
		\label{eq:pickk}
		k=\hat{c}(1-\tau_0)^{-\frac{q}{2-q}},
		\end{equation} we got from \eqref{eq:timepos},
		\begin{equation}
		\label{eq:picknu}
		\nu:=\frac{\rho}{3}
		\end{equation}
		and let $\theta$ be given by \eqref{eq:cylcompconst} for this $k$ and $\nu$.
		We have \eqref{eq:timepos} for all $(x,t)\in B_\rho(\hat{x})\times(\sigma,2\sigma)$ so we can repeat the previous with a smaller constant $\hat{\sigma}=\min\{\theta,\sigma\}$. We want to use a translated and scaled version of the comparison function
		\begin{equation*}
		\Psi\left(\frac{x-\hat{x}}{3},\frac{t-\hat{\sigma}}{3^q}\right)
		\end{equation*}
		in the annular cylindrical domain
		\begin{equation*}
		\hat{\mathcal{C}}:=\left\{\rho<\abs{x-\hat{x}}<3\right\}\times(\hat{\sigma},2\hat{\sigma})
		\end{equation*}
		where we introduced $\Psi$ in \eqref{eq:cylcompfunc}. This rescaled $\Psi$ is a viscosity subsolution to \eqref{eq:rgnppar} in $\hat{\mathcal{C}}$ by Lemma \ref{le:cylcomp}. Notice that this $\Psi$ vanishes for $x\in \partial B_3(\hat{x})$ or $t=\hat{\sigma}$ and that 
		\begin{equation*}
		\Psi\left(\frac{x-\hat{x}}{3},\frac{t-\hat{\sigma}}{3^q}\right)\leq k =\hat{c}(1-\tau_0)^{-\frac{q}{2-q}}
		\end{equation*}
		everywhere in $\hat{\mathcal{C}}$. Combining this estimate with \eqref{eq:timepos} we have
		\begin{equation*}
		\Psi\left(\frac{x-\hat{x}}{3},\frac{t-\hat{\sigma}}{3^q}\right)\leq v(x,t)
		\end{equation*}
		for all $(x,t)\in\{\abs{x-\hat{x}}=\rho\}\times(\hat{\sigma},2\hat{\sigma})$ by continuity of $v$. Thus we can again use the comparison principle to get $\Psi\leq v$ in the entire set $\hat{\mathcal{C}}$.
		
		In $B_2(\hat{x})\times\{t=2\hat{\sigma}\}$, for our chosen $k=\hat{c}(1-\tau_0)^{-\frac{q}{2-q}}$, we have
		\begin{align*}
		v(x,2\hat{\sigma})&\geq \hat{c}(1-\tau_0)^{-\frac{q}{2-q}}\left(1-\left(\frac{2}{3}\right)^2\right)_+^{\frac{q}{q-1}}\left(1+\left[\hat{c}(1-\tau_0)^{-\frac{q}{2-q}}\right]^{\frac{2-q}{q-1}}\zeta\left(\frac{\left(\frac{2}{3}\right)^q}{2\eta\hat{\sigma} }\right)^{\frac{1}{q-1}}\right)^{-\frac{q-1}{2-q}}\\
		&\geq \inf_{0\leq\tau\leq 1} \hat{c}(1-\tau)^{-\frac{q}{2-q}}\left(1-\left(\frac{2}{3}\right)^2\right)_+^{\frac{q}{q-1}}\left(1+\left[\hat{c}(1-\tau)^{-\frac{q}{2-q}}\right]^{\frac{2-q}{q-1}}\zeta\left(\frac{\left(\frac{2}{3}\right)^q}{2\eta\hat{\sigma} }\right)^{\frac{1}{q-1}}\right)^{-\frac{q-1}{2-q}}\\
		&=:\mu_0(n,p,q).
		\end{align*}
		Thus by taking infimum over $B_1(0)\subset B_2(\hat{x})$, we get
		\begin{equation*}
		\inf_{B_{1}(0)}v(\cdot,2\hat{\sigma})\geq\mu_0
		\end{equation*}
		so we have proved estimate \eqref{eq:parharnack1} for $c_0:=2\hat{\sigma}$.

		\textbf{Case of $q$ near 2.}
		The case of $q$ near 2 is quite similar to the proof we presented above. Let $\eps>0$ be a small number to be fixed later and assume that $q\in(2-\eps,2+\eps)$. This time we define a family of nested expanding cylinders by
		\begin{equation*}
		\mathcal{Q}_\tau:=\left\{(x,t)\in Q\mid \abs{x}<\tau, t\in(-\tau^q,0)\right\}
		\end{equation*}
		so they no longer depend on the constant $\sigma$. We define
		\begin{equation*}
		M_\tau:=\sup_{\mathcal{Q}_\tau}v, \qquad N_\tau:=(1-\tau)^{-\beta}
		\end{equation*}
		for $\beta>0$ to be chosen later similar to the proof of the degenerate case. Again let $\tau_0\in[0,1)$ be the largest root of equation $M_\tau=N_\tau$ to ensure
		\begin{equation*}
		M_{\tau_0}=(1-\tau_0)^{-\beta} \quad \text{ and } \quad M_{\hat{\tau}}\leq 2^{\beta}(1-\tau_0)^{\beta}
		\end{equation*}
		for $\hat{\tau}=\frac{1+\tau_0}{2}$. By continuity of $v$, it achieves the value $M_{\tau_0}$ at some point $(\hat{x},\hat{t})\in\overline{\mathcal{Q}}_{\tau_0}$ and for radii $R=\frac{1-\tau_0}{2}$, we have
		\begin{equation*}
		\sup_{\mathcal{Q}_{R}(1)}v\leq\sup_{\mathcal{Q}_{\hat{\tau}}(1)}v=M_{\hat{\tau}}\leq2^{\beta}(1-\tau_0)^{-\beta}=:\omega_0>1.
		\end{equation*}
		For $q<2$, we can repeat the same steps as we used to obtain \eqref{eq:spacepos} to find $\rho:=\delta R$ for some $\delta:=\delta(n,p,q)$ small enough so that
		\begin{align*}
		v(x,\hat{t})&\geq v(\hat{x},\hat{t})-\hat{C}\left(\frac{\delta R}{R}\right)^\nu2^{\beta}(1-\tau_0)^{-\beta}
		\\&=\left(1-\hat{C}\delta^\nu2^{\beta}\right)(1-\tau_0)^{-\beta}\geq\frac12(1-\tau_0)^{-\beta}=:\kappa
		\end{align*}
		for all $x\in B_{\rho}(\hat{x})$. For $q\geq2$, we repeat the same steps but use \cite[Corollary 7.2]{Parviainen2020} instead of Corollary \ref{cor:oscfix} and pick the smaller of the two radii to get the same estimate. Without loss of generality, assume $\rho=\delta R$.
		
		This is where we need the special subsolution constructed in Lemma \ref{le:compnear2}.
		Let $\kappa$ and $\rho$ be the constants we set above and
		\begin{equation}
		\label{eq:compnear2}
		\mathcal{G}(x,t):=\frac{\kappa \rho^{\frac{\nu}{\lambda(\nu)}}}{\Sigma(t)^\nu}\left(1-\left(\frac{\abs{x}}{\Sigma(t)^{\lambda(\nu)}}\right)^{\frac{q}{q-1}}\right)_+^{\frac{q}{q-1}}.
		\end{equation}
		Consider the translated version $\mathcal{G}(x-\hat{x},t-\hat{t})$ which is a viscosity subsolution to \eqref{eq:rgnppar} in $\Rn\times\R^+$ as long as our exponent $q$ is close enough to 2. Let $\nu$ be the constant given by Lemma \ref{le:compnear2} and pick $\eps=\min\left\{\frac{4(1+2\nu)}{1+4\nu}-2,\frac{1}{3},\frac{1}{\nu}\right\}$. The first two numbers ensure that $\mathcal{G}$ is a viscosity solution by Lemma \ref{le:compnear2} and the restriction $\eps\leq\frac{1}{\nu}$ is here to ensure that $\lambda(\nu)\geq0$ for all $q$ in our range.
		
		At time level $t=c_0$, the support of $\mathcal{G}(x-\hat{x},c_0-\hat{t})$ is the set
		\begin{equation*}
		\supp\mathcal{G}(x-\hat{x},c_0-\hat{t})=\left\{x\in\Rn\Big|\abs{x-\hat{x}}<\Sigma(c_0-\hat{t})^{\lambda(\nu)}\right\}
		\end{equation*}
		where
		\begin{equation*}
		\lambda(\nu)=\frac{1-\nu(q-2)}{q}
		\end{equation*}
		and
		\begin{align*}
		\Sigma(c_0-\hat{t})&=\left(\frac{p-1}{q-1}\right)\kappa^{q-2}\rho^{(q-2)\frac{\nu}{\lambda(\nu)}}\left(c_0-\hat{t}\right)+\rho^{\frac{1}{\lambda(\nu)}}. \\&
		=\left(\frac{p-1}{q-1}\right)\left(\frac12(1-\tau_0)^{-\beta}\right)^{q-2}\left(\delta R\right)^{(q-2)\frac{\nu}{\lambda(\nu)}}\left(c_0-\hat{t}\right)+\rho^{\frac{1}{\lambda(\nu)}} \\&
		=A(1-\tau_0)^{(q-2)\left(\frac{\nu}{\lambda(\nu)}-\beta\right)}\left(c_0-\hat{t}\right)+\rho^{\frac{1}{\lambda(\nu)}}.
		\end{align*}
		Here we used $R=\frac{1-\tau_0}{2}$ and defined \begin{equation*}
		A:=\left(\frac{p-1}{q-1}\right)\left(\frac12\left(\frac{\delta}{2}\right)^{\frac{\nu}{\lambda(\nu)}}\right)^{q-2}.
		\end{equation*}
		We choose
		\begin{equation*}
		\beta=\frac{\nu}{\lambda(\nu)} \quad \text{ and } \quad c_0=\frac{3^{\frac{1}{\lambda(\nu)}}}{A}+\hat{t}
		\end{equation*}
		
		and since $\abs{\hat{x}}<1$ and $\hat{t}\in(-1,0]$, these choices ensure
		\begin{equation*}
		\supp\mathcal{G}(x-\hat{x},c_0-\hat{t})=\left\{\abs{x-\hat{x}}<\left(A\left(\frac{3^{\frac{1}{\lambda(\nu)}}}{A}+\hat{t}-\hat{t}\right)+\rho^{\frac{1}{\lambda(\nu)}}\right)^{\lambda(\nu)}\right\}\supset B_3(\hat{x})
		\end{equation*}
		and thus $B_2(0)\subset\supp\mathcal{G}(x-\hat{x},c_0-\hat{t})$. In the set $\supp\mathcal{G}(x-\hat{x},\hat{t}-\hat{t})=B_{\rho}(\hat{x})$, we have
		\begin{equation*}
		\mathcal{G}(x-\hat{x},\hat{t}-\hat{t})=\kappa\left(1-\left(\frac{\abs{x-\hat{x}}}{\rho}\right)^{\frac{q}{q-1}}\right)_+^{\frac{q}{q-1}}\leq\kappa\leq v(x,t)
		\end{equation*}
		and similarly $\mathcal{G}\leq v$ on the rest of $\partial_{p}\left(B_2(\hat{x})\times[\hat{t},c_0]\right)$ because we assumed $v$ to be positive. Hence by the comparison principle Theorem \ref{thm:comp}
		\begin{align*}
		\inf_{B_{1}(0)}v(\cdot,c_0)&\geq\inf_{B_{1}(0)}\mathcal{G}(\cdot,c_0)\\
		&=\frac{\kappa \rho^{\frac{\nu}{\lambda(\nu)}}}{\left(3^{\frac{1}{\lambda(\nu)}}+\rho^{\frac{1}{\lambda(\nu)}}\right)^\nu}\left(1-\left(\frac{1}{\left(3^{\frac{1}{\lambda(\nu)}}+\rho^{\frac{1}{\lambda(\nu)}}\right)^{\lambda(\nu)}}\right)^{\frac{q}{q-1}}\right)_+^{\frac{q}{q-1}}\\
		&\geq\frac12 \left(\frac{\delta}{2}\right)^{\frac{\nu}{\lambda(\nu)}}\frac{1}{\left(3^{\frac{1}{\lambda(\nu)}}+\rho^{\frac{1}{\lambda(\nu)}}\right)^\nu}\left(1-\left(\frac{1}{3}\right)^{\frac{q}{q-1}}\right)_+^{\frac{q}{q-1}}\\
		&\geq3^{-\frac{\nu}{\lambda(\nu)}}2^{-(1+2\nu)} \left(\frac{\delta}{2}\right)^{\frac{\nu}{\lambda(\nu)}}\left(1-\left(\frac{1}{3}\right)^{\frac{q}{q-1}}\right)_+^{\frac{q}{q-1}}\\
		&=:\mu_0(n,p,q),
		\end{align*}
		so we have proven \eqref{eq:parharnack1}. Notice that all constants used here are stable as $q\to2$ from either side. 
	\end{proof}
	\section{Backward intrinsic Harnack's inequality}
	\label{sec:backward}
	In this section, we will prove the backward intrinsic Harnack's inequality for the optimal range of exponents \eqref{eq:range}. We proved the singular case as Theorem 5.2 in \cite{Kurkinen2022} but the degenerate case has not been proven before to the best of our knowledge for equation \eqref{eq:rgnppar}. The degenerate case is proven for the standard $p$-parabolic equation \cite[Section 5.3]{Dibenedetto2012}. All proofs are based on using the forward inequality in a specific way taking into account the intrinsic scaling. In the degenerate case, we move backward in time centered at $x_0$ seeking for a time where the function obtains a value larger than $\mu u(x_0,t_0)$. We handle the case of such time existing and not existing separately and show that in both cases we get the backward inequality using the forward inequality. The main difference to the singular case is that when $q\geq 2$, we have to assume that ${u(x_0,t_0)>0}$ or the inequality will not hold. The case $q=2$ follows directly from forward Harnack's inequality as we do not have to worry about the intrinsic scaling. In the singular case, the amount of space needed around our space-time cylinder depends on $n$, $p$, and $q$ but we improve this result using covering arguments in the next section.
	\begin{thm:gbackharnack}
		Let $u \geq 0$ be a viscosity solution to \eqref{eq:rgnppar} in $Q_{1}^{-}(1)$ and let the range condition \eqref{eq:range} hold. Fix $\left(x_{0}, t_{0}\right) \in Q_{1}^{-}(1)$ such that $u(x_0,t_0)>0$. Then there exist $\gamma=\gamma(n, p, q)$, ${c}={c}(n, p, q)$ and ${\sigma=\sigma(n,p,q)>1}$ such that
		\begin{equation*}
		\gamma^{-1}\sup_{B_r(x_0)}u(\cdot,t_0-\theta r^{q})\leq u(x_0,t_0)\leq\gamma\inf_{B_{r}(x_0)}u(\cdot,t_0+\theta r^{q})
		\end{equation*}
		where
		\begin{equation*}
		\theta={c}u\left(x_{0}, t_{0}\right)^{2-q},
		\end{equation*}
		whenever $(x_0,t_0)+Q_{\sigma r}(\theta) \subset Q_{1}^{-}(1)$.
	\end{thm:gbackharnack}
	\begin{proof}
		Let $c$ and $\mu$ be the constants we get from Theorem \ref{thm:parharnack} and let ${\theta={c}u\left(x_{0}, t_{0}\right)^{2-q}}$.
		\begin{case}[q<2]The case $q<2$ is Theorem 5.2. in \cite{Kurkinen2022} where we get the theorem for constant $\sigma=\frac{6}{\alpha}$ where $\alpha=(2\mu)^{\frac{q-2}{q}}<1$.
		\end{case}
		Apart from the non-emptyness of $\mathcal{U}_{\alpha}$ this proof extends directly to the case $q=2$ but this can be done easier as we do not have intrinsic time scaling in this case.
		\begin{case}[q=2]
			Let $\bar{t}=t_0-cr^2$ and $y\in B_r(x_0)$. Now by Theorem \ref{thm:parharnack} at the point $(y,\bar{t})$, we get an estimate
			\begin{equation}
			\label{eq:gbackharnack1}
			u(y,\bar{t})\leq\mu\inf_{B_{r}(y)}u(\cdot,\bar{t}+cr^2)=\mu\inf_{B_{r}(y)}u(\cdot,t_0)\leq \mu u(x_0,t_0).
			\end{equation}
			This holds for any $y\in B_r(x_0)$ and thus by taking supremum over all of them we get
			\begin{equation}
			\label{eq:gbackharnack2}
			\sup_{B_r(x_0)}u(\cdot,t-cr^2)\leq\mu u(x_0,t_0),
			\end{equation}
			as desired. The use of Harnack's inequality in \eqref{eq:gbackharnack1} is justified because in the space direction $B_{4r}(y)\subset B_{5r}(x_0)\subset B_{\frac{6}{\alpha}r}(x_0)$ and in the time direction we have
			\begin{equation}
			\bar{t}-c(4r)^2\geq t_0-cr^2-c(4r)^2=t_0-c(5r^2)\geq t_0-c(\sigma r^2),
			\end{equation}
			for any $\sigma>5$.
			
			Inequality \eqref{eq:gbackharnack2} combined with Theorem \ref{thm:parharnack} proves the inequality in the case $q=2$.
		\end{case}
		\begin{case}[q>2]	
			Finally, let $q>2$ where we again have to deal with the time-scaling. Let $\rho$ be a radius such that $(x_0,t_0)+Q_{6\rho}(\theta)\subset Q_{1}^-(1)$ for $\theta=cu(x_0,t_0)^{2-q}$ and define the set
			\begin{equation*}
			{\mathcal{T}=\{t\in (t_0-\theta(4\rho)^q,t_0)\mid u(x_0,t)=2\mu u(x_0,t_0)\}}.
			\end{equation*}
			Now $\mathcal{T}$ is either empty or non-empty. If it happens that $\mathcal{T}\not=\emptyset$, there exists a largest $\tau\in\mathcal{T}$ by continuity of $u$. For a time like this, it must hold that
			\begin{equation}
			\label{eq:degremark5}
			t_0-\tau>cu(x_0,\tau)^{q-2}\rho^q=c\left(2\mu u(x_0,t_0)\right)^{2-q}\rho^q,
			\end{equation}
			because otherwise we can choose $\hat{\beta}\in(0,1)$ such that
			\begin{equation*}
			\tau+cu(x_0,\tau)^{q-2}(\hat{\beta} \rho)^q=t_0
			\end{equation*} and use Theorem \ref{thm:parharnack} on the point $(x_0,\tau)$ for radius $\hat{\beta}\rho$ to get
			\begin{equation*}
			2\mu u(x_0,t_0)=u(x_0,\tau)\leq\mu\inf_{B_{\hat{\beta}\rho}(x_0)}u(\cdot,t_0)\leq u(x_0,t_0).
			\end{equation*}
			This is a contradiction assuming that we have suitable space to use the forward Harnack's inequality. This is automatically satisfied in space as we are centered at $x_0$ and in time we have
			\begin{align*}
			\tau-cu(x_0,\tau)^{2-q}(4\hat{\beta}\rho)^q&>\tau-cu(x_0,\tau)^{2-q}(4\rho)^q=\tau-c(2\mu)^{2-q}u(x_0,t_0)^{2-q}(4\rho)^q\\&>t_0-\theta(4\rho)^q-(2\mu)^{2-q}\theta(4\rho)^q=t_0-\left(1+(2\mu)^{2-q}\right)\theta(4\rho)^q\\&>t_0-\theta(6\rho)^q,
			\end{align*}
			where the last inequality holds for all $\mu>1$ as $1+(2\mu)^{2-q}4^q<2\cdot4^q<6^q$ for $q\geq2$.
			Here we used the fact that $\tau>t_0-\theta(4\rho)^q$ and $u(x_0,\tau)=2\mu u(x_0,t_0)$ by the definition of $\mathcal{T}$.
			Set
			\begin{equation}
			\label{eq:degremark8}
			s=t_0-c(2\mu u(x_0,t_0))^{2-q}\rho^q
			\end{equation} and notice that by \eqref{eq:degremark5}, it holds $s\in(\tau,t_0)$ and
			\begin{equation*}
			u(x_0,s)\leq2\mu u(x_0,t_0).
			\end{equation*}
			Assume thriving for a contradiction that there exists $y\in B_{\rho}(x_0)$ such that
			\begin{equation}
			\label{eq:degremark7}
			u(y,s)=2\mu u(x_0,t_0),
			\end{equation}
			and note that
			\begin{equation*}
			s+cu(y,s)^{2-q}\rho^q=t_0.
			\end{equation*}
			Therefore assuming there is enough room to use Theorem \ref{thm:parharnack}, we get
			\begin{equation*}
			2\mu u(x_0,t_0)=u(y,s)\leq\mu\inf_{B_{\rho}(y)}u(\cdot,s+cu(x_0,t_0)^{2-q}\rho^q)=\mu\inf_{B_{\rho}(y)}u(\cdot,t_0)\leq\mu u(x_0,t_0).
			\end{equation*}
			We have enough room in space as $B_{4\rho}(y)\subset B_{5\rho}(x_0)$ and by \eqref{eq:degremark8} and \eqref{eq:degremark7} in time, it holds
			\begin{align*}
			s-cu(y,s)^{2-q}(4\rho)^q&=t_0-c(2\mu u(x_0,t_0))^{2-q}\rho^q-c(2\mu u(x_0,t_0))^{2-q}(4\rho)^q\\
			&=t_0-\theta\left[\left(2\mu+4^{\frac{q}{2-q}}2\mu\right)^{\frac{2-q}{q}}\rho\right]^q>t_0-\theta(6\rho)^q,
			\end{align*}
			where the last inequality holds for all $\mu>1$ because for $q>2$ we have
			\begin{equation*}
			\left(2\mu+4^{\frac{q}{2-q}}2\mu\right)^{\frac{2-q}{q}}<(2\mu)^{\frac{2-q}{q}}<1.
			\end{equation*}
			Therefore such $y\in B_{\rho}(x_0)$ cannot exist and we have
			\begin{equation*}
			u(y,s)<2\mu u(x_0,t_0) \quad \text{ for all } y\in B_{\rho}(x_0)
			\end{equation*}
			and thus by definition of $s$
			\begin{equation}
			\label{eq:degremark6}
			\sup_{B_{\rho}(x_0)}u(\cdot,t_0-c(2\mu u(x_0,t_0))^{2-q}\rho^q)\leq2\mu u(x_0,t_0)
			\end{equation}
			Let $r=(2\mu)^{\frac{2-q}{q}}\rho\leq \rho$ and rewrite \eqref{eq:degremark6} as
			\begin{equation*}
			u(x_0,t_0)\geq(2\mu)^{-1}\sup_{B_{\rho}(x_0)}u\left(\cdot,t_0-\theta \left((2\mu)^{\frac{2-q}{q}}\rho\right)^q\right)\geq(2\mu)^{-1}\sup_{B_{r}(x_0)}u\left(\cdot,t_0-\theta r^q\right).
			\end{equation*}
			This combined with Theorem \ref{thm:parharnack} for radius $r$ and taking $2\mu$ gives
			\begin{equation}
			\label{eq:degrehar1}
			{(2\mu)}^{-1}\sup_{B_{r}(x_0)}u\left(\cdot,t_0-\theta r^q\right)\leq u(x_0,t_0)\leq\mu\inf_{B_{r}(x_0)}u(\cdot,t_0+\theta r^q)\leq2\mu\inf_{B_{r}(x_0)}u(\cdot,t_0+\theta r^q)
			\end{equation}
			which is what we wanted.
			If it happens that $\mathcal{T}=\emptyset$, we have 
			\begin{equation}
			\label{eq:degremark3}
			u(x_0,t)<2\mu u(x_0,t_0) \text{ for all } t\in(t_0-\theta(4\rho)^q,t_0)
			\end{equation} 
			by continuity of $u$. Assume thriving for a contradiction that
			\begin{equation}
			\label{eq:degremark4}
			\sup_{B_{r}(x_0)}u(\cdot,t_0-\theta r^q)>2\mu^2u(x_0,t_0)
			\end{equation}
			which implies by continuity that there exists a point $x_*\in B_r(x_0)$ such that
			\begin{equation}
			\label{eq:degremark1}
			u(x_*,t_0-\theta r^q)=2\mu^2u(x_0,t_0).
			\end{equation}
			Assuming we have enough room to use Theorem \ref{thm:parharnack} around the point $(x_*,t_0-\theta r^q)$, we have
			\begin{equation}
			\label{eq:degremark2}
			u(x_*,t_0-\theta r^q)\leq\mu\inf_{B_{r}(x_*)}u(\cdot,t_0-\theta r^q+cu(x_*,t_0-\theta r^q)^{2-q}r^q).
			\end{equation}
			The required space here is $(x_0,t_0)+Q_{5r}\subset Q_{1}^-$ because we need to make sure that $B_{4r}(x_*)\subset B_1$. In time we do not need more room because
			\begin{align*}
			t_0-\theta r^q-cu(x_*,t_0-\theta r^q)^{2-q}(4r)^q&=t_0-c\left[u(x_0,t_0)^{2-q}+4^q\left(2\mu^2u(x_0,t_0)\right)^{2-q}\right]r^q\\
			&=t_0-\left[1+4^q(2\mu^2)^{2-q}\right]cu(x_0,t_0)^{2-q}r^q\\
			&=t_0-\theta\left(\left[1+4^q(2\mu^2)^{2-q}\right]^{\frac{1}{q}}r\right)^q\\&\geq t_0-\theta(6r)^q
			\end{align*}
			where the last inequality holds assuming $q\geq2$ and
			\begin{equation*}
			1+4^q(2\mu^2)^{2-q}<6^q
			\end{equation*}
			which is true for any $\mu\geq1$ as $(2\mu^2)^{2-q}\leq1$.
			We can estimate the time level by using equation \eqref{eq:degremark1} to get
			\begin{align*}
			t_0-\theta r^q-cu(x_*,t_0-\theta r^q)^{2-q}r^q&=t_0-c\left(u(x_0,t_0)^{2-q}-u(x_*,t_0-\theta r^q)^{2-q}\right)r^q \\
			&=t_0-c\left(u(x_0,t_0)^{2-q}-\left(2\mu^2 u(x_0,t_0)\right)^{2-q}\right)r^q \\
			&=t_0-\left(1-(2\mu)^{2-q}\right)\theta r^q<t_0,
			\end{align*}
			where the last inequality follows from $q>2$ and taking $\mu>1$. Therefore because $x_0\in B_{r}(x_*)$, combining equations \eqref{eq:degremark2} and \eqref{eq:degremark3} we get a contradiction
			\begin{equation*}
			2\mu^2u(x_0,t_0)=u(x_*,t_0-\theta r^q)\leq \mu u(x_0,\cdot,t_0-\theta r^q+cu(x_*,t_0-\theta r^q)^{2-q}r^q)<2\mu^2u(x_0,t_0).
			\end{equation*}
			Thus inequality \eqref{eq:degremark4} cannot hold and we have
			\begin{equation}
			\sup_{B_{r}(x_0)}u(\cdot,t_0-\theta r^q)\leq2\mu^2u(x_0,t_0).
			\end{equation}
			Dividing both sides by $2\mu^2$ and combining this with Theorem \ref{thm:parharnack} gives us
			\begin{equation}
			\label{eq:degrehar2}
			(2\mu^2)^{-1}\sup_{B_{r}(x_0)}u(\cdot,t_0-\theta r^q)\leq u(x_0,t_0)\leq\mu\inf_{B_{r}(x_0)}u(\cdot,t_0+\theta r^q)\leq2\mu^2\inf_{B_{r}(x_0)}u(\cdot,t_0+\theta r^q).
			\end{equation}
			as desired. Because $\mathcal{T}$ has to be either empty or non-empty, combining \eqref{eq:degrehar1} and \eqref{eq:degrehar2} gives us Harnack's inequality for constant $\gamma=2\mu^2$.
			
			Our space requirement $(x_0,t_0)+Q_{6\rho}(\theta)\subset Q_{1}^-(1)$ becomes
			$(x_0,t_0)+Q_{\frac{6}{\alpha}r}(\theta)\subset Q_{1}^-(1)$ for $\alpha=(2\mu)^{\frac{2-q}{q}}<1$ so we have the result for $\sigma=\frac{6}{\alpha}$.
		\end{case}
	\end{proof}
	\section{The covering argument}
	\label{sec:cov}
	The intrinsic Harnack requires a lot of room around the target cylinder if $\mu$ from Theorem \ref{thm:parharnack} happens to be large. The amount of needed room can be reduced by using a covering argument but details about this are hard to find in the literature. In the time-independent case, this can be done easily by covering by small balls but in our case, the intrinsic scaling in the time direction can cause problems with the sets. We apply the covering argument in two steps: We first prove, in Lemma \ref{lem:time covering} below, that we can reduce the needed room as much as we want in the time variable by relaxing our constant in the space variable. Then by a second covering argument, we prove that we can gain back what we lost in the space variable without relaxing the time direction.

	We only consider the forward Harnack's inequality as the presented proof can be directly modified for the backward version. We point out, however, that since the argument iteratively applies Harnack's inequality, it yields different constants $c$ and $\mu$ for the backward and forward versions.
	
	We use right-angled paths connecting two points to deal with space and time variables separately.
	Given $(x,t),(y,s)\in\mathbb{R}^{n+1}$, we denote by $\gamma_{(x,t)}^{(y,s)}:[0,1)\rightarrow\mathbb{R}^{n+1}$
	a path from $(x,t)$ to $(y,s)$ such that
	\[
	\gamma([0,1))=[(x,t),(y,t)]\cup[(y,t),(y,s)).
	\]
	That is, $\gamma_{(x,t)}^{(y,s)}$ first moves from $(x,t)$ to $(y,t)$
	in space, and then from $(y,t)$ to $(y,s)$ in time. 
	For $a\in\mathbb{R}$, we denote by $\left\lceil a\right\rceil \in\mathbb{Z}$
	the number $a$ rounded up to the nearest integer. 

	\begin{lemma}
		\label{lem:time covering} 
		Let $u \geq 0$ be a viscosity solution to \eqref{eq:rgnppar} in $Q_{1}^{-}(1)$ and let the range condition \eqref{eq:range} hold. Fix $\left(x_{0}, t_{0}\right) \in Q_{1}^{-}(1)$ such that $u(x_0,t_0)>0$. Then for any $\sigma_t>1$ there exist $\mu=\mu(n, p, q,\sigma_t)$, $\alpha = \alpha(n, p, q, \sigma_t)$, $c=c(n, p, q, \sigma_t)$ and $\sigma_x=\sigma_x(n,p,q,\sigma_t)$ such that
		\begin{equation*}
		u(x_{0},t_{0})\leq\mu\inf_{B_{\alpha r}(x_{0})}u(\cdot,t_{0}+\theta r^{q})
		\end{equation*}
		whenever
		\begin{equation*}
		(x_{0},t_{0})+B_{\sigma_{x}r}(x_{0})\times(t_{0}-\theta(\sigma_{t}r)^{q},t_{0}+\theta(\sigma_{t}r)^{q})\subset Q_{1}^{-}(1),
		\end{equation*}
		where $\theta:={c}u(x_{0},t_{0})^{2-q}$.
	\end{lemma}
	\begin{proof}
		Let $\tilde{c},\tilde{\mu}$ and $\tilde{\sigma}$ be the constants given by Theorem \ref{thm:gbackharnack}. That
		is, we have
		\begin{equation}
		u(x,t)\leq\tilde{\mu}\inf_{B_{l}(x)}u(\cdot,t+\tilde{c}u(x,t)^{2-q}l^{q}),\label{eq:basic harnack}
		\end{equation}
		whenever
		\begin{align*}
		\left|x\right|+\tilde{\sigma}l&<1,\\
		t\pm\tilde{c}u(x,t)^{2-q}(\tilde{\sigma}l)^{q}&\subset[0,1).
		\end{align*}
		We may assume that $\sigma_{t}<\tilde{\sigma}$ as otherwise the
		claim holds by \eqref{eq:basic harnack}. 
		We denote
		\[
		\kappa := \frac{\tilde \sigma^q - 1}{\sigma_t^q - 1}
		\]
		and set
		\begin{align*}
		\alpha & := \kappa^{-\frac{1}{q}}, \\
		\sigma_x & := \alpha (\tilde \sigma \max(1, \mu ^{\frac{2-q}{q}\lceil\kappa\rceil}) + 1), \\
		c & := \tilde c (\lceil\kappa\rceil + 1) \kappa ^{-1}, \\
		\rho_i & := \left( \mu^{(2-q)i} \kappa^{-1}\right)^{\frac{1}{q}}r = \mu^{\frac{(2-q)i}{q}}\alpha r.
		\end{align*}
		
		Let $(\hat x, \hat t)$ be the target point, that is, $\hat x \in B_{\alpha r}(x_0)$ and
		\[
		\hat t := t_0 + c u(x_0, t_0)^{2-q} r^q.
		\]
		It now suffices to prove that $u(x_0, t_0) \leq \mu u(\hat x, \hat t)$. We proceed by iteration.
		
		\textbf{Initial step:} Set 
		\[
		t_1^\ast := t_0 + \tilde c u(x_0, t_0)^{2-q} \rho_0^q.
		\]
		Then, since $|x-x_0| \leq \alpha r = \rho_0$, we have by Harnack's inequality
		\[
		u(x_0, t_0) \leq \tilde \mu u(\hat x, t_1^\ast).
		\]
		Now, let $(\hat x, t_1)$ be the first point along the path $\gamma_{(\hat x, t_1^\ast)}^{(\hat x, \hat t)}$ such that 
		\[
		u(x_0, t_0) = \tilde \mu u(\hat x, t_1).
		\]
		If no such point exists, then by continuity we must have
		\[
		u(x_0, t_0) \leq \tilde \mu u(\hat x, \hat t)
		\]
		and the claim already holds.
		
		\textbf{Iteration step:} Let $i \in \{2, \ldots\}$ and suppose that we have already chosen $(\hat x, t_{i-1})$ such that
		\[
		u(x_0, t_0) = \tilde \mu ^{i-1} u(\hat x, t_{i-1}).
		\]
		Set
		\[
		t_i^{\ast} := t_{i-1} + \tilde c u(\hat x, t_{i-1})^{2-q}\rho_i^q.
		\]
		If $u(\hat x, t_{i-1}) < \tilde \mu u(\hat x, t_i^{\ast})$, then we move along the path $\gamma_{(\hat x, t_i^\ast)}^{(\hat x, \hat t)}$ until we find a point $(\hat x, t_i)$ such that $u(\hat x, t_{i-1}) = \tilde \mu u(\hat x, t_i)$ so that
		\[
		u(x_0, t_0) = {\tilde \mu}^{i-1} u(\hat x, t_{i-1}) = {\tilde \mu}^i u(\hat x, t_i).
		\]
		If no such point exists, then by continuity we must have
		\[
		u(x_0, t_0) = {\tilde \mu}^{i-1} u(\hat x, t_{i-1}) \leq {\tilde \mu}^i u (\hat x, \hat t),
		\]
		and the claim of the lemma follows.
		
		If the iteration does not end prematurely, we continue until $t_i^\ast \geq \hat t$. When that happens, we apply Harnack's inequality one more time with a radius $\rho \leq \rho_{i-1}$ so that we obtain an estimate at the exact time level $\hat t$. We define $i_{\hat t}$ as the smallest natural number such that $t_{i_{\hat t} + 1}^\ast \geq \hat t$. We have
		\[
		i_{\hat t} \leq \lceil \kappa \rceil
		\]
		as otherwise
		\begin{align*}
		t_{i_{\hat t}} ^\ast \geq t_{\lceil\kappa\rceil + 1}^\ast & = t_{\lceil\kappa\rceil} + \tilde c u(\hat x, t_{\lceil\kappa\rceil})^{2-q}\rho_{\lceil\kappa\rceil}^q \\
		& \geq t_{\lceil\kappa\rceil}^\ast + \tilde c {\tilde \mu}^{(q-2)\lceil\kappa\rceil}u(x_0, t_0)^{2-q}\rho_{\lceil\kappa\rceil}^q \\
		& \geq t_0 + \sum_{i = 0}^{\lceil\kappa\rceil} \tilde c {\tilde \mu}^{(q-2)i} u(x_0, t_0)^{2-q} \rho_i^q \\
		& = t_0 + \tilde c u(x_0, t_0)^{2-q} \sum_{i=0}^{\lceil\kappa\rceil} \kappa^{-1}  r^q \\
		& = t_0 + \tilde c u(x_0, t_0)^{2-q} (\lceil\kappa\rceil + 1) \kappa ^{-1} \\
		& = \hat{t},
		\end{align*}
		which would be against the definition of $i_{\hat t}$. Consequently, the procedure yields the estimate
		\[
		u(x_0, t_0) \leq {\tilde \mu}^{\lceil\kappa\rceil + 1} u(\hat x, \hat t).
		\]
		We still need to verify that there is enough room to apply Harnack's inequality throughout the iteration. To this end, notice that the biggest jump in the time direction that we can do is 
		\[
		J = \tilde c u(x_0, t_0)^{2-q}\kappa^{-1} r^q.
		\]
		Therefore we always have room in time direction, since in the worst case the jump starts from $\hat t - J$, and then we have (using that $\tilde c \leq c$)
		\begin{align*}
		\hat t - J + J{\tilde \sigma}^q & = t_0 + c u(x_0, t_0)^{2-q}r^q + ({\tilde \sigma}^q - 1)\tilde c u(x_0, t_0)^{2-q}\kappa^{-1} r^q \\
		& \leq t_0 + cu(x_0, t_0)^{2-q}r^q + \tilde c (\sigma_t^q-1)u(x_0, t_0)^{2-q}r^q \\
		& \leq t_0 + cu(x_0, t_0)^{2-q}(\sigma_t r)^q < 1.
		\end{align*}
		We also have enough room in space direction since
		\begin{align*}
		\rho_i \tilde \sigma + |\hat x - x_0| & \leq \tilde \sigma \max(\rho_0, \rho_{\lceil\kappa\rceil}) + \alpha r \\
		& = \alpha (\tilde \sigma \max(1, \mu ^{\frac{2-q}{q}\lceil\kappa\rceil}) + 1)r \\
		& = \sigma_x r < 1. \qedhere
		\end{align*}
	\end{proof}
	
	We are now ready to prove the general form of Harnack's inequality. We remark that as the space required around the intrinsic cylinder $(x_0, t_0) + Q_{\sigma r}(\theta) \subset Q_1$ tends to zero (i.e.~when $\sigma \rightarrow 1$), the waiting time coefficient $c$ blows up if $q>2$, and tends to zero if $q<2$.
	
	\begin{theorem}\label{thm:harnack_general_form}
		Let $u \geq 0$ be a viscosity solution to \eqref{eq:rgnppar} in $Q_{1}^{-}(1)$ and let the range condition \eqref{eq:range} hold. Fix $\left(x_{0}, t_{0}\right) \in Q_{1}^{-}(1)$ such that $u(x_0,t_0)>0$. Then for any $\sigma>1$ there exist $\mu=\mu(n, p, q,\sigma)$ and ${c}={c}(n, p, q, \sigma)$ such that
		\begin{equation}
		u(x_{0},t_{0})\leq\mu\inf_{B_{r}(x_{0})}u(\cdot,t_{0}+cu(x_{0},t_{0})^{2-q}r^{q}),\label{eq:claim}
		\end{equation}
		whenever 
		\begin{equation*}
		(x_{0},t_{0})+B_{\sigma r}(x_{0})\times(t_{0}-\theta(\sigma r)^{q},t_{0}+\theta(\sigma r)^{q})\subset Q_{1}^{-}(1),
		\end{equation*}
		where $\theta:=cu(x_{0},t_{0})^{2-q}$.
	\end{theorem}
	\begin{proof}
		Let $\tilde{c},\sigma_{x}, \alpha$ and $\tilde \mu$ be the constants that we get from Lemma \ref{lem:time covering} for $\sigma_{t}:=\sigma$.
		Then we have
		\begin{equation}
		u(x,t)\leq\tilde{\mu}\inf_{B_{\alpha l}(x)}u(\cdot,t+\tilde{c}u(x,t)^{2-q}l^{q})\label{eq:initial harnack}
		\end{equation}
		whenever
		\begin{equation*}
		(x,t)+B_{\sigma_{x}l}(x)\times(t-\tilde{c}u(x,t)^{2-q}(\sigma l)^{q},t+\tilde{c}u(x,t)^{2-q}(\sigma l)^{q})\subset Q_{1}^{-}(1),
		\end{equation*}
		i.e.
		\begin{align*}
		\left|x\right|+\sigma_{x}l<1,\\
		t\pm\tilde{c}u(x,t)^{2-q}(\sigma l)^{q}\subset[0,1).
		\end{align*}
		We denote
		\[
		\varrho:=\frac{\sigma-1}{\sigma_{x}}
		\]
		and
		\[
		c:=\tilde{c}\varrho^{q}\sum_{k=1}^{\left\lceil {(\alpha\varrho)}^{-1}\right\rceil + 1}\tilde{\mu}^{(q-2)(k-1)}.
		\]
		Let $(\hat{x},\hat{t})$ be the target point, that is, $\hat{x}\in B_{r}(x_{0})$
		and 
		\[
		\hat{t}:=t_{0}+c u(x_{0},t_{0})^{2-q}r^{q}.
		\]
		We now proceed by iteration (see Figure \ref{fig:harnack_chain_spaceq<2} below).
		\begin{figure}[h]\label{fig:space_iteration}
			\begin{tikzpicture}[scale=0.8]
			\filldraw[black] (0, 0) circle (1pt); %(x_0, t_0)
			\filldraw[black] (2, 2+2) circle (1pt);
			\filldraw[black] (3.5, 2+2) circle (1pt); %(x_1, t_1)
			\filldraw[black] (5.5, 3+3) circle (1pt);
			\filldraw[black] (7, 5+3) circle (1pt); %(x_2, t_2)
			\filldraw[black] (7, 6+3) circle (1pt);
			\draw[black, line width = 0.25mm] plot[smooth,domain=0:2] (\x, {1 * \x*\x});
			\draw[lightgray, line width = 0.25mm] plot[smooth,domain=-2:0] (\x, {1 * \x*\x});
			
			\draw[black, dotted, line width = 0.25mm] (2, 2+2) -- (3.5, 2+2);
			\draw[black, line width = 0.25mm] plot[smooth,domain=3.5:5.5] (\x, {0.5 * (\x-3.5)*(\x-3.5) + 2+2});
			\draw[lightgray, line width = 0.25mm] plot[smooth,domain=1.5:3.5] (\x, {0.5 * (\x-3.5)*(\x-3.5) + 2+2});
			
			\draw[black, dotted, line width = 0.25mm] (5.5, 3+3) -- (7, 3+3);
			\draw[black, dotted, line width = 0.25mm] (7, 3+3) -- (7, 5+3);
			\draw[black, line width = 0.25mm] plot[smooth,domain=7:9] (\x, {0.25 * (\x-7)*(\x-7) + 5+3});
			\draw[lightgray, line width = 0.25mm] plot[smooth,domain=5:7] (\x, {0.25 * (\x-7)*(\x-7) + 5+3});
			
			\node at (0, 0.5) {$(x_0, t_0)$};
			\node at (2, 2.5+2) {$(x_1^\ast, t_1^\ast)$};
			\node at (3.5, 2.5+2) {$(x_1, t_1)$};
			\node at (5.5, 3.5+3) {$(x_2^\ast, t_2^\ast)$};
			\node at (7, 5.5+3) {$(x_2, t_2)$};
			\node at (7, 6.5+3) {$(\hat x, \hat t)$};

			\draw[black, line width = 0.25mm] (4, -0.25) -- (4, 0);
			\node at (4, -0.5) {$x_0 + \sigma_x \rho$};

			\draw[black, line width = 0.25mm] (11, -0.25) -- (11, 0);
			\node at (11, -0.5) {$\hat x + \sigma_x \rho$};
			
			\draw[gray, line width = 0.25mm] (0, 0) -- (1.5*8, 0);
			\draw[black, line width = 0.25mm] (1.5*8, 0.25) -- (1.5*8, 0);
			\node at (1.5*8, 0.5) {$x_0 + \sigma r$};
			
			%\draw[black, line width = 0.25mm] (2*7, 0.25) -- (2*7, 0);

			%Tässä \sigma_x = 2, \sigma = 1.5, r=0.8, dist(x, \hat_x) = 0.7
			
			\end{tikzpicture}
			\caption{Illustration of the Harnack chain in the proof of Theorem 
				\ref{thm:harnack_general_form} when $q<2$. If $q>2$, the paraboloids get steeper instead.}
			\label{fig:harnack_chain_spaceq<2}
		\end{figure}
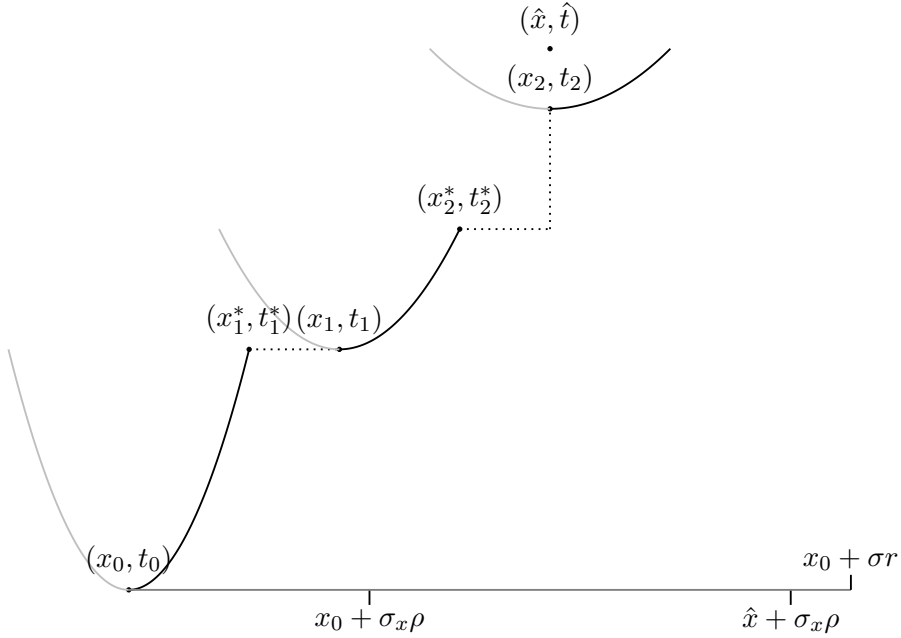
		
		\textbf{Initial step:} Let $\rho:=\varrho r$ and set 
		\[
		t_{1}^{\ast}:=t_{0}+\tilde{c}u(x_{0},t_{0})^{2-q}\rho{}^{q}.
		\]
		By the Harnack's inequality in (\ref{eq:initial harnack}) we have 
		\[
		u(x_{0},t_{0})\leq\tilde{\mu}u(x_{1}^{\ast},t_{1}^{\ast}),
		\]
		where $x_{1}^{\ast}$ is the point in $\overline{B}_{\alpha\rho}(x_{0})$
		that is closest to $\hat{x}$. Now, if $u(x_{0},t_{0})<\tilde{\mu}u(x_{1}^{\ast},t_{1}^{\ast})$,
		then we move along the path $\gamma_{(x_{1}^{\ast},t_{1}^{\ast})}^{(\hat{x},\hat{t})}$
		until we find a point $(x_{1},t_{1})$ such that 
		\[
		u(x_{0},t_{0})=\tilde{\mu}u(x_{1},t_{1}).
		\]
		If no such point exists, then by continuity we must have
		\[
		u(x_{0},t_{0})\leq\tilde{\mu}u(\hat{x},\hat{t}),
		\]
		and we end the iteration.
		
		\textbf{Iteration step:} Let $i\in\left\{ 2,\ldots\right\} $ and
		suppose that we have already chosen $(x_{i-1},t_{i-1})$ such that
		\[
		u(x_{0},t_{0})=\tilde{\mu}^{i-1}u(x_{i-1},t_{i-1}).
		\]
		Set 
		\[
		t_{i}^{\ast}:=t_{i-1}+\tilde{c}u(x_{i-1},t_{i-1})^{2-q}\rho^{q}.
		\]
		By Harnack's inequality in (\ref{eq:initial harnack}) we have
		\[
		u(x_{i-1},t_{i-1})\leq\tilde{\mu}u(x_{i}^{\ast},t_{i}^{\ast}),
		\]
		where $x_{i}^{\ast}$ is the point in $\overline{B}_{\rho}(x_{i-1})$
		that is closest to $\hat{x}$. Now, if $u(x_{i-1},t_{i-1})<\tilde{\mu}u(x_{i}^{\ast},t_{i}^{\ast})$,
		then we move along the path $\gamma_{(x_{i}^{\ast},t_{i}^{\ast})}^{(\hat{x},\hat{t})}$
		until we find a point $(x_{i},t_{i})$ such that $u(x_{i-1},t_{i-1})=\tilde{\mu}u(x_{i},t_{i})$
		so that
		\[
		u(x_{0},t_{0})=\tilde{\mu}^{i-1}u(x_{i-1},t_{i-1})=\tilde{\mu}^{i}u(x_{i},t_{i}).
		\]
		If no such point exists, then by continuity we must have
		\[
		u(x_{0},t_{0})=\tilde{\mu}^{i-1}u(x_{i-1},t_{i-1})\leq\tilde{\mu}^{i}u(\hat{x},\hat{t}),
		\]
		and we end the iteration.
		
		If the iteration does not end prematurely, we continue until $t_{i}^{\ast}\geq\hat{t}$.
		When that happens, we apply the Harnack's inequality \eqref{eq:initial harnack}  one more time with a radius smaller or equal to $\rho$ so that we hit $\hat{t}$. We define $i_{\hat t}$ as the smallest natural number such that $t^\ast_{i_{\hat{t}}+1} \geq \hat{t}$. That is, $t_{\hat{i}}$ is the time from which it remains to apply the Harnack's inequality \eqref{eq:initial harnack} one more time to reach the target time $\hat{t}$. Next, we show that our selections of the constants ensure the finiteness of  $i_{\hat t}$ and that $x_{i_{\hat t}} = \hat x $. For finiteness, we observe that
		\begin{equation*}
		i_{\hat{t}} \leq \lceil (\alpha \varrho)^{-1} \rceil .
		\end{equation*}
		Indeed, otherwise
		\begin{align*}
		t_{i_{\hat{t}}}^{\ast}\geq t_{\left\lceil (\varrho\alpha)^{-1}\right\rceil +1}^{\ast} & \geq t_{0}+\sum_{k=1}^{\left\lceil (\varrho\alpha)^{-1}\right\rceil +1}\tilde{c}u(x_{k-1},t_{k-1})^{2-q}\rho^{q}\\
		& =t_{0}+\sum_{k=1}^{\left\lceil (\varrho\alpha)^{-1}\right\rceil +1}\tilde{c}\frac{1}{\tilde{\mu}^{(k-1)(2-q)}}u(x_{0},t_{0})^{2-q}\rho^{q}\\
		& =t_{0}+\tilde{c}u(x_{0},t_{0})^{2-q}\rho^{q}\sum_{k=1}^{\left\lceil (\varrho\alpha)^{-1}\right\rceil +1}\tilde{\mu}^{(q-2)(k-1)}\\
		& =t_{0}+cu(x_{0},t_{0})^{2-q}r^{q}\\
		& =\hat{t}, 
		\end{align*}
		which would be against the definition of $i_{\hat t}$.
		Next we estimate the smallest $i_{\hat{x}}\in\left\{ 1,\ldots\right\} $
		such that $x_{i_{\hat{x}}}=\hat{x}$ (observe that the construction
		ensures that $x_{i}=\hat{x}$ for all $i\geq i_{\hat{x}}$). At each
		iteration step, unless we have already reached $\hat{x},$ we move
		at least $\alpha \rho$ closer towards $\hat{x}$. Since $\left|x_{0}-\hat{x}\right|\leq r$, we thus have
		\[
		i_{\hat x} \leq \left\lceil \frac{r}{\alpha\rho} \right\rceil = \left\lceil \frac{r}{\alpha \varrho r} \right\rceil = \left\lceil (\alpha \varrho)^{-1} \right\rceil.
		\]
		We want to show that $i_{\hat x} \leq i_{\hat t}$, as this implies that $x_{i_{\hat t}} = \hat x$. For this end, we may assume that the Harnack chain does not skip in time direction using the paths $\gamma_{(x_i, t_i)}^{(\hat x, \hat t)}$, as otherwise the chain automatically reaches $\hat x$. Using this we conclude that $i_{\hat x} \leq i_{\hat t}$ must hold since otherwise
		\begin{align*}
		t_{i_{\hat t} + 1}^\ast\leq t_{i_{\hat x}}^\ast & \leq t_{\lceil (\alpha \varrho)^{-1} \rceil}^\ast \\
		& = t_0 + \tilde c u(x_0,t_0)^{2-q}\rho^q \sum_{k=1}^{\lceil(\rho \alpha)^{-1} \rceil} {\tilde \mu}^{(q-2)(k-1)} \\
		& < t_0 + \tilde c u(x_0,t_0)^{2-q}\rho^q \sum_{k=1}^{\lceil(\rho \alpha)^{-1} \rceil + 1} {\tilde \mu}^{(q-2)(k-1)} \\
		& = \hat t,
		\end{align*}
		which is against the definition of $i_{\hat t}$. Thus the procedure reaches $\hat{x}$ before we apply Harnack's inequality one last time. This yields the estimate
		\[
		u(x_{0},t_{0})\leq\tilde{\mu}^{\left\lceil {(\alpha \varrho)}^{-1}\right\rceil +1}u(\hat{x},\hat{t}).
		\]
		
		We still need to check that we have room to use Harnack's inequality.
		The room in space is clear from the definition of $\rho$ since
		\[
		\left|\hat{x}-x_{0}\right|+\sigma_{x}\rho=r+\sigma_{x}\frac{\sigma-1}{\sigma_{x}}r\leq r+(\sigma-1)r=\sigma r<1.
		\]
		For the room in time, observe that we always end the Harnack at most the time level $\hat{t}$. Therefore, the worst-case scenario would be if our biggest possible jump in Harnack's inequality ended up at $\hat{t}$. Since the sequence $u(x_{i},t_{i})$ is decreasing, the biggest possible jump is 
		\[
		J:=\begin{cases}
		\tilde{c}u(x_{0},t_{0})^{2-q}\rho^{q}, & \text{if }q<2,\\
		\tilde{c}u(x_{0},t_{\left\lceil (\alpha\varrho)^{-1}\right\rceil})^{2-q}\rho^{q}, & \text{if }q>2.
		\end{cases}
		\]
		To land on $\hat{t}$, the jump would have to start from $\hat{t}-J$. Thus it suffices to ensure that
		\[
		(\hat{t}-J)+J\sigma^{q}=\hat{t}+(\sigma^{q}-1)J<1.
		\]
		This holds, since if $q < 2$, we have
		\begin{align*}
		\hat{t}+(\sigma^{q}-1)J & =\hat{t}+(\sigma^{q}-1)\tilde{c}u(x_{0},t_{0})^{2-q}\rho^{q}\\
		& =t_{0}+cu(x_{0},t_{0})^{2-q}r^{q}+(\sigma^{q}-1)cu(x_{0},t_{0})^{2-q}(\frac{\tilde{c}\varrho^{q}}{c})\\
		& \leq t_{0}+cu(x_{0},t_{0})^{2-q}(\sigma\rho)^{q}<1,
		\end{align*}
		and if $q>2$, we have 
		\begin{align*}
		\hat{t}+(\sigma^{q}-1)J & =\hat{t}+\tilde{c}u(x_{0},t_{k})^{2-q}\rho^{q}(\sigma^{q}-1)\\
		& =t_{0}+cu(x_{0},t_{0})^{2-q}r^{q}+cu(x_{0},t_{0})^{2-q}r^{q}(\sigma^{q}-1)\left(\frac{\tilde{c}\varrho^{q}\mu^{(q-2)(\left\lceil (\alpha\varrho)^{-1}\right\rceil)}}{c}\right)\\
		& \leq t_{0}+cu(x_{0},t_{0})^{2-q}(\sigma r)^{q}<1. \qedhere
		\end{align*}
	\end{proof}
	\section{Optimality of the range of exponents}
	\label{sec:range}
	Intrinsic Harnack's inequality may fail outside of the range condition \eqref{eq:range} as for such exponents, viscosity solutions of \eqref{eq:rgnppar} vanish in finite time as we will prove in this section. The solutions of the standard $p$-parabolic equation in the corresponding subcritical exponent range behave in a similar way. Idea, behind the proof is to use the equivalence result proven by Parviainen and Vázquez \cite{Parviainen2020} to transfer the problem onto a one-dimensional divergence form equation and then to prove that a solution to this equation vanishes. We use the weak formulation for a time-mollified solution with a suitable test function after first proving that this formulation holds for all weak solutions as the separate lemma. Next, we simplify both sides of the formulation, estimate using Sobolev's inequality and ultimately get a vanishing upper bound for the norm of the solution. We do this first in bounded domains and then prove the global result using convergence and stability results. This global result Proposition \ref{prop:finiteextspace} gives us a counterexample to the intrinsic Harnack's inequality \ref{thm:gbackharnack} and thus proves that range \eqref{eq:range} is optimal.
	
	As proven by Parviainen and Vázquez, radial viscosity solutions to \eqref{eq:rgnppar} are equivalent to weak solutions of the one-dimensional equation
	\begin{equation}
	\partial_{t}u-\frac{p-1}{q-1}\Delta_{q,d}u=0\quad\text{in }(-R,R)\times(0,T).\label{eq:radial eq}
	\end{equation}
	Here, denoting by $u'$ the radial derivative of $u$,
	\[
	\Delta_{q,d}u:=\left|u'\right|^{q-2}\left((q-1)u''+\frac{d-1}{r}u'\right) 
	\]
	is heuristically the usual radial $q$-Laplacian in a fictitious dimension
	\begin{equation*}
	d:=\frac{(n-1)(q-1)}{p-1}+1.
	\end{equation*} If $d$ happens to be an integer, then solutions
	to \eqref{eq:radial eq} are equivalent to radial weak solutions of
	the $q$-parabolic equation in $B_R\times(0,T)\subset\mathbb{R}^{d+1}$ by \cite[Section 3]{Parviainen2020}. If $d\not\in\N$, we still have an equivalence between radial viscosity solutions of \eqref{eq:rgnppar} and continuous weak solutions of \eqref{eq:radial eq} as proven in \cite[Theorem 4.2]{Parviainen2020}. 
	
	A weak solution of \eqref{eq:radial eq} is in a weighted Sobolev space but we are only interested in continuous solutions and thus will assume this in the following definition. The description of the exact definition in the elliptic case is in \cite[Definition 2.2]{Siltakoski2021}. The following definition is written in a slightly different form but is equivalent to the definition given by Parviainen and Vazquez \cite[Definition 4.1]{Parviainen2020}. We use the notation $\ud z:=r^{d-1}\ud r\ud t$ for the natural parabolic measure for this problem and denote the distributional derivative of $v$ by $v'$ and define it by
	\begin{equation*}
	\int_{0}^{R}v'\vp\ud r=-\int_{0}^{R}v\vp'\ud r
	\end{equation*}
	for all $\vp\in C_0^\infty((0,R))$ so it coincides with standard derivative for differentiable functions.
		\begin{definition}
			\label{def:1dimn}
			Let $0<T\leq\infty$ and $0<R\leq\infty$. A function $u\in C\left((-R,R)\times(0,T)\right)$ such that $u'\in C\left((-R,R)\times(0,T)\right)$ and $u'(0,t)=0$ is a continuous weak solution to \eqref{eq:radial eq} if we have
			\begin{equation*}
			\int_{t_1}^{t_2}\int_{0}^{R}u\partial_{t}\phi-\frac{p-1}{q-1}\left(\abs{u'}^{q-2}u'\right)\phi'\ud z=0	
			\end{equation*}
			for all $0 < t_1 < t_2 < T$ and $\phi\in C^{\infty}_0((-R,R)\times(0,T))$.
		\end{definition}
		We define time-mollification and prove a basic result for it in Lemma \ref{le:molli} below for the convenience of the reader. Let $\eps>0$ and $\eta_\eps : \mathbb {R} \rightarrow [0, \infty)$ be the standard mollifier such that $\supp \eta_\varepsilon \subset (-\varepsilon, \varepsilon)$. The time-mollification of $u \in L^1((0,R)\times(0,T))$ is defined by
		\begin{equation}
		\label{eq:mollifier}
		u_\eps(r,t):=\eta_\eps*u(r,t)=\int_{0}^{T}\eta_\eps(t-s)u(r,s) \ud s.
		\end{equation}
		\begin{lemma}
			\label{le:molli}
			Let $u$ be a continuous weak solution to \eqref{eq:radial eq} and let $u_\eps$ denote the mollification \eqref{eq:mollifier}. Then for all $0 < t_1 < t_2 < T$ and $\phi\in C^\infty((0,R)\times(0,T))$ such that $\supp\phi(\cdot,t)\Subset(-R,R)$ for any $t\in(t_1,t_2)$, we have
			\begin{equation*}
			\int_{t_1}^{t_2}\int_{0}^{R}u_\eps\partial_{t}\phi-\frac{p-1}{q-1}\left(\abs{u'}^{q-2}u'\right)_\eps\phi'\ud z=\int_{0}^{R}\left(u_\eps(r,t_2)\phi(r,t_2)-u_\eps(r,t_1)\phi(r,t_1)\right)r^{d-1}\ud r	
			\end{equation*}
		\end{lemma}
		\begin{proof}
			Let first $\varepsilon>0$ and $\varphi\in C_{0}^{\infty}((-R,R)\times(0,T))$
			be such that $\supp\varphi\Subset(-R,R)\times(\varepsilon,T-\varepsilon)$.
			Because $\eta_{\varepsilon}$ is even, we have by partial integration
			(the boundary terms vanish since $\varphi(0,\cdot)\equiv\varphi(T,\cdot)\equiv0$)
			\begin{align*}
			\partial_{t}\varphi_{\varepsilon}(t,r)=\partial_{t}\int_{0}^{T}\eta_{\varepsilon}(t-s)\varphi(r,s)\ud s & =\int_{0}^{T}-\partial_{t}\eta_{\varepsilon}(t-s)\varphi(r,s)\ud s\\
			& =\int_{0}^{T}\eta_{\varepsilon}(t-s)\partial_{s}\varphi(r,s)\ud s.
			\end{align*}
			Thus by Fubini's theorem
			\begin{align*}
			\int_{0}^{T}\int_{0}^{R}u(r,t)\partial_{t}\varphi_{\varepsilon}(r,t)r^{d-1}\ud r\ud t & =\int_{0}^{T}\int_{0}^{R}\int_{0}^{T}u(r,t)\eta_{\varepsilon}(t-s)\partial_{s}\varphi(r,s)r^{d-1}\ud s\ud r\ud t\\
			& =\int_{0}^{T}\int_{0}^{R}\int_{0}^{T}u(r,t)\eta_{\varepsilon}(t-s)\ud t\,\partial_{s}\varphi(r,s)r^{d-1}\ud r\ud s\\
			& =\int_{0}^{T}\int_{0}^{R}u_{\varepsilon}(r,s)\partial_{s}\varphi(r,s)r^{d-1}\ud r\ud s.
			\end{align*}
			Similarly, for the space derivative we have
			\begin{align*}
			& \int_{0}^{T}\int_{0}^{R}\left|u^{\prime}(r,t)\right|^{q-2}u^{\prime}(r,t)\varphi_{\varepsilon}^{\prime}(r,t)r^{d-1}\ud r\ud t\\
			& \ \ =\int_{0}^{T}\int_{0}^{R}\int_{0}^{T}\eta_{\varepsilon}(t-s)\partial_{r}\varphi(r,s)\left|u^{\prime}(r,t)\right|^{q-2}u^{\prime}(r,t)\ud sr^{d-1}\ud r\ud t\\
			& \ \ =\int_{0}^{T}\int_{0}^{R}\int_{0}^{T}\eta_{\varepsilon}(t-s)\left|u^{\prime}(r,t)\right|^{q-2}u^{\prime}(r,t)\ud t\partial_{r}\varphi(r,s)r^{d-1}\ud r\ud s\\
			& \ \ =\int_{0}^{T}\int_{0}^{R}(\left|u^{\prime}\right|^{q-2}u^{\prime})_{\varepsilon}(r,s)\varphi^{\prime}(r,s)r^{d-1}\ud r\ud s.
			\end{align*}
			By the last two displays, we obtain
			\begin{equation}\label{eq:lala}
			\int_{0}^{T}\int_{0}^{R}u_{\varepsilon}\partial_{t}\varphi-\frac{p-1}{q-1}(\left|u^{\prime}\right|^{q-2}u^{\prime})_{\varepsilon}\varphi^{\prime}\ud z=0.
			\end{equation}
			Let now $0<t_{1}<t_{2}<T$ and $\phi\in C^{\infty}((-R,R)\times(0,T))$
			be such that $\phi(\cdot,t)\Subset(-R,R)$ for any $t\in(t_{1},t_{2})$.
			Define the cut-off function
			\[
			\xi_{h}(t):=\begin{cases}
			0, & t\in(0,t_{1}-h)\\
			\frac{1}{h}(t-t_{1}), & t\in[t_{1}-h,t_{1}),\\
			1, & t\in[t_{1},t_{2}),\\
			1-\frac{1}{h}(t-t_{2}), & t\in[t_{2},t_{2}+h),\\
			0, & t\in[t_{2}+h,T).
			\end{cases}
			\]
			Since $\varphi_{h}:=\xi_{h}\phi$ is Lipschitz, it satisfies \eqref{eq:lala}
			by the first part of the proof and a simple approximation argument. Since by
			continuity all $t_{1},t_{2}\in(0,T)$
			satisfy
			\[
			\int_{0}^{T}\int_{0}^{R}u_{\varepsilon}\partial_{t}\varphi_{h}\ud z\rightarrow\int_{t_{1}}^{t_{2}}\int_{0}^{R}u_{\varepsilon}\partial_{t}\phi\ud z+\int_{0}^{R}(u_{\varepsilon}(r,t_{2})\phi(r,t_{2})-u_{\varepsilon}(r,t_{1})\phi(r,t_{1}))r^{d-1}\ud r
			\]
			as $h\rightarrow0$, the claim of the lemma follows. 
		\end{proof}
		Our proof of finite extinction uses the following Sobolev's inequality, which is heuristically speaking the Gagliardo-Nirenberg inequality for radial functions in the fictitious dimension $d$. The standard formulation of the Gagliardo-Nirenberg inequality requires $q<n$ and hence does not work for our one-dimensional case.
	\begin{theorem}[Radial Sobolev's inequality]
		\label{thm:sobolev}
		Suppose that $1\leq q<d$. Let $v\in C^{\infty}(0,\infty)\cap C[0,\infty)$
		be such that $v(r)\equiv0$ for all large $r>0$. Then there exists $C=C(d,q)$ such that
		\[
		\left(\int_{0}^{\infty}\left|v(r)\right|^{\frac{dq}{d-q}}r^{d-1}\ud r\right)^{\frac{d-q}{dq}}\leq C\left(\int_{0}^{\infty}\left|v'(r)\right|^{q}r^{d-1}\ud r\right)^{\frac{1}{q}}.
		\]
	\end{theorem}
	
	\begin{proof}
		Suppose first that $q=1$. We denote
		\[
		g(r):=\left|v(r)\right|^{\frac{d}{d-1}}=\left|\int_{r}^{\infty}v^{\prime}(s)\ud s\right|^{\frac{d}{d-1}}.
		\]
		Since $d/(d-1)>1$, we have $g\in C^{1}(\mathbb{R})$ and 
		\begin{align*}
		g^{\prime}(r)= & \frac{d}{d-1}\left|v(r)\right|^{\frac{1}{d-1}}v^{\prime}(r)\sgn(v(r))\\
		= & \frac{d}{d-1}\left|\int_{r}^{\infty}v^{\prime}(s)\ud s\right|^{\frac{1}{d-1}}v^{\prime}(r)\sgn(v(r)).
		\end{align*}
		Integrating by parts and using that $g(r)=0$ for large $r$, we obtain
		\begin{align*}
		\int_{0}^{\infty}\left|v(r)\right|^{\frac{d}{d-1}}r^{d-1}\ud r & =\lim_{k\rightarrow\infty}\int_{0}^{k}g(r)r^{d-1}\ud r\\
		& =\lim_{k\rightarrow\infty}\left(g(r)\frac{r^{d}}{d}\Big|_{r=0}^{r=k}-\int_{0}^{k}g^{\prime}(r)\frac{r^{d}}{d}\right)\\
		& =\frac{1}{d}\int_{0}^{\infty}g^{\prime}(r)r^{d}\ud r\\
		& =\frac{1}{d-1}\int_{0}^{\infty}\left|\int_{r}^{\infty}v^{\prime}(s)\ud s\right|^{\frac{1}{d-1}}v^{\prime}(r)\sgn(v(r))r^{d}\ud r\\
		& \leq\frac{1}{d-1}\int_{0}^{\infty}\left(\int_{r}^{\infty}\left|v^{\prime}(s)\right|\ud s\right)^{\frac{1}{d-1}}r\cdot\left|v^{\prime}(r)\right|r^{d-1}\ud r.
		\end{align*}
		This we can further estimate as
		\begin{align*}									\frac{1}{d-1}\int_{0}^{\infty}&\left(\int_{r}^{\infty}\left|v^{\prime}(s)\right|s^{d-1}\underset{\leq1}{\underbrace{\frac{r^{d-1}}{s^{d-1}}}}\ud s\right)^{\frac{1}{d-1}}\left|v^{\prime}(r)\right|r^{d-1}\ud r\\&\leq\frac{1}{d-1}\int_{0}^{\infty}\left(\int_{0}^{\infty}\left|v^{\prime}(s)\right|s^{d-1}\ud s\right)^{\frac{1}{d-1}}\left|v^{\prime}(r)\right|r^{d-1}\ud r\\
		& =\frac{1}{d-1}\left(\int_{0}^{\infty}\left|v^{\prime}(s)\right|s^{d-1}\ud s\right)^{\frac{1}{d-1}}\int_{0}^{\infty}\left|v^{\prime}(r)\right|r^{d-1}\ud r\\
		& =\frac{1}{d-1}\left(\int_{0}^{\infty}\left|v^{\prime}(r)\right|r^{d-1}\ud r\right)^{\frac{d}{d-1}}
		\end{align*}
		so that
		\begin{equation}
		\label{eq:radialsobolev1}
		\left(\int_{0}^{\infty}\left|v(r)\right|^{\frac{d}{d-1}}r^{d-1}\ud r\right)^{\frac{d-1}{d}}\leq C\int_{0}^{\infty}\left|v^{\prime}(r)\right|r^{d-1}\ud r.
		\end{equation}
		Suppose then that $1<q<d$. Using \eqref{eq:radialsobolev1} with
		$v:=u^{\frac{dq-q}{d-q}},$ we obtain
		\begin{align*}
		\left(\int_{0}^{\infty}\left|u(r)\right|^{\frac{dq}{d-q}}r^{d-1}\ud r\right)^{\frac{d-1}{d}} & \leq C\int_{0}^{\infty}\left|u^{\prime}(r)\right|\left|u(r)\right|^{\frac{d(q-1)}{d-q}}r^{d-1}\ud r\\
		& \leq C\left(\int_{0}^{\infty}\left|u^{\prime}(r)\right|^{q}r^{d-1}\ud r\right)^{\frac{1}{q}}\left(\int_{0}^{\infty}\left|u(r)\right|^{\frac{dq}{d-q}}r^{d-1}\ud r\right)^{\frac{q-1}{q}},
		\end{align*}
		which implies the desired inequality.
	\end{proof}

	Now we have the needed tools to state and prove the finite extinction of solutions. We do this by first proving the result for solutions of a Dirichlet problem in simple cylinders and then expanding this result to the entire space by convergence results. The existence of global solutions with extinction in finite time is a counterexample for the intrinsic Harnack's inequality as we show at the end of this section.
	The proof uses the following notation for the weighted Lebesgue norm
	\begin{equation}
	\label{eq:Lqnorm}
	\norm{v}_{L^q(r^{d-1},(0,R))}:=\left(\int_{0}^{R}\abs{v}^qr^{d-1}\ud r\right)^{\frac{1}{q}}.
	\end{equation}
	
	We only consider radially symmetric initial data in what follows. The finite extinction holds in the general situation by comparison principle.
	
	\begin{proposition}
		\label{prop:finiteextball}
		Assume $q$ does not satisfy the range condition \eqref{eq:range} and let $R>0$. Let $u$ be a viscosity solution of
		\begin{equation}
		\label{eq:dirichlet}
		\begin{cases}
		\partial_t u=\abs{\nabla u}^{q-p}\div\left(\abs{\nabla u}^{p-2}\nabla u\right) & \text{ in }B_R\times(0,T),\\
		u(\cdot,0)=u_0(\cdot)\geq0& \text{ where }u_0\in L^\infty(B_R)\cap C(B_R) \text{ is radial},\\
		u(\cdot,t)=0 &\text{ on }\partial B_R \text{ for any }t\in(0,T).
		\end{cases}
		\end{equation}
		There exists a finite time $T^*:=T^*(n,p,q,u_0)$, such that
		\begin{equation*}
		u(\cdot,t)\equiv 0 \quad \text{ for all }t\geq T^*
		\end{equation*}
		and
		\begin{equation*}
		0<T^*\leq
		C\norm{u_0}_{L^s(r^{d-1},(0,R))}^{2-q}
		\end{equation*}
		where $C:=C(n,p,q)$ and $s=\frac{d(2-q)}{q}$.
	\end{proposition}
	\begin{proof}
		The existence of a solution $u \in C( \overbar B_R\times[0,T])$ to the Cauchy-Dirichlet problem \eqref{eq:dirichlet} can be proven for example by modified Perron's method (see \cite[Theorem 2.6]{Parviainen2020}) and the comparison principle ensures that it is radial.  Therefore, by the equivalence result \cite[Theorem 4.2]{Parviainen2020}, $u$ is a continuous weak solution to
		\begin{equation}
		\label{eq:onedim}
		\begin{cases}
		\partial_{t}u-\frac{p-1}{q-1}\left|u'\right|^{q-2}\left((q-1)u''+\frac{d-1}{r}u'\right) =0&\text{ in }(-R,R)\times(0,T),\\
		u(\cdot,0)=u_0(\cdot)\geq0&\text{ where }u_0\in L^\infty((-R,R)),\\
		u(-R,t)=u(R,t)=0 & \text{ for any }t\in(0,T).
		\end{cases}
		\end{equation}
		Let \begin{equation*}
		s=\frac{d(2-q)}{q}
		\end{equation*}
		and notice that $s>1$ because we assumed $q<\frac{2d}{d+1}$. We define the test function
		$ \varphi:=u_{\varepsilon,h}^{s-1}-h^{s-1}$, where $u_{\varepsilon,h}:=u_{\varepsilon}+h$ for  $\varepsilon,h>0$ and $u_{\varepsilon}$ denotes the time-mollification. We add this $h$ to ensure that our function remains strictly positive as we have negative exponents during the calculation.
		Then $\varphi$ is an admissible test function and by Lemma \ref{le:molli} we have
		\begin{align} 
		\int_{t_{1}}^{t_{2}}\int_{0}^{R}u_{\varepsilon}\partial_{t}\varphi\ud z-\int_{0}^{R}(u_{\varepsilon}\varphi(r,t_{2})-u_{\varepsilon}\varphi(r,t_{1}))r^{d-1}\ud r  & =  \frac{p-1}{q-1}\int_{t_{1}}^{t_{2}}\int_{0}^{R}(\left|u^{\prime}\right|^{q-2}u^{\prime})_{\varepsilon}\partial_{r}(u_{\varepsilon,h}^{s-1})\ud z\\
		& =:A_{\varepsilon,h} 	\label{eq:aaa} 
		\end{align}
		for all $0 < t_1 < t_2 < T$.
		We rewrite the first term on the left-hand side using integration
		by parts and Fubini's theorem
		\begin{align*}
		\int_{t_{1}}^{t_{2}}\int_{0}^{R}u_{\varepsilon}\partial_{t}\varphi\ud z & =\int_{t_{1}}^{t_{2}}\int_{0}^{R}u_{\varepsilon}\partial_{t}u_{\varepsilon,h}^{s-1}\ud z\\
		& =-\int_{t_{1}}^{t_{2}}\int_{0}^{R}u_{\varepsilon,h}^{s-1}\partial_{t}u_{\varepsilon}\ud z+\int_{0}^{R}(u_{\varepsilon}u_{\varepsilon,h}^{s-1}(r,t_{2})-u_{\varepsilon}u_{\varepsilon,h}^{s-1}(r,t_{1}))r^{d-1}\ud r\\
		& =-\int_{t_{1}}^{t_{2}}\int_{0}^{R}\frac{1}{s}\partial_{t}u_{\varepsilon,h}^{s}\ud z+\int_{0}^{R}(u_{\varepsilon}u_{\varepsilon,h}^{s-1}(r,t_{2})-u_{\varepsilon}u_{\varepsilon,h}^{s-1}(r,t_{1}))r^{d-1}\ud r\\
		& =-\frac{1}{s}\int_{t_{1}}^{t_{2}}\partial_{t}\left\Vert u_{\varepsilon,h}\right\Vert _{L^{s}(r^{d-1},(0,R))}^{s}\ud t+\int_{0}^{R}(u_{\varepsilon}u_{\varepsilon,h}^{s-1}(r,t_{2})-u_{\varepsilon}u_{\varepsilon,h}^{s-1}(r,t_{1}))r^{d-1}\ud r.
		\end{align*}
		Hence, since $u_{\varepsilon,h}^{s-1}-\varphi=h^{s-1}$, the equation \eqref{eq:aaa} becomes
		\[
		\frac{1}{s}\int_{0}^{R}(u_{\varepsilon,h}^{s}(r,t_{1})-u_{\varepsilon,h}^{s}(r,t_{2}))r^{d-1}\ud r -  h^{s-1}\int_{0}^{R}(u_{\varepsilon}(r,t_{2})-u_{\varepsilon}(r,t_{1}))r^{d-1}\ud r =A_{\varepsilon,h}.
		\]
		Since we eliminated the time derivative, we may let $\varepsilon\rightarrow0$
		to obtain
		\begin{equation}
		\label{eq:lalala2}
		\frac{1}{s}\int_{0}\left(u_{h}^{s}(r,t_{1})-u_{h}^{s}(r,t_{2})\right)r^{d-1}\ud r  - h^{s-1}\int_{0}^{R}(u(r,t_{2})-u(r,t_{1}))r^{d-1}\ud r =A_{h}.
		\end{equation}
		Next, we rewrite $A_{h}$ as follows
		\begin{align*}
		A_{h}=\frac{p-1}{q-1}\int_{t_{1}}^{t_{2}}\int_{0}^{R}\left|u^{\prime}\right|^{q-2}u^{\prime}\partial_{r}(u_{h}^{s-1})\ud z & =\frac{p-1}{q-1}\int_{t_{1}}^{t_{2}}\int_{0}^{R}\left|u_{h}^{\prime}\right|^{q-2}u_{h}^{\prime}\partial_{r}(u_{h}^{s-1})\ud z\\
		& =\frac{(p-1)(s-1)}{q-1}\int_{t_{1}}^{t_{2}}\int_{0}^{R}\left|u_{h}^{\prime}\right|^{q}u_{h}^{s-2}\ud z,
		\end{align*}
		where by Sobolev's inequality in Theorem \ref{thm:sobolev} forcing vanishing boundary values
		\begin{align*}
		\int_{t_{1}}^{t_{2}}\int_{0}^{R}\left|u_{h}^{\prime}\right|^{q}u_{h}^{s-2}\ud z & =\int_{t_{1}}^{t_{2}}\int_{0}^{R}\left|u^{\prime}u_{h}^{\frac{s-2}{q}}\right|^{q}\ud z\\
		& =\left(\frac{q}{s+q-2}\right)^{q}\int_{t_{1}}^{t_{2}}\int_{0}^{R}\left|\partial_{r}(u_{h}^{\frac{s+q-2}{q}}(r,t)-u_{h}^{\frac{s+q-2}{q}}(R,t))\right|^{q}\ud z\\
		& \geq C_{1}\left(\frac{q}{s+q-2}\right)^{q}\int_{t_{1}}^{t_{2}}\left(\int_{0}^{R}\left|u_{h}^{\frac{s+q-2}{q}}(r,t)-u_{h}^{\frac{s+q-2}{q}}(R,t)\right|^{\frac{dq}{d-q}}\ud z\right)^{\frac{d-q}{d}}.
		\end{align*}
		Here $C_{1}$ is the constant in Sobolev's inequality. Since $s=\frac{d(2-q)}{q}$, we have by the last two displays
		\begin{align*}
		\liminf_{h\rightarrow0}A_{h} & \geq C_{1}\frac{(p-1)(s-1)}{q-1}\left(\frac{q}{s+q-2}\right)^{q}\int_{t_{1}}^{t_{2}}\left(\int_{0}^{R}\left|u^{\frac{s+q-2}{q}}\right|^{\frac{dq}{d-q}}\ud z\right)^{\frac{d-q}{d}}\\
		& =:C_{2}\int_{t_{1}}^{t_{2}}\left(\int_{0}^{R}\left|u\right|^{s}\ud z\right)^{\frac{d-q}{d}}.
		\end{align*}
		Consequently, letting $h\rightarrow0$ in \eqref{eq:lalala2}, we obtain
		\[
		\frac{1}{s}\int_{0}^{R}(u^{s}(r,t_{1})-u^{s}(r,t_{2}))r^{d-1}\ud r\geq C_{2}\int_{t_{1}}^{t_{2}}\left(\int_{0}^{R}\left|u\right|^{s}\ud z\right)^{\frac{d-q}{d}}.
		\]
		Denoting $v(t):=\left\Vert u(\cdot,t)\right\Vert _{L^{s}(r^{d-1},(0,R))}$
		and multiplying the inequality by $-s$, we have
		\begin{equation}
		v^s(t_{2})-v^s(t_{1})\leq-C_{2}s\int_{t_{1}}^{t_{2}}v(t)^{s\frac{d-q}{d}}\ud t.\label{eq:lalala3}
		\end{equation}
		Observe that this implies in particular that $v$ is decreasing. Next,
		we derive a distributional inequality which implies that $v$ must
		in fact vanish for large times. For this end, let $\kappa:=q-2+s$
		and observe that for any $0<a<b$ we have 
		\begin{align*}
		a^{2-q}-b^{2-q}=-\frac{1}{b^{\kappa}}\int_{a}^{b}(2-q)b^{\kappa}t^{1-q}\ud t & \leq-\frac{1}{b^{\kappa}}\int_{a}^{b}(2-q)t^{\kappa}t^{1-q}\ud t\\
		& =-\frac{1}{b^{\kappa}}\int_{a}^{b}(2-q)t^{s-1}\ud t\\
		& =\frac{2-q}{s}\frac{1}{b^{\kappa}}(a^{s}-b^{s}).
		\end{align*}
		Let then $\varphi\in C_{0}^{\infty}(0,T)$ be non-negative. Next,
		we apply the integration by parts formula for difference quotients and
		the fact that $v$ is decreasing together with the above elementary
		inequality. This way, we obtain by dominated convergence theorem
		\begin{align*}
		 -\int_{0}^{T}v^{2-q}(t)\varphi^{\prime}(t)\ud t
		&=-\lim_{\delta\rightarrow0}\int_{0}^{T}v^{2-q}(t)\frac{\varphi(t-\delta)-\varphi(t)}{-\delta}\ud t\\
		&  =\lim_{\delta\rightarrow0}\int_{0}^{T}\varphi(t)\frac{v^{2-q}(t+\delta)-v^{2-q}(t)}{\delta}\ud t\\
		&  \leq\lim_{\delta\rightarrow0}\int_{0}^{T}\varphi(t)\frac{1}{v^{\kappa}(t+\delta)}\frac{2-q}{s}\frac{v^{s}(t+\delta)-v^{s}(t)}{\delta}\ud t
		\end{align*}
		Here we can use the estimate \eqref{eq:lalala3}
		\begin{align*}
		(2-q)&\lim_{\delta\rightarrow0}\int_{0}^{T}\varphi(t)\frac{1}{v^{\kappa}(t+\delta)}\frac{1}{s}\frac{v^{s}(t+\delta)-v^{s}(t)}{\delta}\ud t\\&\leq-(2-q)C_{2}\lim_{\delta\rightarrow0}\int_{0}^{T}\varphi(t)\frac{1}{v^{\kappa}(t+\delta)}\frac{1}{\delta}\int_{t}^{t+\delta}v(l)^{s\frac{d-q}{d}}\ud l\\
		&  =-(2-q)C_{2}\int_{0}^{T}\varphi(t)\frac{v(t)^{s\frac{d-q}{d}}}{v^{\kappa}(t)}\ud t\\
		&  =-(2-q)C_{2}\int_{0}^{T}\varphi(t)\ud t,
		\end{align*}
		where the last two identities follow from continuity and the computation
		\[
		s\frac{d-q}{d}-\kappa=s\frac{d-q}{d}+2-q-s=\frac{(2-q)(d(d-q)+dq-d^{2})}{dq}=0.
		\]
		Hence we have established the distributional inequality
		\[
		\int_{0}^{T}-v^{2-q}(t)\varphi^{\prime}(t)+(2-q)C_{2}\varphi(t)\ud t\leq0\quad\text{for all non-negative }\varphi\in C_{0}^{\infty}(0,T).
		\]
		Since $v$ is continuous up to the boundary, this yields
		\[
		v^{2-q}(t)-v^{2-q}(0)+(2-q)C_{2}t\le0\quad\text{for all }t\in[0,T],
		\]
		which is, recalling $v(t)=\left\Vert u(\cdot,t)\right\Vert _{L^{s}(r^{d-1},(0,R))}$, equivalent with
		\begin{align*}
		\norm{u(\cdot,t)}_{L^s(r^{d-1},(0,R))}
		&\leq\norm{u_0(\cdot)}_{L^s(r^{d-1},(0,R))}\left(1-(2-q)C_2\norm{u_0(\cdot)}_{L^s(r^{d-1},(0,R))}^{q-2}t\right)^{\frac{1}{2-q}}.
		\end{align*}
		Thus as long as the original $T>0$ is large enough, $u$ vanishes for time $T^*$ satisfying
		\begin{equation*}
		0<T^*\leq C\norm{u_0(\cdot)}_{L^s(r^{d-1},(0,R))}^{2-q} \text{ for }C=((2-q)C_2)^{-1}.\qedhere
		\end{equation*}
	\end{proof}
	Next, we expand this local result to a global result.
	\begin{proposition}
		\label{prop:finiteextspace}
		Assume $q$ does not satisfy the range condition \eqref{eq:range}. Let $u$ be a viscosity solution of
		\begin{equation}
		\label{eq:dirichletspace}
		\begin{cases}
		\partial_t u=\abs{\nabla u}^{q-p}\div\left(\abs{\nabla u}^{p-2}\nabla u\right) & \text{ in }\Rn\times\R^+\\
		u(\cdot,0)=u_0(\cdot)\geq0&\text{ where radial }
		u_0\in C_0(B_R) \text{ for some }R>0.
		\end{cases}
		\end{equation}
		There exists a finite time $T^*:=T^*(n,p,q,u_0)$, such that
		\begin{equation*}
		u(\cdot,0)\equiv 0 \quad \text{ for all }t\geq T^*
		\end{equation*}
		and 
		\begin{equation}
		\label{eq:finiteextspace1}
		0<T^*\leq
		C\norm{u_0}_{L^s(r^{d-1},(0,R))}^{2-q}
		\end{equation}
		where $C:=C(n,p,q)$ and $s=\frac{d(2-q)}{q}$.
	\end{proposition}
	\begin{proof}
		Let $u_i$ be the radial viscosity solution to the bounded problem \eqref{eq:dirichlet} for $R=i\in\N$. Now by Proposition \ref{prop:finiteextball} there exists a finite time $T_i^*$ satisfying
		\begin{equation*}
			0<T_i^*\leq C\norm{u_0(\cdot)}_{L^s(r^{d-1},(0,i))}^{2-q},
		\end{equation*}
		such that $u_i(\cdot,t)\equiv0$ for $t\geq T_i^*$.
		By the comparison principle \ref{thm:comp} we have $u_{i+1}\geq u_i$ in $B_{i}\times(0,i)$ which implies that $T_{i+1}^*\geq T_{i}^*$ and because we assumed that $u_0$ has compact support this sequence of extinction times has a limit $T_{\lceil R \rceil}^*$.
		
		Using the Hölder estimates proven in \cite{Imbert2019}, we have that each $u_i$ is Hölder continuous in both variables and the Hölder constant only depends on $n,p,q$ and $\norm{u_i}_{L^\infty (B_{i}\times(0,i))}$. By the comparison principle these $L^\infty$-norms are bounded from above by $\norm{u_0}_{L^\infty(\Rn\times\R^+)}$ and thus the sequence $(u_i)_{i=1}^\infty$ is uniformly equicontinuous. By construction, $u_i\to u$ converges pointwise as $i\to\infty$ passing to a subsequence if necessary and because of the equicontinuity, the Arzelà-Ascoli theorem ensures that the convergence is uniform.

		For any compact subset $A\subset\Rn\times\R^+$, $u_i$ is a viscosity solution to \eqref{eq:rgnppar} in $A$ for $i$ large enough and thus $u$ is also a viscosity solution in this set by stability result proven by Ohnuma and Sato \cite[Theorem 6.1, Proposition 6.2]{Ohnuma1997}. Because $A$ is arbitrary, $u$ is a viscosity solution in the entire space and by construction, it has the correct initial value. This solution is unique as proven by \cite[Corollary 4.10]{Ohnuma1997} and this proves that $u$ vanishes after finite time $T_{\lceil R \rceil}^*$ satisfying \eqref{eq:finiteextspace1}.
	\end{proof}
	Now we have the tools needed to show that intrinsic Harnack's inequality does not hold for $q$ not satisfying the range condition \eqref{eq:range}. Let $u$ be a viscosity solution to \eqref{eq:dirichletspace} and $T^*$ the finite extinction time given by Proposition \ref{prop:finiteextspace}. Choose $(x_0,t_0)\in\R^n\times(0,T^*)$ close enough to satisfy
	\begin{equation*}
	\label{eq:close}
	T^*-t_0<\frac{t_0}{\sigma^q},
	\end{equation*}
	and choose $r>0$ to satisfy
	\begin{equation*}
	cu(x_0,t_0)^{2-q}r^q=T^*-t_0
	\end{equation*}
	where $c$ and $\sigma$ are the constants given by Harnack's inequality.
	By these choices
	\begin{equation*}
	t_0-cu(x_0,t_0)^{2-q}(\sigma r)^q=t_0-\sigma^q\left(T^*-t_0\right)>0
	\end{equation*}
	and therefore $(x_0,t_0)+Q_{\sigma r}(\theta)\subset \R^n\times\R^+$ and thus we can use the Harnack's inequality to obtain
	\begin{equation*}
	0<u(x_0,t_0)\leq\gamma\inf_{B_r(x_0)}u(\cdot, T^*)=0,
	\end{equation*}
	which is a contradiction. 
	
	There are some known Harnack-type results with additional assumptions for the $p$-parabolic equation in the subcritical range, see for example \cite[Proposition 1.1]{Dibenedetto2009}.
	%\bibliographystyle{alphaabbr}
	%\bibliography{../../../bibtex/Harnack}
	\newpage
	\begin{center}
		{\textsc{Declarations}}
	\end{center}
	\textbf{Ethical Approval.} 
	Not applicable.
	
	\textbf{Competing interests.}
	There are no competing interests.
	
	\textbf{Authors' contributions.}
	Both authors have the same level of contribution.
	
	\textbf{Funding.} 
	Jarkko Siltakoski was supported by the Magnus Ehrnrooth Foundation.
	
	\textbf{Availability of data and materials.}
	Not applicable.

\end{document}

\begin{lemma}
	\label{lem:supersolution}
	Suppose that $1<q<2$, $p>1$ and let $R>0$. Then there exists a
	positive constant $\lambda=\lambda(n,p,q)$ such that the function
	\begin{equation}
		v(x,t):=\lambda t^{\frac{1}{2-q}}\left(\frac{1}{R^{\frac{1}{1-q}}(R^{\frac{q}{q-1}}-\left|x\right|^{\frac{q}{q-1}})}\right)^{\frac{q}{2-q}}\label{eq:explicit formula}
	\end{equation}
	is a viscosity supersolution to \eqref{eq:rgnppar} in
	$B_{R}(0)\times(0,\infty)$.
\end{lemma}
This supremum estimate replaces the weak Harnack inequality used in the divergence form case.
\begin{lemma}
	Assume $u$ is a viscosity solution to \eqref{eq:rgnppar} in $Q_{4R}^-(1)$.
\end{lemma}
\begin{lemma}
	\label{le:backintime}
	Assume $u$ is a viscosity solution to \eqref{eq:rgnppar} in $\Om_T$ and let $(\hat{x},\hat{t})\in\Om_T$, $R>0$ and $\tau>0$ such that $B_{R}(\hat{x})\times[\hat{t}-\tau,\hat{t}]\subset\Om_T$. Then there exists a constant $\lambda_1:=\lambda_1(n,p,q)$ such that
	\begin{equation*}
		\sup_{B_{R}(\hat{x})}u(\cdot,t)\leq \sup_{B_{2R}(\hat{x})}u(\cdot,\hat{t})+\lambda_1\tau^{\frac{1}{2-q}}R^{-\frac{q}{2-q}}.
	\end{equation*}
	for all $t\in[\hat{t}-\tau,\hat{t}]$.
\end{lemma}
\begin{proof}
	Let
	\begin{equation*}
		v(x,t)=\lambda (\hat{t}-t)^{\frac{1}{2-q}}\left(\frac{1}{(2R)^{\frac{1}{1-q}}((2R)^{\frac{q}{q-1}}-\left|x-\hat{x}\right|^{\frac{q}{q-1}})}\right)^{\frac{q}{2-q}}+\sup_{B_{2R}(\hat{x})}u(\cdot,\hat{t}).
	\end{equation*}
	This is a viscosity supersolution to \eqref{eq:rgnppar} in $B_{2R}(\hat{x})\times(-\infty,\hat{t}]$ by Lemma \ref{lem:supersolution}. We also have
	\begin{equation*}
		\begin{cases}
			u(x,\hat{t})\leq\sup_{B_{2R}(\hat{x})}u(\cdot,\hat{t})=v(x,\hat{t}) & \text{ on } B_{2R}(\hat{x}) \\
			\limsup_{\substack{(x,t)\to(y,s)\\ x\in B_{2R}(\hat{x})}}v(x,t)=\infty & \text{ for all }(y,s)\in\partial B_{2R}(\hat{x})\times(-\infty,\hat{t}).
		\end{cases}
	\end{equation*}
	Thus necessarily $u\leq v$ on $\partial_{p}(B_{2R}(\hat{x})\times[\hat{t}-\tau,\hat{t}])$ and we can use comparison principle Theorem \ref{thm:comp} to deduce that
	\begin{equation*}
		u(x,t)\leq \lambda (\hat{t}-t)^{\frac{1}{2-q}}\left(\frac{1}{(2R)^{\frac{1}{1-q}}((2R)^{\frac{q}{q-1}}-\left|x-\hat{x}\right|^{\frac{q}{q-1}})}\right)^{\frac{q}{2-q}}+\sup_{B_{2R}(\hat{x})}u(\cdot,\hat{t})
	\end{equation*}
	for all $x\in B_{2R}(\hat{x})$ and $t\in[\hat{t}-\tau,\hat{t}]$.
	Taking supremum over $B_R(\hat{x})$ from both sides, we have
	\begin{align*}
		\sup_{B_{R}(\hat{x})}u(\cdot,t)&\leq \sup_{B_{2R}(\hat{x})}u(\cdot,\hat{t})+\lambda \left(2^{\frac{1}{1-q}}\left(2^{\frac{q}{q-1}}-1\right)\right)^{-\frac{q}{2-q}}\tau^{\frac{1}{2-q}}R^{-\frac{q}{2-q}}.
		\\&= \sup_{B_{2R}(\hat{x})}u(\cdot,\hat{t})+\lambda_1\tau^{\frac{1}{2-q}}R^{-\frac{q}{2-q}}.
	\end{align*}
	for all $t\in[\hat{t}-\tau,\hat{t}]$, which is what we wanted.
\end{proof}
We have all we need to prove the intrinsic Harnack's inequality.